\documentclass[11pt,letterpaper]{article}
\usepackage[utf8]{inputenc}
\usepackage{amsmath,amsthm,amssymb,mathtools}
\usepackage[margin=1.05in]{geometry}
\usepackage{microtype,etoolbox}
\usepackage{xcolor}
\definecolor{paperlink}{RGB}{0,0,114}
\usepackage[shortlabels]{enumitem}
\setlist[enumerate]{leftmargin=2.25em,itemsep=.25em,topsep=.35em}
\usepackage{fancyhdr}
\usepackage{float}
\usepackage{tikz}
\usetikzlibrary{arrows.meta,patterns}
\numberwithin{equation}{section}
\mathtoolsset{showonlyrefs}
\allowdisplaybreaks[1]
\newtheorem{theorem}{Theorem}[section]
\newtheorem{proposition}[theorem]{Proposition}
\newtheorem{lemma}[theorem]{Lemma}
\newtheorem{corollary}[theorem]{Corollary}
\theoremstyle{definition}
\newtheorem{definition}[theorem]{Definition}
\theoremstyle{remark}
\newtheorem{remark}[theorem]{Remark}
\newcommand{\R}{\mathbb R}
\newcommand{\N}{\mathbb N}
\newcommand{\Dp}{\Delta_p}
\newcommand{\capc}{\operatorname{cap}}
\newcommand{\capce}{\operatorname{cap}_{\mathrm e}}
\newcommand{\upperH}{\overline H}
\newcommand{\lowerH}{\underline H}
\newcommand{\essliminf}{\operatorname*{ess\,lim\,inf}}
\newcommand{\esssup}{\operatorname*{ess\,sup}}
\newcommand{\essinf}{\operatorname*{ess\,inf}}
\newcommand{\dd}{\,d}
\newcommand{\spacetime}{\R^{n+1}}
\newcommand{\ind}{\mathbf 1}
\newcommand{\supp}{\operatorname{supp}}
\renewcommand{\div}{\operatorname{div}}
\newcommand{\redpot}{\widehat u}
\newcommand{\pc}{p_{\mathrm c}}
\def\vint_#1{\mathchoice%
          {\mathop{\kern 0.2em\vrule width 0.6em height 0.69678ex depth -0.58065ex
                  \kern -0.8em \intop}\nolimits_{\kern -0.4em#1}}%
          {\mathop{\kern 0.1em\vrule width 0.5em height 0.69678ex depth -0.60387ex
                  \kern -0.6em \intop}\nolimits_{#1}}%
          {\mathop{\kern 0.1em\vrule width 0.5em height 0.69678ex
              depth -0.60387ex
                  \kern -0.6em \intop}\nolimits_{#1}}%
          {\mathop{\kern 0.1em\vrule width 0.5em height 0.69678ex depth -0.60387ex
                  \kern -0.6em \intop}\nolimits_{#1}}}
\usepackage[colorlinks=true,linktoc=all,linkcolor=paperlink,citecolor=paperlink,urlcolor=paperlink,bookmarksnumbered=true,pdfstartview=Fit]{hyperref}
\hypersetup{pdftitle={Wiener's criterion and resolutivity for the p-parabolic equation},pdfauthor={Benny Avelin and Mikko Parviainen},pdfsubject={Boundary regularity and Perron resolutivity for the p-parabolic equation},pdfkeywords={Wiener criterion, boundary regularity, nonlinear parabolic capacity, Perron method, obstacle problem}}

\begin{document}
\title{Wiener's criterion and resolutivity\\
for the~$p$-parabolic equation}
\author{Benny Avelin\thanks{Department of Mathematics, Uppsala University, Box~480, 751~06 Uppsala, Sweden.}
\and Mikko Parviainen\thanks{Department of Mathematics and Statistics, University of Jyv\"askyl\"a, P.O.~Box~35, FI-40014 University of Jyv\"askyl\"a, Finland.}}
\date{}
\maketitle
\begin{abstract}
We characterize boundary regularity for~$u_t=\Dp u$ in arbitrary bounded open subsets of~$\R^{n+1}$ for~$n\ge2$ and~$p\in(2n/(n+1),2)\cup(2,\infty)$. We prove equivalent Wiener criteria in terms of capacitary potentials on prescribed cylinders and in terms of intrinsic capacities. Both require divergence at every positive obstacle height. Independently, we prove Perron resolutivity for continuous data on the full topological boundary of every nonempty bounded open set for all~$n\ge1$ and~$p>1$. As an application of the Wiener criterion, we prove a Petrovski\u{\i}-type criterion in the singular range, which gives regularity of the origin for~$|x|<C(-t)^{1/p}$ for every~$C>0$. For~$2<p<n$, examples show that divergence at every sufficiently large height does not imply regularity.
\end{abstract}
\setcounter{tocdepth}{1}
\begingroup\small
\makeatletter
\patchcmd{\l@section}{1.0em}{.45em}{}{}
\makeatother
\tableofcontents
\endgroup
\clearpage

\section{Introduction}
\label{sec:intro}

We study boundary regularity for the~$p$-parabolic equation
\begin{equation*}
 u_t=\Dp u,\qquad \Dp u=\div(|\nabla u|^{p-2}\nabla u)\,.
\end{equation*}
Let~$\Theta\subset\R^{n+1}$ be bounded and open. A boundary point is regular if the Perron solution attains every continuous boundary datum there. We characterize regularity without geometric assumptions on~$\partial\Theta$, for
\begin{equation*}
 n\ge2,\qquad \pc:=\frac{2n}{n+1},\qquad p\in(\pc,2)\cup(2,\infty)\,.
\end{equation*}
Our first criterion uses values of capacitary potentials on prescribed cylinders. An equivalent intrinsic capacity criterion supplies the estimates used in the proof and the geometric examples. We also prove an independent resolutivity theorem for every~$n\ge1$ and~$p>1$. Resolutivity completes the Perron construction for continuous boundary data by proving that the upper and lower Perron solutions coincide. Together with the Wiener criteria this gives a rather complete story of the Perron method.

\paragraph{The potential criterion.}
For~$z_0=(x_0,t_0)$ and~$r,s>0$, write
\begin{equation}\label{eq:parabolic-cylinders}
 Q^-_{r,s}(z_0)=B_r(x_0)\times(t_0-s,t_0),\qquad Q_{r,s}(z_0)=B_r(x_0)\times(t_0-s,t_0+s)\,.
\end{equation}
Balls without a displayed center are centered at the spatial origin. We call~$(K,Q)$ a parabolic condenser if~$Q=D\times(b,c)$, where~$D\subset\R^n$ is a smooth bounded domain, and~$K\Subset Q$ is compact. For~$a>0$, let~$\redpot^a_{K,Q}$ be the balayage of~$a\ind_K$, with zero initial and lateral values; see Definition~\ref{def:reduction}.

Translate the boundary point to the origin and set~$E=\R^{n+1}\setminus\Theta$. Put
\begin{equation}\label{eq:potential-scales}
 s_j=2^{-j},\qquad R_j=s_j^{\frac{1}{2p}},\qquad j\ge1\,.
\end{equation}
Define
\begin{equation}\label{eq:potential-test-sets}
 K_j=E\cap\overline{Q^-_{R_j,2s_j}(0,-2s_j)},\qquad Q_j=Q_{2R_j,3s_j}(0,-2s_j)\,.
\end{equation}
Since~$K_j$ lies strictly before time zero,~$\redpot^a_{K_j,Q_j}$ is~$p$-parabolic near the origin.

Our first main result is a Wiener criterion in the sense of Landis~\cite{Landis}, stated in terms of capacitary potentials. We present this non-intrinsic form first because the cylinders~$Q_j$ are prescribed independently of the obstacle height~$a$ and the domain~$\Theta$, which enters only through~$K_j$. The equivalent intrinsic capacity criterion in Theorem~\ref{cor:intrinsic-wiener-intro} is closer in form to the classical Wiener criterion.

\begin{theorem}[Wiener criterion in terms of potentials]\label{thm:potential-wiener}
Let~$n\ge2$ and~$p\in(\pc,2)\cup(2,\infty)$. For every bounded open~$\Theta\subset\R^{n+1}$ with~$0\in\partial\Theta$, 
\begin{equation}\label{eq:potential-wiener}
 0\text{ is regular}
 \quad\Longleftrightarrow\quad
 \sum_{j=1}^{\infty}\frac{\redpot^a_{K_j,Q_j}(0,0)}a=\infty
 \quad\text{for every }a>0\,.
\end{equation}
\end{theorem}

This answers the well known open question of an explicit Wiener criterion on arbitrary bounded open space--time sets, raised in~\cite[Section~1]{BBGP} and reiterated in~\cite[Section~1]{BB26}.

For~$p\ne2$, normalizing the obstacle height requires a rescaling of time and hence changes the test geometry. The requirement to test every height~$a>0$ parallels the use of a family of barriers in the characterization of boundary regularity; see~\cite[Theorem~3.3]{BBGP}. A single traditional barrier does not suffice in general, as shown in~\cite[Proposition~1.2]{BBG} for the singular range and~\cite[Corollary~1.3]{AP26} for the degenerate range. Proposition~\ref{prop:g-lipschitz-critical-irregular} shows why divergence at every sufficiently large height does not suffice.

\paragraph{Intrinsic capacity criterion.}
Write~$\capc_a(K,Q)$ for the nonlinear parabolic capacity of~\cite{KKKP,MS} with the solution bounded by~$a$. The reference problem has zero initial and lateral data and homogeneous continuation after the top of~$Q$; Definition~\ref{def:parabolic-capacity} gives the precise conventions.

For~$L>4$, radius~$r>0$ and increment~$q>0$, define
\begin{equation}\label{eq:intro-intrinsic-cylinders}
 Q_r^{q,-}=B_r\times(-4q^{2-p}r^p,-2q^{2-p}r^p),\qquad Q_r^q=B_{8Lr}\times(-5q^{2-p}r^p,q^{2-p}r^p)\,.
\end{equation}
In the~$j$th summand below, each~$q>0$ determines the unique~$r>0$ satisfying~$q^{2-p}r^p=s_j$. Thus~$q$ changes the spatial radius while the duration remains fixed. The quantity~$qr^n$ is the intrinsic Riesz mass scale, also given by Barenblatt scaling~\cite[Section~2]{BBG}; see Subsection~\ref{subsec:intrinsic-scaling}.

For~$p>2$, the choice of~$R_j$ makes~$Q_j$ eventually contain all corresponding intrinsic reference cylinders at each fixed height. For~$\pc<p<2$, increments whose reference cylinders do not fit inside~$Q_j$ have a summable upper bound; see Subsection~\ref{subsec:potential-proof}.

\begin{theorem}[Intrinsic Wiener criterion]\label{cor:intrinsic-wiener-intro}
Under the hypotheses of Theorem~\ref{thm:potential-wiener}, there are constants~$L=L(n,p)>4$ and~$C_*=C_*(n,p)>0$ such that the origin is regular if and only if, for every~$a>0$,
\begin{equation*}
 \sum_{j=1}^{\infty}\sup\left(\left\{q\in(0,a]:
 \capc_a\bigl(E\cap\overline{Q_r^{q,-}},Q_r^q\bigr)>C_*qr^n\right\}\cup\{0\}\right)=\infty\,.
\end{equation*}
\end{theorem}

We write~$\omega_a(j)$ for this supremum, which is zero when no increment satisfies the capacity inequality.

The elliptic Wiener criterion and known parabolic potential estimates~\cite[Theorem~1.2]{KuMi14} naturally lead one to ask whether regularity admits a characterization by a non-intrinsic capacity sum. Remark~\ref{rem:g-nonintrinsic-capacity} shows that, for~$2<p<n$, the positive-power sums obtained by freezing~$q=1$ cannot provide such a characterization.

Divergence must hold at arbitrarily small heights. For~$2<p<n$ and~$M>0$, let
\begin{equation*}
 \Theta_M=(B_1\times(-1,1))\cap\{(x,t):x_n>-M|t|^{\frac{1}{p}}\}\,.
\end{equation*}
The origin is irregular, although~$\sum_j\omega_a(j)=\infty$ at every sufficiently large height. At every sufficiently small height, the test sets are empty on a tail and the series converges; see Proposition~\ref{prop:g-lipschitz-critical-irregular}.

The time-slice identity in Proposition~\ref{lem:capacity-time-slice} gives an exact reformulation in terms of capacity differences; see Remark~\ref{cor:capacity-wiener}.

\paragraph{Resolutivity.}
Independently of the Wiener criteria, we prove resolutivity in the larger range~$n\ge1$,~$p>1$. For~$f\in C(\partial\Theta)$, let~$\upperH_\Theta f$ and~$\lowerH_\Theta f$ be the upper and lower Perron solutions of Definition~\ref{def:perron}. The datum is resolutive when these coincide.

\begin{theorem}\label{thm:resolutivity}
Let $n\ge1$, $1<p<\infty$, and let $\Theta\subset\R^{n+1}$ be nonempty, bounded and open. For every $f\in C(\partial\Theta)$,
\begin{equation*}
 \lowerH_\Theta f=\upperH_\Theta f\quad\text{everywhere in }\Theta.
\end{equation*}
The boundary in these definitions is the full topological boundary, including its top portions.
\end{theorem}

Theorem~\ref{thm:resolutivity} solves the resolutivity problem posed in~\cite[Section~8, problem~7]{BBGP} for arbitrary bounded open space--time sets. Its proof is in Section~\ref{sec:r-resolutivity}.

\subsection{Geometric consequences}
As an application of the intrinsic Wiener criterion, we prove a Petrovski\u{\i}-type criterion. In particular, for~$n\ge2$ and~$\pc<p<2$, the origin is regular for the critical domain~$|x|<C(-t)^{\frac{1}{p}}$ for every~$C>0$. Proposition~\ref{prop:g-singular-power} settles the endpoint left open after~\cite[Theorem~1.1]{BBG}. For~$2<p<n$, the same critical domain is irregular. Table~\ref{tab:g-power-classification} compares the classifications; the relation to Petrovski\u{\i}'s heat-equation criterion~\cite{Petrovskii35} is discussed in Subsection~\ref{sec:g-power-section}.

For time-dependent~$\operatorname{Lip}(1,\gamma)$ graphs, regularity holds for~$\gamma>1/p$ when~$2<p<n$, and for~$\gamma\ge1/p$ when~$\pc<p<2$. The domain~$\Theta_M$ in Proposition~\ref{prop:g-lipschitz-critical-irregular} shows that regularity can fail at the degenerate endpoint; see Subsection~\ref{sec:g-parabolic-lipschitz}. Related boundary estimates for intrinsic parabolic non-tangentially accessible (NTA) domains under exterior conditions were obtained in~\cite{A16}.

For stationary cylinders, we compare the intrinsic series with the elliptic Wiener criterion without a thickness assumption on the complement. The soda-can examples recover the degenerate classification of~\cite[Theorem~1.1]{AP26} through upper and lower capacity bounds. The degenerate calculations use~$2<p<n$, and the singular calculations use~$n\ge2$,~$\pc<p<2$. Flat boundaries, exterior cones and logarithmically shrinking tubes are included as brief consequences.

\subsection{Comparison with other Wiener criteria}
For the elliptic~$p$-Laplacian, Maz'ya~\cite{Maz70} proved sufficiency of the Wiener condition in 1970, and Kilpel\"ainen and Mal\'y~\cite{KM} proved necessity in 1994, completing the characterization for general~$p>1$. For~$1<p\le n$, the criterion uses relative elliptic capacity ratios raised to~$1/(p-1)$; see~\cite[Theorem~1.1]{KM}. For the heat equation,~\cite{Landis} gives a series of caloric potentials, while~\cite{Lanconelli,EG} develop capacity criteria using portions of the complement between backward heat-kernel levels. Linear homogeneity permits normalization to unit height.

The Perron method for the~$p$-parabolic equation was developed by Kilpel\"ainen and Lindqvist~\cite{KL96}. The characterization of regularity by a family of barriers is given in~\cite[Theorem~3.3]{BBGP}. For cylindrical domains, I.~I.~Skrypnik established elliptic-capacity Wiener criteria for classes of degenerate and supercritical singular equations under structural and dimension restrictions~\cite{Sk00,Sk04s,Sk04d,Sk04n}. Quantitative boundary estimates and sufficient elliptic-capacity conditions for cylinders with continuous Sobolev boundary data are given in~\cite[Theorem~1.1 and Corollary~1.1]{GLL} for~$\pc<p<2$ and in~\cite[Theorem~1.1 and Corollary~1.1]{GL20} for~$p>2$. Our criteria characterize Perron regularity on arbitrary bounded open space--time sets. We use nonlinear measure capacity; on compact sets in a common reference problem, the variational capacity of~\cite{AKP} is comparable to it for~$p\ge2$.

For linear hypoelliptic diffusion operators,~\cite[Theorem~1.1]{LTU} gives Gaussian-capacity tests under two-sided Gaussian bounds and additional metric hypotheses. The criterion in~\cite[Theorem~1.3]{TU} uses balayages on portions of the complement between fundamental-solution levels.

We exclude~$p=2$ from the Wiener criteria. In the intrinsic test, varying~$q$ then leaves the test sets unchanged, and the resulting series can converge at a regular boundary point of the heat equation. The prescribed condensers remain well-defined, but our estimates do not give a uniform passage to~$p=2$; see Remark~\ref{rem:linear-limit}.

\subsection{Proofs of the main theorems}\label{subsec:intro-proof-overview}
\paragraph{Sufficiency: }
We use divergence of the Wiener series to construct approximate barriers. A mass-to-positivity estimate for local capacitary potentials gives a positive increment from each successful capacity test (Lemma~\ref{lem:reception}). We select nested cylinders while retaining a fixed fraction of these increments (Lemma~\ref{lem:selection}). Successive comparisons then make the potential of the selected union arbitrarily close to its prescribed height near the origin (Proposition~\ref{prop:saturation}). Subtracting this potential from its height gives the approximate barriers used in Subsubsection~\ref{subsec:joint-sufficiency}.

For the mass-to-positivity estimate, let~$a,q,r>0$, put~$s=q^{2-p}r^p$, and let~$K\subset\overline{Q_r^{q,-}}$ be compact. For the constants~$L,C_*$ in Theorem~\ref{cor:intrinsic-wiener-intro}, there are constants~$c\in(0,1]$ and~$\eta\in(0,1/4)$, depending only on~$n,p$, such that
\begin{equation*}
 \capc_a(K,Q_r^q)>C_*qr^n
 \quad\Longrightarrow\quad
 \redpot^a_{K,Q_r^q}\ge cq
 \quad\text{on }B_r\times[-\eta s,\eta s]\,.
\end{equation*}
The constants are independent of~$a$. Writing~$V=\redpot^a_{K,Q_r^q}$, we have~$\capc_a(K,Q_r^q)=\mu_V(K)$, where~$\mu_V$ is the Riesz measure of~$V$.

As in the proof of the Carleson estimate by the first author together with Gianazza and Salsa~\cite[Section~2.4.2]{AGS}, we argue by contradiction. For~$p>2$, weak Harnack and persistence show that a small later value bounds the earlier spatial averages. For~$\pc<p<2$, comparison with homogeneous zero-boundary solutions gives the corresponding bound. Local boundedness then controls~$V$ in an annulus outside the source. A Caccioppoli estimate controls the flux through this annulus, and testing the equation bounds~$\mu_V(K)$; see Lemma~\ref{lem:reception}.

For~$p>2$, this calculation parallels the boundary Riesz-measure estimates in~\cite[Theorem~5.2]{AKN} and~\cite[Lemma~7.1]{A20}. The zero extensions there, across lateral boundary portions with zero Dirichlet data, satisfy~$u_t-\Dp u=-\mu$, where~$\mu\ge0$ is supported on the boundary. Our reflected potential satisfies
\begin{equation*}
 (a-V)_t-\Dp(a-V)=-\mu_V
\end{equation*}
with~$\mu_V$ supported on~$K$. Our estimate applies to arbitrary compact~$K$ in both ranges of~$p$.

Choose an admissible~$q_j>\omega_a(j)/2$ at each index with~$\omega_a(j)>0$. Then~$\sum_jq_j=\infty$. From each finite block of scales, we select a subsequence whose successive reference cylinders lie in the preceding regions of positivity. The selected increments retain a fixed fraction of the original sum, so divergence gives finite chains with arbitrarily large total increment. Selection proceeds backwards for~$p>2$ and forwards for~$p<2$; see Lemma~\ref{lem:selection}. For~$p<2$, cylinders too large for a prescribed neighborhood carry only a summable contribution and can be discarded; see Remark~\ref{s:singular-locality}. Thus the chains can be chosen arbitrarily close to the origin.

Along this Harnack-type chain, the lower bounds accumulate on a single superparabolic majorant. Fix~$H>a$, and let~$v\ge0$ be superparabolic and satisfy~$v\ge H$ on the selected obstacles. If~$v\ge b$ on the next reference cylinder~$Q_r^q$, with~$b\le H-a$, then~$v-b\ge a$ on its obstacle. Comparison gives
\begin{equation*}
 v\ge b+\redpot^a_{K,Q_r^q}\ge b+cq
 \quad\text{on }B_r\times[-\eta s,\eta s]\,.
\end{equation*}
Starting with~$b=0$, we repeat until the accumulated bound reaches~$H-a$. Since this holds for every admissible majorant, Proposition~\ref{prop:saturation} gives, for the selected union~$F$ in a common reference cylinder~$Q^*$,
\begin{equation*}
 0\le H-\redpot^H_{F,Q^*}\le a
 \quad\text{near the origin}.
\end{equation*}

Since~$F\subset E$, both~$\redpot^H_{F,Q^*}$ and~$H-\redpot^H_{F,Q^*}$ are~$p$-parabolic in~$\Theta\cap Q^*$. Thus~$H-\redpot^H_{F,Q^*}$ supplies the local approximate barriers. Here~$a$ controls the error near the origin, while~$H$ controls the boundary datum away from it. Letting~$a\to0$ gives the required Perron competitors. This corresponds to the barrier-family characterization of regularity in~\cite[Theorem~3.3]{BBGP}.

It suffices to use~$a=1/k$, $k\in\N$: monotonicity of~$\omega_a(j)$ in~$a$ makes divergence at these heights equivalent to divergence at every positive height.

\paragraph{Necessity: } We use convergence of the Wiener series at one height to construct an irregularity barrier. Upper estimates control the contribution from the portions of the complement in successive cylinders (Theorem~\ref{thm:upper-main}). Keeping the obstacle neighborhoods fixed, we pass to a bounded superparabolic limit~$U$ that equals a fixed height~$h>0$ near the local complement in the past but satisfies~$U(z_k)<h/2$ along an interior sequence~$z_k\to0$. Perron comparison then proves irregularity (Proposition~\ref{prop:necessity}).

Fix~$R_0>0$. With~$s_j=2^{-j}$, let
\begin{equation*}
 E_{N,m}=\bigcup_{j=N}^{m}
 \left[E\cap\bigl(\overline B_{R_0}\times[-4s_j,-2s_j]\bigr)\right],
 \qquad N\le m,
\end{equation*}
be the portion of the complement in the union of these cylinders, as in~\eqref{eq:finite-tail}. Its whole-space potential, with zero values before the first obstacle time, satisfies
\begin{equation}\label{eq:intro-cumulative-upper}
 \redpot^{h}_{E_{N,m}}(0,0)
 \le C\sum_{j=N}^{m}\omega_a(j),
 \qquad
 h=\begin{cases}
 a/A,&p>2,\\
 a,&\pc<p<2.
 \end{cases}
\end{equation}
The constants~$A>1$ and~$C>0$ depend only on~$n,p$.

For~$p>2$, the height bound handles any term with~$\omega_a(j)=a$. Otherwise, choose~$\omega_a(j)<b_j<a$. By the definition of~$\omega_a(j)$, the capacity test fails for every~$b_j\le q\le a$. We decompose the complement in each cylinder into its portions in~$\overline B_{\rho_j}$ and the annulus~$\overline B_{R_0}\setminus B_{\rho_j}$, where~$\rho_j^p=b_j^{p-2}s_j$. Failure at~$b_j$ bounds the inner capacity. Failures at larger increments, together with the general cylinder-capacity estimate, give outer capacity bounds at every radius (Lemma~\ref{lem:core-remainder}).

After reducing the height, scaling and iteration give an integral growth bound that forces the potential of the union of the annular portions to vanish at the origin (Lemma~\ref{lem:remainder-waiting}). This combines the decay argument of~\cite[Section~7]{AKN} with the memory-effect idea of~\cite[Theorem~6.1]{A20}. By the difference estimate in Theorem~\ref{thm:removal}, each inner part increases the potential at the origin by at most~$Cb_j$, independently of the obstacles already present. Adding the inner parts gives~$\redpot^h_{E_{N,m}}(0,0)\le C\sum_{j=N}^{m}b_j$. Letting~$b_j\downarrow\omega_a(j)$ gives~\eqref{eq:intro-cumulative-upper}.

For~$\pc<p<2$, we regularize increments chosen just above~$\omega_a(j)$ to obtain~$b_j$ with decreasing intrinsic radii, increasing their total sum by at most a fixed factor. Comparison with homogeneous solutions controls the excess above successive levels, whose increases sum to at most~$C\sum_j b_j$. Letting the total approximation error tend to zero gives~\eqref{eq:intro-cumulative-upper}.

In both ranges, we choose continuous obstacles equal to~$h$ on fixed open neighborhoods of these portions of the complement. Keeping these choices fixed as we add further cylinders gives increasing potentials whose limit~$U$ still equals~$h$ on these neighborhoods. For a sufficiently small tail and~$p>2$, the bound~$U(0,0)\le h/4$ and its representation by a past essential lower limit give a sequence approaching the origin from the past with~$U<h/2$. For~$p<2$, small spatial averages on shrinking balls give such a sequence. These points lie in~$\Theta$, since~$U=h$ on the nearby complement in the past; see Proposition~\ref{prop:necessity}.

Choose a continuous boundary datum~$0\le g\le h$, equal to~$h$ near the origin and zero outside a sufficiently small neighborhood. Perron comparison gives
\begin{equation*}
 \lowerH_\Theta g(z_k)\le U(z_k)<h/2,
 \qquad g(0)=h.
\end{equation*}
Thus the lower Perron solution does not attain the boundary datum at the origin, proving irregularity. Comparing intrinsic and prescribed potentials then proves Theorem~\ref{thm:potential-wiener}; see Section~\ref{sec:upper} and Subsection~\ref{subsec:potential-proof}.

\paragraph{Resolutivity: } We first approximate the continuous boundary datum uniformly by restrictions of Lipschitz functions~$F$ with~$F_t-\Dp F\in L^\infty(Q)$, where~$Q$ is a smooth cylinder containing~$\overline\Theta$. Lemma~\ref{lem:r-data} supplies this approximation for every~$p>1$, including the singular range, where smoothness alone need not give bounded~$\Dp F$.

Keeping~$F$ fixed, we first construct a supersolution above~$F$. All correction problems are solved on~$Q$ with zero initial and lateral values. The obstacles force the corrections to vanish near~$\partial\Theta$; on each interior compact set, the obstacle constraints become inactive as the indices increase. For the positive corrections, we take~$H_j\ge0$ solving upper-obstacle problems with prescribed right-hand side~$h=1+\|F_t-\Dp F\|_\infty$. Their equations give
\begin{equation*}
 (F+H_j)_t-\Dp(F+H_j)=F_t-\Dp F+h-\nu_j,\qquad \nu_j\ge0.
\end{equation*}
Here~$\nu_j$ is the obstacle reaction enforcing the upper bound on~$H_j$. As~$j\to\infty$, the upper obstacles become inactive on interior compact sets, and~$H_j\uparrow H$ with
\begin{equation*}
 (F+H)_t-\Dp(F+H)=F_t-\Dp F+h\ge1
 \quad\text{in }\Theta.
\end{equation*}
Although~$H$ need not have zero boundary values,~$H\ge0$ ensures that~$F+H$ belongs to the upper Perron class.

We next add negative corrections to cancel the positive right-hand side~$F_t-\Dp F+h$ and obtain a solution in the limit. With~$H_j$ fixed, we construct~$w_{j,m}\le0$ by lower-obstacle problems with prescribed right-hand side~$-F_t+\Dp F-h$. Each~$w_{j,m}$ vanishes near~$\partial\Theta$, and
\begin{equation*}
 (F+H_j+w_{j,m})_t-\Dp(F+H_j+w_{j,m})=\mu_{j,m}-\nu_j,
\end{equation*}
where~$\mu_{j,m}\ge0$ is the lower-obstacle reaction. Both corrections vanish near the boundary, but the right-hand side has no prescribed sign. To obtain upper and lower Perron competitors, we remove the two obstacles in different orders; see Lemma~\ref{lem:r-obstacle}.

First letting~$j\to\infty$ with~$m$ fixed gives corrections~$w_m\le0$ and functions
\begin{equation*}
 U_m=F+H+w_m,\qquad (U_m)_t-\Dp U_m=\mu_m\ge0.
\end{equation*}
Since~$w_m$ vanishes near~$\partial\Theta$ and~$H\ge0$, these are upper Perron competitors. In the other order, letting~$m\to\infty$ with~$j$ fixed removes the lower obstacle and gives~$\widetilde w_j=\lim_m w_{j,m}$ with
\begin{equation*}
 V_j=F+H_j+\widetilde w_j,\qquad (V_j)_t-\Dp V_j=-\nu_j\le0.
\end{equation*}
Since~$H_j$ vanishes near~$\partial\Theta$ and~$\widetilde w_j\le0$, these are lower Perron competitors. Thus the boundary inequalities use the vanishing of~$w_m$ or~$H_j$ for a fixed index, and require no boundary trace of either limiting correction~$H$ or~$w=\lim_m w_m$. They hold for the canonical representatives on the full topological boundary, including terminal portions.

We compare~$w_{j,m}$ and~$w_m$, which have the same lower obstacle and right-hand side, but are constructed using~$H_j$ and~$H$, respectively. The energy estimate in Lemma~\ref{lem:r-obstacle} has constants independent of~$j,m$. Letting~$m\to\infty$ first and then~$j\to\infty$ gives~$\widetilde w_j\to w$ in~$L^2(Q)$; see Lemma~\ref{lem:r-two-limits}. Thus~$U_m$ and a subsequence of~$V_j$ converge almost everywhere to~$F+H+w$, and the Perron inequalities give
\begin{equation*}
 F+H+w\le\lowerH_\Theta(F|_{\partial\Theta})
 \le\upperH_\Theta(F|_{\partial\Theta})\le F+H+w
 \quad\text{almost everywhere in }\Theta.
\end{equation*}
Continuity makes the two Perron solutions equal everywhere. Uniform stability with respect to the boundary datum then gives resolutivity for every~$f\in C(\partial\Theta)$, for all~$n\ge1$ and~$p>1$; see Section~\ref{sec:r-resolutivity} and Subsection~\ref{sec:r-perron-proof}.

\section{Solutions and the Perron method}
\label{sec:preliminaries}

Let~$n\ge1$ and~$1<p<\infty$. We consider
\begin{equation}\label{eq:pde}
 u_t=\Dp u,\qquad \Dp u=\div A(\nabla u),\qquad A(\xi)=|\xi|^{p-2}\xi\,.
\end{equation}
Here~$\nabla$ and~$\div$ act only on the spatial variables. We also write~$p'=p/(p-1)$. Unless another dependence is stated, positive constants~$C,c$ depend only on~$n,p$ and may change between occurrences; the same convention applies to~$\lesssim$ and~$\gtrsim$. Balls without a displayed center are centered at the spatial origin.

\subsection{Weak solutions}
\begin{definition}[Weak solutions]\label{def:weak-solutions}
Let~$\Theta\subset\R^{n+1}$ be open. A function~$u$ is a weak supersolution of~\eqref{eq:pde} in~$\Theta$ if it belongs to~$L^p(t_1,t_2;W^{1,p}(D))$ on every open cylinder~$D\times(t_1,t_2)\Subset\Theta$ and satisfies
\begin{equation*}
 \iint_\Theta\bigl(-u\varphi_t+A(\nabla u)\cdot \nabla \varphi\bigr)\,dx\,dt\ge0
 \qquad\bigl(0\le\varphi\in C_c^\infty(\Theta)\bigr)\,.
\end{equation*}
A weak subsolution has the same local Sobolev regularity and satisfies the reverse inequality for every nonnegative test function. A weak solution has this regularity and satisfies equality for every~$\varphi\in C_c^\infty(\Theta)$. A continuous weak solution is called~$p$-parabolic.
\end{definition}

Definition~\ref{def:weak-solutions} imposes no time-continuity condition. On a full finite cylinder~$D\times(b,c)$ with~$D$ smooth and bounded, we use bounded functions in~$L^p(b,c;W^{1,p}(D))\cap L^\infty(b,c;L^2(D))$. Zero lateral values mean membership in~$L^p(b,c;W^{1,p}_0(D))$. For evolution problems with bounded forcing, the solutions belong to~$C([b,c];L^2(D))$; see~\cite[Theorem~1.1 and Remarks~1.1(1) and~1.2(2)]{Liu}. If~$W^{1,p}_0(D)$ does not embed continuously into~$L^2(D)$, the construction uses~$W^{1,p}_0(D)\cap L^2(D)$ as the test space, as in Appendix~\ref{sec:r-flux}. We use the~$L^2$ traces for initial values. Time regularization and truncation justify the energy tests; for obstacle potentials, we first test the bounded penalizations and then pass to the limit.

\subsection{Superparabolic functions}
For a spatial box~$D=\prod_{i=1}^n(a_i,b_i)$ and~$t_1<t_2$, we call~$Q=D\times(t_1,t_2)$ a space-time box and write
\begin{equation*}
 \partial_p Q=(\overline D\times\{t_1\})\cup(\partial D\times[t_1,t_2]),
\end{equation*}
for its parabolic boundary. We use the comparison definition of~\cite[Definition~2.2]{BBGP}.

\begin{definition}[Superparabolic and subparabolic functions]\label{def:superparabolic}
Let~$\Theta\subset\R^{n+1}$ be open. A function~$u:\Theta\to(-\infty,\infty]$ is~$p$-superparabolic in~$\Theta$ if
\begin{enumerate}[(i)]
\item $u$ is lower semicontinuous;
\item $u$ is finite on a dense subset of~$\Theta$;
\item for every space-time box~$Q\Subset\Theta$ and every~$p$-parabolic function~$v$ in~$Q$ that is continuous on~$\overline Q$, the inequality~$v\le u$ on~$\partial_pQ$ implies~$v\le u$ in~$Q$.
\end{enumerate}
A function~$v:\Theta\to[-\infty,\infty)$ is~$p$-subparabolic if~$-v$ is~$p$-superparabolic.
\end{definition}

Every superparabolic function used in an energy or measure argument below is locally bounded. Such functions are weak supersolutions by~\cite[Theorem~8]{KLPSh} for~$p>2$ and~\cite[Theorem~2.12]{GKM} for~$1<p<2$; at~$p=2$ this is the classical result for supercaloric functions. Finiteness on a dense subset alone does not imply the weak supersolution property in the degenerate range, as explained in~\cite[Theorems~2--3]{KLPSh}.

Conversely, a locally bounded weak supersolution has the canonical representative
\begin{equation*}
 u_*(x,t)=\essliminf_{(y,\tau)\to(x,t)}u(y,\tau),
\end{equation*}
which is lower semicontinuous and equals~$u$ almost everywhere for every~$p>1$ by~\cite[Theorem~2.3, with time exponent one]{Liao}. The theorem requires local spatial~$p$-energy, local integrability and a local essential lower bound. These hold for the bounded functions used here; no time-continuity assumption is needed. Applying~\cite[Lemma~2.7]{BBGP} to~$u_*$ on a compactly contained box gives comparison with a continuous~$p$-parabolic function almost everywhere. Essential lower-limit regularization gives the inequality at every interior point. We always use this representative. In particular, if~$u\ge F$ almost everywhere on an open set and~$F$ is continuous, then~$u\ge F$ everywhere there. Negation gives the corresponding upper representative and upper-limit statements for subsolutions.

For the bounded superparabolic functions in the Wiener arguments, the representative also satisfies
\begin{equation}\label{eq:past-representative}
 u(x,t)=\essliminf_{(y,\tau)\to(x,t),\ \tau<t}u(y,\tau)
\end{equation}
at every point for all~$p>1$ by~\cite[Theorem~3.1, with time exponent one]{Liao}. See also~\cite[Theorem~5.1]{KLpoint} for~$p>2$ and Subsection~\ref{subsec:capacity-foundations} for~$1<p<2$. The proof of resolutivity for every~$p>1$ uses only the full essential lower-limit representative~$u_*$.

\begin{definition}[Riesz measure]\label{def:riesz-measure}
Let~$u$ be locally bounded and~$p$-superparabolic in an open set~$\Theta\subset\R^{n+1}$. Its Riesz measure is the nonnegative Radon measure~$\mu_u$ determined by
\begin{equation}\label{eq:riesz}
 \iint_\Theta\bigl(-u\varphi_t+A(\nabla u)\cdot \nabla \varphi\bigr)\,dx\,dt
 =\int_\Theta\varphi\,d\mu_u
 \qquad\bigl(\varphi\in C_c^\infty(\Theta)\bigr)\,.
\end{equation}
\end{definition}

\subsection{Upper and lower Perron classes}
All Perron boundary conditions refer to the full topological boundary.

\begin{definition}[Perron classes and solutions]\label{def:perron}
Let~$\Theta\subset\R^{n+1}$ be nonempty, bounded and open, and let~$f\in C(\partial\Theta)$. The upper class~$\mathcal U_f$ consists of~$p$-superparabolic functions~$U$ in~$\Theta$ that are bounded below and satisfy
\begin{equation*}
 \liminf_{\Theta\ni z\to\zeta}U(z)\ge f(\zeta)
 \qquad(\zeta\in\partial\Theta)\,.
\end{equation*}
The lower class~$\mathcal L_f$ consists of~$p$-subparabolic functions~$V$ in~$\Theta$ that are bounded above and satisfy
\begin{equation*}
 \limsup_{\Theta\ni z\to\zeta}V(z)\le f(\zeta)
 \qquad(\zeta\in\partial\Theta)\,.
\end{equation*}
The upper and lower Perron solutions are
\begin{equation*}
 \upperH_\Theta f(z)=\inf_{U\in\mathcal U_f}U(z),\qquad
 \lowerH_\Theta f(z)=\sup_{V\in\mathcal L_f}V(z)\,.
\end{equation*}
These definitions apply at every~$z\in\Theta$.
\end{definition}
These are the classes of~\cite[Definition~2.8]{BBGP}. Constants make them nonempty. Comparison gives~$\lowerH_\Theta f\le\upperH_\Theta f$, and both Perron solutions are bounded continuous weak solutions in~$\Theta$; see~\cite[Theorem~2.6 and the paragraph after Definition~2.8]{BBGP}.

\begin{definition}[Resolutive data and regular boundary points]\label{def:regularity}
Let~$\Theta\subset\R^{n+1}$ be nonempty, bounded and open. A datum~$f\in C(\partial\Theta)$ is resolutive if~$\lowerH_\Theta f=\upperH_\Theta f$ everywhere in~$\Theta$. A point~$\zeta_0\in\partial\Theta$ is regular for~\eqref{eq:pde} in~$\Theta$ if
\begin{equation*}
 \lim_{\Theta\ni z\to\zeta_0}\upperH_\Theta f(z)=f(\zeta_0)
 \qquad\text{for every }f\in C(\partial\Theta)\,.
\end{equation*}
\end{definition}
Sign symmetry gives~$\lowerH_\Theta f=-\upperH_\Theta(-f)$, so the same definition results from using the lower Perron solution; see~\cite[Definition~3.1 and the following observation]{BBGP}.

\section{Potentials and intrinsic capacity}
\label{sec:potential}

We define the reference potentials and their capacities, then use intrinsic scaling to formulate the Wiener terms~$\omega_a(j)$. Successful capacity tests give the positive increments used for sufficiency; failed tests give the upper bounds used for necessity. Capacity approximation and bounded penalization justify the energy estimates and the use of continuous obstacles on fixed neighborhoods of the complement.

Throughout Sections~\ref{sec:potential}--\ref{sec:necessity},~$n\ge2$ and~$p\in(\pc,2)\cup(2,\infty)$, where~$\pc=2n/(n+1)$. The range is specified in each subsection or statement.

\subsection{Balayage and capacities}\label{subsec:capacity-foundations}
The definitions below are shared by the two ranges. We use the cited results of~\cite{KKKP,AKP} for~$p>2$; the singular capacity and representative results are specified at the end of this subsection.

\begin{definition}[Balayage and capacitary potentials]\label{def:reduction}
Let~$D\subset\R^n$ be a smooth bounded domain and let~$b\in\R$. For a bounded nonnegative obstacle~$\psi$ compactly supported in~$D\times(b,\infty)$, its reduced function~$u_\psi$ is the pointwise infimum of all nonnegative~$p$-superparabolic functions~$v$ on~$D\times(b,\infty)$ satisfying~$v\ge\psi$. The balayage~$\redpot_\psi$ of~$\psi$ is the lower semicontinuous regularization of~$u_\psi$,
\begin{equation*}
 \redpot_\psi(z)=\lim_{r\downarrow0}\inf_{\zeta\in B_r^{n+1}(z)}u_\psi(\zeta)\,,
\end{equation*}
where~$B_r^{n+1}(z)$ is a Euclidean space--time ball and~$r$ is small enough that it lies in~$D\times(b,\infty)$. For compact~$K\Subset D\times(b,\infty)$ and~$a>0$, let~$\redpot_{K,D,b}^a$ be the balayage of~$a\ind_K$. If~$Q=D\times(b,c)$ and~$K\Subset Q$, write~$\redpot_{K,Q}^a$ for its restriction to~$Q$. After~$c$ the potential is continued by the homogeneous equation with zero lateral values.
\end{definition}
We use~$\redpot_\psi$ in all capacity and comparison arguments. The reduced function~$u_\psi$ and the balayage~$\redpot_\psi$ are distinguished as in~\cite[Definition~4.1]{KKKP} and~\cite[Definition~2.15]{MS}; the corresponding hatted notation for capacitary potentials is also used in~\cite[Definition~2.5]{AKP}.

The inequality~$u_\psi\ge\psi$ need not be preserved by lower semicontinuous regularization when~$\psi$ is discontinuous. For example, let~$F\Subset D$ be a closed ball, let~$t_1>b$ and~$a>0$, and set~$u=\redpot^a_{F\times\{t_1\},D,b}$. Then~$u=0$ on~$D\times(b,t_1]$, whereas
\begin{equation*}
 \lim_{t\downarrow t_1}u(\cdot,t)=a\ind_F
 \quad\text{in }L^2(D),\qquad
 \mu_u=a\ind_F(x)\,dx\otimes\delta_{t_1}\,.
\end{equation*}
Here~$\delta_{t_1}$ is the unit point mass at~$t_1$. In particular,~$\mu_u(F\times\{t_1\})=a|F|$; see Proposition~\ref{lem:capacity-time-slice}. This positive mass is concentrated on a single time slice and cannot be recovered by integrating the elliptic capacities of the spatial sections with respect to time. We therefore use admissible majorants or continuous outer approximations when comparing compact obstacles.

The representative in~\eqref{eq:past-representative} is determined by earlier times. For~$p>2$, unbounded superparabolic functions may have infinite right limits, as in the friendly giant; see~\cite[Section~3]{KLPSh}.

If~$0\le\psi\le a$, replacing any admissible majorant~$v$ by~$\min\{v,a\}$ preserves superparabolicity and admissibility. Thus the same infimum is obtained from majorants bounded between zero and~$a$. $\redpot^a_{K,Q}$ has zero initial and lateral values, satisfies~$0\le \redpot^a_{K,Q}\le a$, and is~$p$-parabolic outside~$K$; see~\cite[Section~4 and Theorem~5.7]{KKKP}. There is no new source at~$c$, so homogeneous continuation does not change the Riesz mass of~$K$.

Zero initial and lateral values for these reference potentials are understood as in~\cite[Section~3.1]{KKKP}: on the full reference cylinder,
\begin{equation*}
 u\in L^p\bigl(b,\infty;W^{1,p}_0(D)\bigr),\qquad
 \lim_{\delta\downarrow0}\frac1\delta
 \int_b^{b+\delta}\!\int_D u^2\,dx\,dt=0\,.
\end{equation*}
The zero lateral trace is part of this energy-space condition. For the bounded compact obstacles used here, the approximation in Lemma~\ref{lem:penalty} preserves these zero initial and lateral values.

\begin{definition}[Nonlinear parabolic capacity]\label{def:parabolic-capacity}
Let~$D\subset\R^n$ be a smooth bounded domain, let~$b<c$, and put~$Q=D\times(b,c)$. For compact~$K\Subset Q$ and~$a>0$, the nonlinear parabolic capacity at height~$a$ is
\begin{equation}\label{eq:capacity-definition}
 \capc_a(K,Q)=\mu_{\redpot^a_{K,Q}}(K)\,.
\end{equation}
The potential and its homogeneous continuation are those of Definition~\ref{def:reduction}.
\end{definition}
Equivalently,
\begin{equation*}
 \capc_a(K,Q)
 =\sup\bigl\{\nu(K):\nu\ge0,\ \operatorname{supp}\nu\subset K\,,
                 \ 0\le u_\nu\le a\bigr\}\,,
\end{equation*}
where~$\nu$ ranges over nonnegative Radon measures and~$u_\nu$ is the solution with source~$\nu$ and zero initial and lateral values in~$D\times(b,\infty)$. The height bound holds throughout this cylinder. The representation for compact sets, continuity of capacity on decreasing compact sets, and approximation by smooth outer obstacles follow from~\cite[Theorem~5.7, Lemmas~5.8 and~4.2]{KKKP}. At unit height, the notation agrees with the nonlinear parabolic capacity of~\cite[Definition~5.1, Remark~5.2, and Theorem~5.7]{KKKP}:
\begin{equation*}
 \capc_1(K,Q)=\capc\bigl(K,D\times(b,\infty)\bigr)\,.
\end{equation*}
The convention of homogeneous continuation after~$c$ is the one in~\cite[Remark~5.2(2)]{KKKP}. For compact sets in the full continuation cylinder, the variational capacity is comparable to the measure capacity~$\capc$ by~\cite[Theorem~4.9]{AKP}.

When no reference domain is specified, we use whole-space potentials in~$\R^n\times\R$ with zero data before the first source. Bounded compactly supported obstacles have whole-space potentials obtained by exhaustion, as in~\cite[Remark~5.2(3)]{KKKP}. The estimates below are uniform in the exhausting domains.

For~$\pc<p<2$, Definitions~\ref{def:reduction} and~\ref{def:parabolic-capacity} use the same zero initial and lateral energy values and homogeneous continuation. For compact sets in the singular range, these properties are justified in~\cite[Definition~3.1 and Lemma~3.2]{MS}. The scalar~$a$ is the prescribed obstacle height, and all capacities in this argument are taken on compact sets. The intrinsic change of variables in Subsection~\ref{subsec:intrinsic-scaling} multiplies Riesz mass by~$qr^n$.

In this range, bounded superparabolic functions are weak supersolutions by~\cite[Theorem~2.12]{GKM}, and bounded increasing limits remain superparabolic by~\cite[Lemma~2.8]{GKM}. Continuous compactly supported obstacles have continuous potentials by~\cite[Lemma~2.16]{MS}. Compact potentials are~$p$-parabolic off the obstacle by~\cite[Lemma~2.17]{MS}, and their total Riesz mass equals the measure capacity by~\cite[Lemma~3.2]{MS}. Capacity is continuous along decreasing compact sets by~\cite[Lemma~3.3]{MS}.

On a common cylinder, an almost-everywhere limit of uniformly bounded weak supersolutions is a weak supersolution, and the corresponding Riesz measures converge weakly; see~\cite[Lemma~2.2]{MS}.

The canonical superparabolic representative of a bounded weak supersolution agrees almost everywhere with that function and is given by the essential lower limit from earlier times; see~\cite[Lemmas~2.4 and~2.13--2.14]{MS}. In particular, a canonical supersolution which lies above a continuous obstacle almost everywhere lies above it everywhere. The corresponding statements for subsolutions follow by negation. The canonical value on a time-slice obstacle is not identified with its right trace.

In both ranges, capacity has the following monotonicity properties:
\begin{align}
 K_1\subset K_2&\quad\Longrightarrow\quad
 \capc_a(K_1,Q)\le\capc_a(K_2,Q)\,,
 \label{eq:set-monotonicity}\\
 a_1\le a_2&\quad\Longrightarrow\quad
 \capc_{a_1}(K,Q)\le\capc_{a_2}(K,Q)\,,
 \label{eq:cap-height-monotone}\\
 Q_1\subset Q_2,\ K\Subset Q_1&\quad\Longrightarrow\quad
 \capc_a(K,Q_2)\le\capc_a(K,Q_1)\,.
 \label{eq:domain-capacity}
\end{align}
In~\eqref{eq:domain-capacity}, both problems have zero data before their first sources. An admissible measure in the definition of~$\capc_a(K,Q_2)$ (Definition~\ref{def:parabolic-capacity}) gives a smaller or equal solution in the smaller zero-boundary problem. This proves the capacity inequality. The potentials increase with the spatial domain:~$\redpot_{K,Q_1}^a\le \redpot_{K,Q_2}^a$ on~$Q_1$.

\subsection{Intrinsic scaling}\label{subsec:intrinsic-scaling}
Put
\begin{equation}\label{eq:lambda}
 \lambda=n(p-2)+p>0\,.
\end{equation}
The height~$a$ is fixed while the trial increment~$q$ varies. An increment~$q$ at spatial scale~$r$ has time scale~$q^{2-p}r^p$. Under
\begin{equation*}
 x=ry,\qquad t=b+q^{2-p}r^p\tau,\qquad u(x,t)=qv(y,\tau)\,,
\end{equation*}
the equation is unchanged and Riesz mass is multiplied by~$qr^n$. Indeed, both terms of~\eqref{eq:riesz} acquire this factor after change of variables. In particular, the capacitary mass has the same scaling. For~$Q=D\times(b,c)$, the exact relation to the unit-height notation is
\begin{equation*}
 \capc_a(K,Q)
 =a\,\capc\bigl(\{(x,a^{p-2}(t-b)):(x,t)\in K\},D\times(0,\infty)\bigr)\,.
\end{equation*}
Thus normalizing the obstacle height also changes the time geometry. For the Wiener series, the intrinsic duration~$q^{2-p}r^p$ is fixed at~$s_j=2^{-j}$. We compare the capacity at the prescribed height~$a$ with the scale~$qr^n$ of the source mass.

At fixed increment~$q$, the intrinsic spatial radius is proportional to~$s^{\frac{1}{p}}$. For a Barenblatt solution of fixed mass, the radius of its half-maximum superlevel set is proportional to~$s^{\frac{1}{\lambda}}$, with~$\lambda$ defined by~\eqref{eq:lambda}; see~\cite[Section~2]{BBG}. Thus these two families of sets have different spatial scaling in time when~$p\ne2$.

For the singular arguments below, we write
\begin{equation}\label{s:singular-range}
 n\ge2,\qquad \pc<p<2,\qquad \lambda>0\,.
\end{equation}
Since~$p>2n/(n+2)$, the embedding~$W^{1,p}_0(D)\hookrightarrow L^2(D)$ is compact and dense for every bounded ball~$D$.

\subsection{Intrinsic Wiener terms}\label{subsec:capacitary-series}\label{subsec:singular-test}
At each dyadic time scale, the capacity test determines increments for which the potential has a positive lower bound on a later cylinder containing the origin. Fix~$L>4$ and~$C_*>0$ from Lemma~\ref{lem:reception}. For~$\pc<p<2$, that lemma takes~$L=8$ and also fixes~$c_0,\eta$. These constants depend only on~$n,p$ and are fixed before the upper estimates.

\begin{definition}[Intrinsic Wiener terms]\label{def:wiener}\label{def:wiener-degenerate}\label{def:wiener-singular}
Let~$\Theta\subset\R^{n+1}$ be bounded and open, let~$0\in\partial\Theta$, and let~$a>0$. Put
\begin{equation*}
 E=\spacetime\setminus\Theta,\qquad s_j=2^{-j},\qquad j=1,2,\ldots\,.
\end{equation*}
For~$0<q\le a$, define
\begin{equation}\label{eq:intrinsic-radius}
 r_j(q)=\bigl(q^{p-2}s_j\bigr)^{\frac{1}{p}}\,.
\end{equation}
With the intrinsic cylinders of~\eqref{eq:intro-intrinsic-cylinders}, set
\begin{align}
 K_{j,q}&=E\cap\overline{Q_{r_j(q)}^{q,-}}
         =E\cap\bigl(\overline B_{r_j(q)}\times[-4s_j,-2s_j]\bigr),\label{eq:test-set}\\
 Q_{j,q}&=Q_{r_j(q)}^q
         =B_{8Lr_j(q)}\times(-5s_j,s_j).\label{eq:test-cylinder}
\end{align}
The Wiener term at height~$a$ is
\begin{equation}\label{eq:omega}
 \omega_a(j)=\sup\left(\left\{q\in(0,a]:\capc_a(K_{j,q},Q_{j,q})>C_*q\,r_j(q)^n\right\}\cup\{0\}\right)\,.
\end{equation}
An increment~$q\in(0,a]$ is admissible at height~$a$ and index~$j$ if it satisfies the strict capacity inequality in~\eqref{eq:omega}.
\end{definition}

The relation~$q^{2-p}r_j(q)^p=s_j$ fixes the duration while~$q$ changes the spatial geometry. We have~$0\le\omega_a(j)\le a$, and admissibility need not be monotone in~$q$. Capacity monotonicity in the height gives
\begin{equation}\label{eq:height-omega}
 a_1\le a_2\quad\Longrightarrow\quad\omega_{a_1}(j)\le\omega_{a_2}(j)\,.
\end{equation}
For~$p>2$, Lemma~\ref{lem:reception} gives a lower bound~$q$ on
\begin{equation}\label{eq:receiver}
 G_{j,q}=B_{Lr_j(q)}\times[-s_j,s_j/2]\,.
\end{equation}
Figure~\ref{fig:intrinsic-test} shows these cylinders. For~$\pc<p<2$, Lemma~\ref{lem:reception} gives the lower bound~$c_0q$ on~$B_{r_j(q)}\times[-\eta s_j,\eta s_j]$. In that range, small increments can give large radii; Remark~\ref{s:singular-locality} supplies the locality needed in the proof.

The equivalence between regularity and divergence of the intrinsic series is proved in Theorem~\ref{thm:main}.

\begin{figure}[htbp]
\centering
\begin{tikzpicture}[
  x=1cm,y=.82cm,>={Stealth[length=4pt]},
  every node/.style={font=\footnotesize},
  axis/.style={thin,->},
  guide/.style={thin,densely dashed,black!45}
]
% A single smooth boundary separates the domain from its complement.
\def\exterior{(-3.18,-5.12)--
  plot[domain=-5.12:1.18,variable=\t,samples=65]
    ({.16*\t+.013*\t*\t},\t)--(-3.18,1.18)--cycle}
\path[fill=black!8] \exterior;
\begin{scope}
  \clip (-.82,-4) rectangle (.82,-2);
  \path[fill=black!35] \exterior;
\end{scope}
% The full auxiliary cylinders extend across the boundary.
\draw[densely dashed,black!60,line width=.65pt]
  (-3.05,-5) rectangle (3.05,1);
\draw[black!70,line width=.65pt] (-2.12,-1) rectangle (2.12,.5);
\draw[densely dashed,black!60] (-.82,-4) rectangle (.82,-2);
\draw[black,line width=.9pt]
  plot[domain=-5.12:1.18,variable=\t,samples=65]
    ({.16*\t+.013*\t*\t},\t);
\draw[axis] (-3.60,-5.18)--(-3.60,1.48) node[above] {$t$};
\draw[axis] (-3.35,0)--(3.55,0) node[right] {$x$};
\fill (0,0) circle (1.7pt);
\node[above right,fill=white,inner sep=1pt] at (.08,.02) {$(0,0)$};
\foreach \yy/\lab in {-5/{-5s},-4/{-4s},-2/{-2s},-1/{-s},.5/{s/2},1/{s}} {
  \draw[guide] (-3.60,\yy)--(-3.05,\yy);
  \node[left] at (-3.64,\yy) {$\lab$};
}
\node[anchor=north east] at (2.97,.94) {$Q_{j,q}$};
\node[anchor=north west] at (-2.03,.45) {$G_{j,q}$};
\node[anchor=south east,fill=white,inner sep=1pt] at (2.02,-.91)
  {$\redpot^a_{K_{j,q},Q_{j,q}}\ge q$};
\node at (-2.08,-2.72) {$E$};
\node at (2.23,-2.72) {$\Theta$};
\node[anchor=west] (klab) at (.96,-3.40) {$K_{j,q}$};
\draw[thin] (klab.west)--(-.60,-3.40);
\node[anchor=east] (bdry) at (-.95,-4.72) {$\partial\Theta$};
\draw[thin] (bdry.east)--(-.469,-4.8);
\draw[<->] (1.13,-2)--(1.13,-1);
\node[right] at (1.20,-1.5) {time gap};
\draw[<->] (-.82,-4.28)--(.82,-4.28)
  node[midway,below,fill=white,inner sep=1pt] {$2r$};
\draw[<->] (-2.12,-5.53)--(2.12,-5.53)
  node[midway,below] {$2Lr$};
\draw[<->] (-3.05,-6.09)--(3.05,-6.09)
  node[midway,below] {$16Lr$};
\node[anchor=north] at (0,-6.62)
  {$r=r_j(q),\quad s=s_j,\quad r^p=q^{p-2}s$};
\end{tikzpicture}
\caption{The intrinsic test for~$p>2$ near a smooth boundary, shown in a local spatial cross-section. The domain~$\Theta$ is white, its complement~$E$ is pale gray, and~$K_{j,q}$ is dark gray. The reference cylinder~$Q_{j,q}=B_{8Lr}\times(-5s,s)$ and the region~$G_{j,q}=B_{Lr}\times[-s,s/2]$ cross~$\partial\Theta$. If~$\capc_a(K_{j,q},Q_{j,q})>C_*qr^n$, Lemma~\ref{lem:reception} gives~$\redpot^a_{K_{j,q},Q_{j,q}}\ge q$ throughout~$G_{j,q}$. Spatial widths are schematic.}
\label{fig:intrinsic-test}
\end{figure}
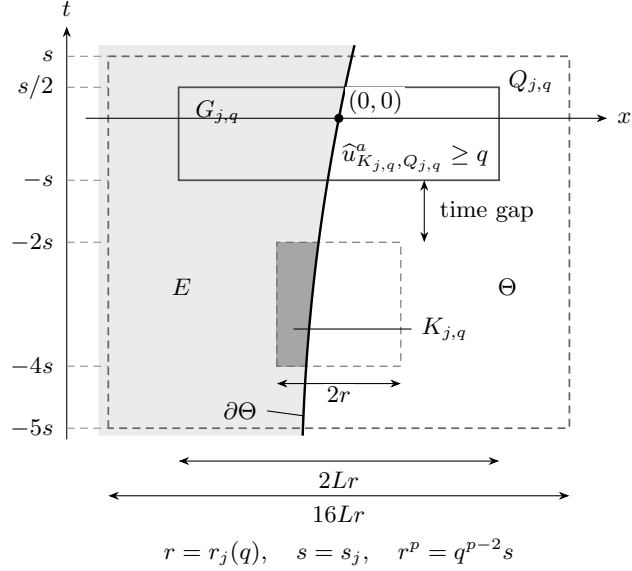

\begin{remark}[The limit~$p\to2$]\label{rem:linear-limit}
The prescribed sets in~\eqref{eq:potential-test-sets} remain well-defined at~$p=2$, but Theorem~\ref{thm:potential-wiener} and Remark~\ref{cor:capacity-wiener} are proved only in the ranges~$\pc<p<2$ and~$p>2$. Their proofs use the intrinsic estimates below and do not give a uniform passage to the heat equation. This is natural since the cases $p\neq 2$ behave quite differently from $p=2$, and this parallels what is known for barriers.

The intrinsic test of Definition~\ref{def:wiener} does not give a Wiener characterization at~$p=2$ by passage to the limit. For fixed~$j$ and~$q>0$, the radius~$r_j(q)=q^{\frac{p-2}{p}}s_j^{\frac{1}{p}}$ tends to~$\sqrt{s_j}$ as~$p\to2$. This convergence is not uniform over~$0<q\le a$, whereas the Wiener terms test arbitrarily small increments. At~$p=2$, changing~$q$ no longer changes the test sets.

For example, let
\begin{equation*}
 \Theta=\{(x,t)\in\R^{n+1}:-1<t<0,\ |x|<\sqrt{-t}\}\,.
\end{equation*}
By~\cite[Theorem~1.1\textnormal{(b)--(c)}]{BBG}, the origin is irregular for~$\pc<p<2$ and regular for~$p=2$. If we substitute~$p=2$ into the intrinsic test, then for every~$j\ge3$,
\begin{equation*}
 \overline B_{\sqrt{s_j}}\times[-4s_j,-2s_j]\subset\Theta\,.
\end{equation*}
The test sets are therefore empty, and the resulting series has a zero tail at every height despite regularity for the heat equation. Its classical Wiener criterion uses heat-kernel level sets~\cite{Lanconelli,EG}.

The proof constants also lose uniformity. In Lemma~\ref{s:regularize-sequence}, the factor~$\gamma_p=2^{-\frac{p+1}{2-p}}$ tends to zero as~$p\uparrow2$, and the required constant~$K\ge2/(\alpha\gamma_p)$, with~$\alpha=2^{-1/[2(2-p)]}$, diverges. In the degenerate range, the argument uses finite propagation, which fails for the heat equation.
\end{remark}

\subsection{Capacity estimates for~$p>2$}
\label{subsec:capacity-estimates}
In this subsection assume~$p>2$.
\begin{definition}[Elliptic capacity]\label{def:elliptic-capacity}
For an open set~$G\subset\R^n$ and compact~$H\Subset G$, the elliptic~$p$-capacity of~$H$ relative to~$G$ is
\begin{equation}\label{eq:elliptic-capacity}
 \capce(H,G)=\inf\left\{\int_G|\nabla \varphi|^p\,dx:\varphi\in C_c^\infty(G),\ \varphi\ge1\text{ on }H\right\}\,.
\end{equation}
Taking~$G=\R^n$ gives the homogeneous elliptic capacity.
\end{definition}
See~\cite[Section~5]{AKP} for this convention.

For~$0<b<R$, the relative elliptic capacity is
\begin{equation}\label{eq:radial-relative-capacity}
 \capce(\overline B_b,B_{2R})
 =|\mathbb S^{n-1}|\left(\int_b^{2R}t^{-\frac{n-1}{p-1}}\,dt\right)^{1-p}.
\end{equation}
For $b\ll R$, its size is $b^{n-p}$ if $p<n$, $(\log(2R/b))^{1-p}$ if $p=n$, and $R^{n-p}$ if $p>n$. The condenser scaling $\capce(\overline B_R,B_{2R})=R^{n-p}\capce(\overline B_1,B_2)$ holds in every case. The degenerate examples in Section~\ref{sec:g-examples} use the asymptotic~$b^{n-p}$ only for~$2<p<n$.

If~$0\le u\le a$ is a measure solution with zero boundary data and~$F$ is compact, source comparison gives~$0\le u_{\mu_u|_F}\le u\le a$. Thus~$\mu_u|_F$ is an admissible measure in the definition of~$\capc_a(F,Q)$ (Definition~\ref{def:parabolic-capacity}), and
\begin{equation}\label{eq:measure-restriction}
 \mu_u(F)\le\capc_a(F,Q)\,.
\end{equation}
The comparison for nonnegative source measures is~\cite[Lemma~3.4]{KKKP}. This argument applies in particular to the bounded penalization sources below.

We use the following estimate in Lemma~\ref{lem:remainder-waiting} to make the source growth bound small by reducing the obstacle height.

\begin{lemma}\label{lem:height-capacity}
For~$0<h\le a$ and every compact~$K$ in a common reference problem,
\begin{equation}\label{eq:height-capacity}
 \capc_h(K)
 \le C\left((h/a)^{p-1}+(h/a)^{\frac{1}{p-1}}\right)\capc_a(K)
 \le C(h/a)^{\frac{1}{p-1}}\capc_a(K)\,.
\end{equation}
The constants depend only on~$n,p$. The same statement holds when both capacities are taken relative to the same finite reference cylinder.
\end{lemma}
\begin{proof}
At unit reference height this is the upper estimate in~\cite[Lemma~6.4]{KKKP}. To compare arbitrary heights on the same set, write~$u(x,t)=a v(x,a^{p-2}(t-b))$ and transform the set and the reference problem by the same time change. Total Riesz mass is multiplied by~$a$. The two obstacle heights become~$1$ and~$h/a$ on the same transformed compact set. Applying the unit-height estimate and scaling back gives~\eqref{eq:height-capacity}. For whole-space capacity, use exhaustion with the same constants, as justified in Subsection~\ref{subsec:whole-space-potentials}.
\end{proof}

\begin{lemma}\label{lem:cylinder-capacity}
Let~$R,s,a>0$ and~$b<t_1<c-s$. Then
\begin{equation}\label{eq:cylinder-capacity}
 \capc_a\bigl(\overline B_R\times[t_1,t_1+s],B_{2R}\times(b,c)\bigr)
 \le C\bigl(aR^n+a^{p-1}sR^{n-p}\bigr)\,.
\end{equation}
The constant depends only on~$n,p$. Estimate~\eqref{eq:cylinder-capacity} also holds for the whole-space capacity of~$\overline B_R\times[t_1,t_1+s]$.
\end{lemma}
\begin{proof}
The measure capacity in~\cite[Definition~3.1 and Lemma~3.2]{MS} uses zero initial and lateral data and the same homogeneous continuation as Definition~\ref{def:parabolic-capacity}. The complement of~$B_{2R}$ is uniformly~$p$-fat with a parameter depending only on~$n,p$, and~$[t_1,t_1+s]\Subset(b,c)$. Thus~\cite[Theorem~1.1 and Lemma~4.8]{MS}, using the reference-ball case of Lemma~4.8, gives~\eqref{eq:cylinder-capacity} at height~$a=1$, with~$C=C(n,p)$.

For general~$a>0$, set~$\tau=a^{p-2}(t-b)$ and~$u(x,t)=av(x,\tau)$. The transformed time interval has length~$a^{p-2}s$ and remains compactly contained in~$(0,a^{p-2}(c-b))$. Total Riesz mass is multiplied by~$a$, so the unit-height estimate gives
\begin{equation*}
 \capc_a\bigl(\overline B_R\times[t_1,t_1+s],B_{2R}\times(b,c)\bigr)\le Ca\bigl(R^n+a^{p-2}sR^{n-p}\bigr)\,.
\end{equation*}
Finally,~\eqref{eq:domain-capacity} bounds the whole-space capacity by this finite-reference capacity.
\end{proof}

\subsection{Whole-space potentials for~$p>2$}
\label{subsec:whole-space-potentials}
Assume~$p>2$.
We justify finite propagation at a fixed height by a radial supersolution. Suppose all sources lie in~$B_R$ and the potential is bounded by~$a$. Put
\begin{equation*}
 d(t)=R+1+\left(\frac{p-1}{p-2}\right)^{p-1}a^{p-2}(t-b)\,.
\end{equation*}
The function
\begin{equation}\label{eq:rough-travelling-wave}
 W(x,t)=\min\{a,\ a(d(t)-|x|)_+^{\frac{p-1}{p-2}}\},\qquad t\ge b\,,
\end{equation}
is a supersolution outside~$B_R$. In the nonconstant region, with~$r=|x|$, we have
\begin{equation*}
 W_t=\partial_r\bigl(|W_r|^{p-2}W_r\bigr)
 =a^{p-1}\left(\frac{p-1}{p-2}\right)^p(d-r)^{\frac{1}{p-2}}\,.
\end{equation*}
The radial curvature term~$(n-1)|W_r|^{p-2}W_r/r$ is nonpositive. At the junction with the constant part, the radial flux decreases. The exterior region has zero initial data, boundary values at most~$a$, and no source. Comparison gives~$u\le W$ there.

The supports of all approximations of obstacles at height~$a$ therefore stay in one fixed ball on every fixed time interval. A sufficiently large bounded reference domain consequently gives the same solution through that interval. By exhausting~$\R^n$ with bounded domains, we may therefore restrict whole-space sources, use continuity of capacity on decreasing compact sets, and apply Lemma~\ref{lem:height-capacity}. After the final source time, the maximum principle keeps the solution below~$a$, so its source remains an admissible measure in the definition of~$\capc_a$ (Definition~\ref{def:parabolic-capacity}).

\subsection{Bounded penalization}\label{subsec:penalization}\label{sec:s-foundations}
Bounded nonnegative sources permit the energy tests and source restrictions used below. A bounded penalty gives the same existence, comparison and convergence argument in both ranges. For continuity, we use the degenerate regularity theorem when~$p>2$ and a torsion barrier with spatial translations when~$\pc<p<2$. For~$\pc<p<2$, the obstacles used below are smooth functions or finite maxima of such functions.

\begin{samepage}
\begin{lemma}[Bounded penalizations and their limit]\label{lem:penalty}\label{s:penalty}
Let~$n\ge2$, $p\in(\pc,2)\cup(2,\infty)$, let~$D$ be smooth and bounded, and put~$Q=D\times(b,c)$, where~$b<c$. Let~$\psi\in C_c(Q)$ satisfy~$0\le\psi\le a$, where~$a\ge0$. If~$p<2$, assume also that~$D=B_{R_D}$ and~$\psi$ is uniformly Lipschitz in the spatial variable. Extend~$\psi$ by zero to~$D\times(b,\infty)$. For each positive integer~$N$, there is a unique energy solution there of
\begin{equation}\label{eq:penalty}
 (u_N)_t-\Dp u_N=N(\psi-u_N)_+\,,
 \qquad u_N=0\text{ on the parabolic boundary}\,.
\end{equation}
It satisfies~$0\le u_N\le a$ and is continuous up to the initial and lateral boundary. The functions increase with~$N$ and converge locally uniformly to the continuous potential~$\redpot_\psi$. Their Riesz measures converge weakly. Put~$K_0=\supp\psi$ and~$M_N=\iint_{D\times(b,\infty)}N(\psi-u_N)_+\,dx\,dt$. The masses~$M_N$ and spatial energies are bounded independently of~$N$. If~$a>0$, then
\begin{equation}\label{eq:penalty-bounds}
 M_N\le\capc_a(K_0,Q),\qquad
 \int_b^\infty\!\int_D|\nabla u_N|^p\,dx\,dt
 \le a\,\capc_a(K_0,Q)\,.
\end{equation}
The constants used for continuity may depend on~$N$.
\end{lemma}
\end{samepage}
\begin{proof}
If~$a=0$, then~$\psi=0$ and~$u_N=0$, so assume~$a>0$.

\emph{Step 1: Existence and comparison.} Fix~$T>b$ and put~$V=W^{1,p}_0(D)$, with norm~$\|\nabla z\|_{L^p(D)}$, and~$H=L^2(D)$. Since~$p>2n/(n+2)$, the embedding~$V\hookrightarrow H$ is dense and compact. Define the bounded penalty
\begin{equation*}
 g_N(x,t,z)=N\bigl(\psi(x,t)-z_+\bigr)_+\,.
\end{equation*}
The penalty is nonincreasing in~$z$ and satisfies~$0\le g_N\le Na$ on the whole real line, so its boundedness does not require~$z\ge0$. The evolution operator~$z\mapsto\Dp z+g_N(\cdot,t,z)$ is measurable in~$t$, hemicontinuous and dissipative by monotonicity of the flux and the decrease of~$g_N$. The bounds~$g_Nz\le Na^2$ and~$|g_N|\le Na$ give the required coercivity and growth. Thus classical monotone evolution theory gives a unique solution; see~\cite{Lions}, \cite[Theorem~30.A]{Zeidler}, and, for a modern treatment,~\cite[Theorem~1.1 and Remarks~1.1(1) and~1.2(2)]{Liu}, applied with~$\alpha=p$, $\beta=0$ and~$\rho=\eta=0$. This solution satisfies
\begin{equation*}
 u_N\in L^p(b,T;V)\cap C([b,T];H),\qquad
 (u_N)_t\in L^{p'}(b,T;V^*)\,.
\end{equation*}
This regularity permits the energy tests of~\cite{DiB,KKKP}. Positive-part comparison with~$0$ and~$a$ gives~$0\le u_N\le a$, so the bounded penalty agrees with~$N(\psi-u_N)_+$. Comparison of two penalty parameters gives~$u_N\le u_{N+1}$. Uniqueness makes the solutions consistent as~$T$ increases, with homogeneous continuation after the support of~$\psi$.

\emph{Step 2: Continuity.} For~$p>2$, the source~$f_N=N(\psi-u_N)_+$ is bounded by~$Na$. The coefficient~$-f_N$ is bounded, the boundary data are zero, and the exterior measure-density condition holds for smooth~$D$. Thus~\cite[Theorems~1 and~3]{DiB86} give continuity in the interior and up to the lateral boundary at positive times. Compact support of~$\psi$ leaves an initial cylinder where uniqueness gives~$u_N=0$, hence continuity also at the initial boundary and its corners.

For~$\pc<p<2$, we use the additional ball and spatial Lipschitz assumptions. A torsion barrier controls the solution at the lateral boundary, and spatial translations give a Lipschitz estimate. Put~$d=p/(p-1)$ and
\begin{equation}\label{s:torsion}
 T_N(x)=\frac1d\left(\frac{Na}{n}\right)^{\frac{1}{p-1}}
       (R_D^d-|x|^d).
\end{equation}
Then $-\Dp T_N=Na$ in $D$, $T_N=0$ on $\partial D$, and $T_N$ is Lipschitz on $\overline D$. Weak comparison gives $u_N\le T_N$.

Let $L_\psi$ be a uniform spatial Lipschitz constant for $\psi$, and let $L_N$ be a Lipschitz constant for $T_N$. For a vector $h$, work on $D_h=D\cap(D-h)$, and set
\begin{equation*}
 k=(L_\psi+L_N)|h|.
\end{equation*}
The positive part of $u_N(x+h,t)-u_N(x,t)-k$ has zero initial and lateral trace on this intersection. At a boundary point where $x\in\partial D$, the torsion bound gives $u_N(x+h,t)\le L_N|h|$; on the other part of the boundary, the translated function has zero trace. These statements also follow as Sobolev trace inequalities. On the set where the positive part is nonzero,
\begin{equation*}
 (\psi(x+h,t)-u_N(x+h,t))_+
 \le (\psi(x,t)-u_N(x,t))_+.
\end{equation*}
Subtracting the translated equations and testing with that positive part therefore gives, by monotonicity,
\begin{equation*}
 u_N(x+h,t)-u_N(x,t)\le k
 \quad\hbox{a.e. on }D_h\times(b,T).
\end{equation*}
Interchanging the points gives the reverse bound. Positive-part tests are justified by time regularization in the energy class. Intersections of two balls are Lipschitz domains, and the zero-trace comparisons can equivalently be carried out by interior exhaustion.

It follows that, for almost every $t$, $u_N(\cdot,t)$ has a spatial Lipschitz representative, with a constant independent of $t$. Its zero extension is also Lipschitz, using $u_N\le T_N$ at the boundary. The $C_tL^2_x$ trace extends this representative to every time. A uniform Lipschitz bound in space and convergence in $L^2$ imply convergence uniformly in space: for a Lipschitz function $z$ with constant at most $L$, the elementary ball argument gives
\begin{equation*}
 \|z\|_\infty\le C_n L^{\frac{n}{n+2}}\|z\|_2^{\frac{2}{n+2}}
\end{equation*}
for the differences of these zero extensions (the case~$L=0$ is immediate). Thus~$u_N\in C([b,T];C(\overline D))$. Compact support of~$\psi$ also leaves an initial cylinder on which~$u_N=0$. The constants in this continuity argument may depend on~$N$.

\emph{Step 3: Comparison with the potential.} Every $u_N$ is now a continuous weak supersolution and hence superparabolic. Let $v\ge\psi$ be a nonnegative superparabolic majorant. In the open set $\{u_N>\psi\}$ the penalization vanishes, so $u_N$ is $p$-parabolic. At its internal boundary $u_N=\psi\le v$; at the initial and lateral boundary $u_N=0\le v$. Parabolic comparison, with no condition on the top, yields $u_N\le v$. Outside this open set the same inequality follows from $u_N\le\psi$. Taking the infimum over $v$ and then regularizing, using continuity of $u_N$, gives $u_N\le \redpot_\psi$.

\emph{Step 4: Uniform mass and energy bounds.} Put~$f_N=N(\psi-u_N)_+$ and~$M_N=\iint f_N\,dx\,dt$. Since~$f_N$ is supported in~$K_0$, we have~$M_N\le N\iint\psi\,dx\,dt<\infty$. Testing the energy equation with~$u_N$ gives, for every~$T>b$,
\begin{equation}\label{s:penalty-energy}
 \frac12\|u_N(\cdot,T)\|_{L^2(D)}^2
 +\int_b^T\!\int_D|\nabla u_N|^p\,dx\,dt\le aM_N\,.
\end{equation}
Time regularization justifies this test in the energy class. Letting~$T\to\infty$ and using Poincar\'e's inequality gives~$u_N\in L^p(b,\infty;W^{1,p}_0(D))$ for each fixed~$N$.

Choose~$\chi\in C_c^\infty(Q)$ with~$0\le\chi\le1$ and~$\chi=1$ near~$K_0$. The equation and H\"older's inequality give
\begin{equation*}
 M_N=-\iint u_N\chi_t\,dx\,dt +\iint|\nabla u_N|^{p-2}\nabla u_N\cdot\nabla\chi\,dx\,dt \le a\|\chi_t\|_{L^1} +\|\nabla\chi\|_{L^p}(aM_N)^{\frac{p-1}{p}}\,.
\end{equation*}
Young's inequality gives
\begin{equation*}
 M_N\le C_p\bigl(a\|\chi_t\|_{L^1}
                  +a^{p-1}\|\nabla\chi\|_{L^p}^p\bigr)\,.
\end{equation*}
By~\eqref{s:penalty-energy}, the spatial energy is also bounded independently of~$N$. Each~$u_N$ has zero Sobolev trace, vanishes before the first obstacle time, and is bounded by~$a$. Thus~$f_N\,dx\,dt$ is an admissible measure in the definition of~$\capc_a(K_0,Q)$ (Definition~\ref{def:parabolic-capacity}), proving the mass bound in~\eqref{eq:penalty-bounds}. Equation~\eqref{s:penalty-energy} gives the bound in~\eqref{eq:penalty-bounds}.

\emph{Step 5: Convergence to the potential.} The pointwise limit~$U=\sup_Nu_N\le a$ is lower semicontinuous and superparabolic. Indeed, on a compact space--time box, subtract~$\varepsilon>0$ from a continuous comparison function. Lower semicontinuity, compactness and monotonicity give one index for which the lowered function lies below~$u_N$ on the parabolic boundary. Comparison and then~$\varepsilon\downarrow0$ prove the assertion. The uniform energy bound also preserves membership in~$L^p(b,\infty;W^{1,p}_0(D))$ in the limit.

Dividing the bound for the source mass by~$N$ gives
\begin{equation*}
 \iint(\psi-u_N)_+\,dx\,dt=M_N/N\longrightarrow0\,.
\end{equation*}
Hence~$U\ge\psi$ almost everywhere. Continuity of~$\psi$ and~\eqref{eq:past-representative} give the inequality everywhere. Minimality yields~$\redpot_\psi\le U$, while Step~3 gives~$U\le\redpot_\psi$. Thus~$U=\redpot_\psi$.

The results for continuous obstacles in~\cite[Section~4, before Lemma~4.2]{KKKP} for~$p>2$ and~\cite[Theorem~2.11]{GKM} for~$\pc<p<2$ give continuity of~$\redpot_\psi$. Dini's theorem therefore gives uniform convergence on compact subsets. Weak convergence of the Riesz measures follows from~\cite[Theorem~3.2]{KKKP} and~\cite[Lemma~2.2]{MS} in the respective ranges.

For compact potentials, bounded outer obstacle approximations in a fixed compact neighborhood admit one common cutoff~$\chi$, so the same uniform energy argument applies.
\end{proof}

\begin{remark}[Continuity in the singular range]\label{s:continuity-repair}
The continuity result of~\cite{DiB86} applies in the degenerate range; its Remark~0.10 excludes the singular range. The torsion and translation argument above supplies continuity for the singular penalizations, with constants which may depend on~$N$.
\end{remark}

For~$p>2$, whole-space applications of Lemma~\ref{lem:penalty} may be made first in a ball containing the supports in~\eqref{eq:rough-travelling-wave}. In particular, the~$L^1$ time-continuity and integrations by parts used below are justified on a common finite time horizon.

\begin{lemma}[Restriction of a bounded source]\label{s:source-restriction}
Let $u$ be a nonnegative bounded energy solution in a smooth ambient reference cylinder. Assume that $u\le a$ and that its source~$f$ is bounded and nonnegative. Suppose $K\Subset Q$ is compact and $Q$ is a smaller smooth cylinder inside the ambient one. Then
\begin{equation}\label{s:source-cap}
 \iint_K f\,dx\,dt\le \capc_{a}(K,Q).
\end{equation}
\end{lemma}
\begin{proof}
Solve in~$Q$ with the bounded source~$f\mathbf1_K$ and zero initial and lateral data by the same monotone evolution theorem. Comparison with~$u$, whose initial and lateral values are nonnegative, bounds the new solution between zero and~$a$. Continue it homogeneously after the top of~$Q$. Its height stays at most~$a$, and testing and Poincar\'e's inequality give the global zero-trace energy condition. Thus~$f\mathbf1_K\,dx\,dt$ is an admissible measure in the definition of~$\capc_a(K,Q)$ (Definition~\ref{def:parabolic-capacity}), and its total mass gives~\eqref{s:source-cap}.
\end{proof}

\section{Lower bounds for capacitary potentials}
\label{sec:lower}

Let~$\Theta\subset\R^{n+1}$ be bounded and open, let~$0\in\partial\Theta$, and put~$E=\R^{n+1}\setminus\Theta$. We prove a mass-to-positivity estimate for local capacitary potentials, with constants independent of the obstacle height~$a$ (Lemma~\ref{lem:reception}).

We next select nested cylinders while retaining a fixed fraction of the positive increments (Lemma~\ref{lem:selection}). Successive comparisons add these increments on a single superparabolic majorant. Divergence at height~$a$ therefore gives a finite-union potential at height~$H>a$ that is at least~$H-a$ near the origin (Proposition~\ref{prop:saturation}). Its reflection supplies the approximate barriers used in Subsubsection~\ref{subsec:joint-sufficiency}.

\begin{lemma}[Comparison on a finite cylinder]\label{lem:local-majorant-comparison}
Let~$Q=D\times(b,c)$ be a smooth bounded cylinder, let~$K\Subset Q$ be compact and let~$a>0$. If~$w\ge0$ is superparabolic in~$Q$ and~$w\ge a$ pointwise on~$K$, then
\begin{equation}\label{eq:local-majorant-comparison}
 \redpot^a_{K,Q}\le w\qquad\text{in }Q\,.
\end{equation}
\end{lemma}
\begin{proof}
The potential~$\redpot=\redpot^a_{K,Q}$ is defined on~$D\times(b,\infty)$. It is~$p$-parabolic on~$Q\setminus K$ and satisfies~$0\le \redpot\le a$. At each boundary point~$z\in K$ of~$Q\setminus K$,
\begin{equation*}
 \limsup_{Q\setminus K\ni y\to z}\redpot(y)
 \le a\le w(z)
 \le\liminf_{Q\setminus K\ni y\to z}w(y)\,.
\end{equation*}
The last inequality is lower semicontinuity of~$w$.

Since~$D$ is smooth, its complement is uniformly~$p$-fat. Apply~\cite[Lemma~2.17]{MS} to the bounded compactly supported obstacle~$a\ind_K$. In both ranges of~$p$, it gives continuity up to the parabolic boundary, zero lateral Sobolev trace, and vanishing in an initial cylinder. Hence~$\redpot$ has continuous zero limits on the lateral and bottom boundary of~$Q$, including their corners.

Since~$w\ge0$, the boundary-limit inequality holds on these portions too. Apply~\cite[Theorem~2.4]{BBGP} in~$Q\setminus K$ with terminal time~$c$. The upper limits of~$\redpot$ are finite and the lower limits of~$w$ are nonnegative. No condition is required at the top time~$c$. This gives comparison off~$K$, while~$\redpot\le a\le w$ gives it on~$K$.
\end{proof}

\subsection{Homogeneous estimates}\label{subsec:homogeneous-lower}
Homogeneous upper bounds control the flux through a source-free annulus. Weak Harnack and positivity estimates convert a small later value into bounds for earlier spatial averages.

\paragraph{The range~$p>2$.}
We use the weak Harnack estimate from~\cite[Theorem~1.1]{KuusiH}. If~$v\ge0$ is a weak supersolution in~$B_{8R}\times(t_0,t_0+T_0)$, then for almost every admissible Lebesgue time~$t$ we have
\begin{equation}\label{eq:weak-harnack}
 \vint_{B_R}v(x,t)\,dx
 \le
 \left(\frac{C_1R^p}{t_0+T_0-t}\right)^{\frac{1}{p-2}}
 +C_2\essinf_{B_{4R}\times(t+\theta/2,t+\theta)}v\,,
\end{equation}
where
\begin{equation*}
 \theta=\min\left\{t_0+T_0-t,\,
 C_1R^p\left(\vint_{B_R}v(x,t)\,dx\right)^{2-p}\right\}\,.
\end{equation*}
The constants~$C_1,C_2$ depend only on~$n,p$. The formula for~$\theta$ is used when the average is positive; if it is zero,~\eqref{eq:weak-harnack} is immediate.

For local upper bounds we use~\cite[Theorem~3.1(i)]{KP}. If~$v\ge0$ is~$p$-parabolic in~$B_{2R}\times(t_0,t_0+T)$, then, for~$0<\tau\le T$,
\begin{equation*}
 \esssup_{B_R\times(t_0+\tau/2,t_0+\tau)}v
 \le C\left(\frac{R^p}{\tau}\right)^{\frac{1}{p-2}}
 +\frac{C\tau}{R^p}
 \left(\esssup_{t_0<t<t_0+\tau}\vint_{B_{2R}}v(x,t)\,dx\right)^{p-1}\,.
\end{equation*}
Thus a bound for the local spatial averages gives a uniform bound on a smaller cylinder, with constants depending on the fixed enlargement of the spatial cylinders and the gaps between the time intervals. We apply this estimate only where the potential is~$p$-parabolic.

\paragraph{The range~$\pc<p<2$.}
The local smoothing estimate of~\cite{BIV} bounds the supremum of a homogeneous comparison solution in terms of its initial mass.

\begin{lemma}[Local smoothing estimate]\label{s:BIV-upper}
Let~$\pc<p<2$ and let $z\ge0$ be $p$-parabolic in $B_R\times(t_0,t_0+T)$ and have an $L^1$ initial trace there. For every $0<\theta<1$,
\begin{equation}\label{s:BIV-upper-eq}
 \sup_{B_{\theta R}}z(\cdot,t_0+T)
 \le C\left[
 T^{-\frac{n}{\lambda}}\left(\int_{B_R}z(x,t_0)\,dx\right)^{\frac{p}{\lambda}}
 +\left(\frac{T}{R^p}\right)^{\frac{1}{2-p}}
 \right],
\end{equation}
whenever the final slice is defined by continuous continuation or by the limit from earlier times. The constant depends only on $n,p,\theta$, not on the lateral values of $z$.
\end{lemma}
\begin{proof}
Apply \cite[Theorem~2.1, with integrability exponent $1$]{BIV} in balls of radius $(1-\theta)R/2$ centered in $B_{\theta R}$. Their initial masses are bounded by the mass over $B_R$. In that theorem the smoothing exponent is $1/\lambda$ and the second term is exactly the second term in \eqref{s:BIV-upper-eq}. Continuous weak solutions are among its local strong solutions by \cite[Corollary~2.1]{BIV}. If the initial value is attained in~$L^1(B_R)$, apply the estimate first from a later Lebesgue time and let that time decrease to~$t_0$. The solutions to which we apply the lemma can be continued beyond the final slice, or the estimate passes to that slice by continuity.
\end{proof}

\begin{lemma}[Positive lower bound from initial mass]\label{s:BIV-lower}
Let~$\pc<p<2$. There are constants $m_0,c_1>0$, depending only on $n,p$, such that the zero-boundary solution~$z$ in~$B_{64}\times(0,\infty)$ with bounded nonnegative initial data supported in $B_2$ satisfies
\begin{equation*}
 \inf_{B_2}z(\cdot,t)\ge 2c_1
 \quad\text{for }5/4\le t\le5,
\end{equation*}
provided its initial mass is at least $m_0$.
\end{lemma}
\begin{proof}
For initial mass $m$ in $B_2$, \cite[Theorems~2.2 and~2.4]{BIV} give the lower estimate
\begin{equation*}
 z(x,t)\ge c\left(\frac{t}{2^p}\right)^{\frac{1}{2-p}},
 \qquad x\in B_2,\quad 0<t<c'2^\lambda m^{2-p}.
\end{equation*}
The distance to the boundary of $B_{64}$ satisfies the spatial separation requirement in those results. Choose $m_0$ so that the upper endpoint of the time interval is strictly larger than $5$, and choose $c_1$ below half the minimum of the displayed lower bound on $[5/4,5]$. This also ensures that the solution has not extinguished on this interval. Approximation extends the statement from smooth initial data to bounded $L^1\cap L^2$ initial data.
\end{proof}

\subsection{Preservation and decay of lower bounds}\label{subsec:preservation-lower}
An initial positive lower bound is preserved on a smaller ball for a short intrinsic time when~$p>2$. For~$p<2$, an explicit supersolution bounds its decrease.

\begin{samepage}
\begin{lemma}[Preservation and decay of an initial lower bound]\label{lem:persistence}\label{s:plateau-loss}
Let~$p>1$, $p\ne2$, and let~$v\ge0$ in~$B_R\times(t_1,t_2)$ have initial right trace at least~$b>0$ on~$B_R$.
\begin{enumerate}[(i)]
\item If~$p>2$ and~$v$ is superparabolic, there is~$\kappa=\kappa(n,p)>0$ such that
\begin{equation}\label{eq:persistence}
 v\ge b\quad\text{on }B_{R/2}\times
 \bigl(t_1,\min\{t_2,t_1+\kappa b^{2-p}R^p\}\bigr)\,.
\end{equation}
For a bounded weak supersolution, an~$L^2$ or almost-everywhere initial trace is sufficient.
\item If~$1<p<2$ and~$v$ is a bounded weak supersolution, there is~$C=C(n,p)$ such that
\begin{equation}\label{s:plateau-loss-eq}
 v(x,t)\ge b-C\left(\frac{t-t_1}{R^p}\right)^{\frac{1}{2-p}},
 \qquad |x|\le R/2,\quad t_1<t<t_2\,.
\end{equation}
This estimate holds for the canonical representative; a continuous trace, or the energy~$L^2$ trace, suffices.
\end{enumerate}
\end{lemma}
\end{samepage}
\begin{proof}
\emph{The range~$p>2$.} Set~$v_b=\min\{v,b\}$ and~$u=b-v_b$. The bounded superparabolic function~$v_b$ is a weak supersolution locally, so~$u$ is a bounded weak subsolution with~$0\le u\le b$. The assumed right trace gives~$u(\cdot,t)\to0$ in~$L^2(B_{R'})$ as~$t\downarrow t_1$, through admissible times, for every~$R'<R$. For an almost-everywhere right trace this follows by bounded convergence; an~$L^2$ trace gives the same conclusion by truncation.

Fix~$0<R'<R$. We construct a boundary layer in~$B_{R'}$ whose support remains outside~$B_{R'/2}$. Put
\begin{equation*}
 \gamma=\frac{p-1}{p-2},\qquad
 f(z)=(1-z^2)^\gamma\quad(0\le z\le1),\qquad
 \ell(t)=\bigl(K_0b^{p-2}(t-t_1)\bigr)^{\frac{1}{p}}\,.
\end{equation*}
For~$R'-\ell(t)<r<R'$, set
\begin{equation*}
 B(r,t)=b f\left(\frac{R'-r}{\ell(t)}\right)\,,
\end{equation*}
and set~$B=0$ for~$0\le r\le R'-\ell(t)$. We consider only times when~$\ell(t)\le R'/2$. On~$r=R'$ we have~$B=b$, and its initial interior trace is zero.

We choose~$K_0$ so that~$B_t\ge\Dp B$. Write~$z=(R'-r)/\ell$ and~$g(z)=-f'(z)=2\gamma z(1-z^2)^{\gamma-1}$. In the nonzero region,
\begin{equation*}
 B_t=\frac{b\ell'}\ell\,zg(z),\qquad
 \Dp B=\frac{b^{p-1}}{\ell^p}
       \left[-(g^{p-1})'+(n-1)\frac\ell r g^{p-1}\right]\,.
\end{equation*}
The bracket in the second formula equals
\begin{equation*}
 (2\gamma)^{p-1}z^{p-2}(1-z^2)^{\gamma-1} \times \left[p\gamma z^2-(p-1) +(n-1)\frac\ell r z(1-z^2)\right]\,.
\end{equation*}
Since~$\ell/r\le1$, the expression in square brackets is negative for~$z$ in a fixed neighborhood of zero. On the part where it is positive,~$z$ is bounded away from zero. Consequently,
\begin{equation*}
 \left[-(g^{p-1})'+(n-1)\frac\ell r g^{p-1}\right]_+
 \le C(n,p)zg(z),\qquad 0<z<1\,.
\end{equation*}
The common factor~$(1-z^2)^{\gamma-1}$ cancels when comparing the two sides, so the bound remains valid up to~$z=1$. Since~$\ell^{p-1}\ell'=K_0b^{p-2}/p$, a sufficiently large~$K_0=K_0(n,p)$ gives~$B_t-\Dp B\ge0$.

At the inner front~$z=1$, both~$B$ and its spatial flux vanish. Extension by zero therefore creates no spatial interface measure. The moving front creates no time atom because~$B$ is continuous across it. Thus~$B$ is a weak supersolution.

For~$\eta>0$, compare on~$B_{R'}\times(t_1+\eta,T)$, where~$T<t_2$ and~$\ell(T)<R'/2$. The energy of~$u$ on this cylinder is local in the original domain. Since~$u\le b=B$ on the lateral boundary in the Sobolev trace sense,~$(u-B)_+$ has zero lateral trace. Subtracting the weak inequalities and testing with~$(u-B)_+$, using time regularization, gives for admissible~$T$
\begin{equation*}
 \begin{split}
 &\frac12\int_{B_{R'}}(u-B)_+^2(x,T)\,dx
 +\int_{t_1+\eta}^T\!\int_{B_{R'}}
   \bigl(A(\nabla u)-A(\nabla B)\bigr)\cdot \nabla (u-B)_+\,dx\,dt\\
 &\hspace{3em}\le\frac12\int_{B_{R'}}
          (u-B)_+^2(x,t_1+\eta)\,dx\,.
 \end{split}
\end{equation*}
Uniform monotonicity makes the diffusion term nonnegative. Since~$B\ge0$, the initial term is bounded by~$\frac12\|u(\cdot,t_1+\eta)\|_{L^2(B_{R'})}^2$, which tends to zero by the right trace. Letting~$\eta\downarrow0$ gives~$u\le B$ almost everywhere. The canonical representative of~$v_b$ gives~$v_b\ge b-B$ pointwise, since~$b-B$ is continuous at positive times. Hence~$v\ge b$ on~$B_{R'/2}$ until~$\ell=R'/2$. With~$\kappa=1/(2^pK_0)$, letting~$R'\uparrow R$ gives~\eqref{eq:persistence}.

\emph{The range~$1<p<2$.} Set
\begin{equation*}
 d=\frac p{p-1},\qquad \beta=\frac p{2-p},\qquad
 \delta=\frac1{2-p},\qquad
 \phi(y)=(1-|y|^d)^{-\beta}\quad(|y|<1).
\end{equation*}
Writing $k=(\beta+1)(p-1)$, differentiation gives $k+1=\beta$ and
\begin{equation}\label{s:plateau-radial}
 \Dp\phi=(\beta d)^{p-1}\phi
 \bigl[n(1-|y|^d)+kd|y|^d\bigr]\le C_\phi\phi.
\end{equation}
Indeed $r^{n-1}(\phi')^{p-1}=(\beta d)^{p-1}r^n(1-r^d)^{-k}$, whose derivative gives the displayed formula. At $r=0$ the vector flux is a continuously differentiable multiple of $y$, so there is no point source at the origin.

The exponent identities $\delta-1=(p-1)\delta$ and $\beta=(p-1)\beta+p$ show that
\begin{equation*}
 B(x,t)=C(t-t_1)^\delta R^{-\beta}\phi(x/R)
\end{equation*}
is a supersolution if $C^{2-p}\delta\ge C_\phi$. Its initial trace is zero in the interior, and it diverges at the lateral boundary. The function $b-v$ is a subsolution, with nonpositive initial trace and upper bound $b$ on the lateral boundary. To compare, first replace $t-t_1$ in $B$ by $t-t_1+\epsilon$, and work in a ball $B_{R'}$ with $R'<R$ so close to $R$ that $B\ge b$ on its boundary throughout the time interval of comparison. The initial inequality holds because $b-v\le0$ and $B\ge0$. Let $R'\uparrow R$ and then $\epsilon\downarrow0$. This gives $b-v\le B$. Since $\phi$ is bounded on $B_{1/2}$, \eqref{s:plateau-loss-eq} follows. Weak comparisons use positive-part testing and the initial trace inequalities~$b-v\le0\le B$.
\end{proof}

\subsection{Lower bounds from capacity}\label{subsec:capacity-lower}\label{subsec:degenerate-lower}\label{sec:s-lower}\label{sec:reception}
We argue by contraposition, using the same idea that replaces a direct Harnack-chain argument in the proof of the Carleson estimate in~\cite[Section~2.4.2]{AGS}. For~$p>2$, weak Harnack and persistence show that large earlier spatial averages would force a large later value. Thus a small later value bounds those averages. For~$p<2$, comparison with homogeneous zero-boundary solutions gives the corresponding bound. Local boundedness and a Caccioppoli estimate control the flux through a source-free annulus. Testing the equation then bounds~$\capc_a(K,Q_r^q)$. Contraposition gives a positive lower bound, which Lemma~\ref{lem:persistence} propagates to a neighborhood of the origin.

\begin{samepage}
\begin{lemma}[A lower bound from capacity]\label{lem:reception}\label{s:reception}
Let~$n\ge2$ and~$p\in(\pc,2)\cup(2,\infty)$. There are constants~$L=L(n,p)>4$ and~$C_*=C_*(n,p)>0$ with the following property; for~$p<2$ take~$L=8$. Let~$r,q>0$, put~$s=q^{2-p}r^p$, and let
\begin{equation*}
 Q=Q_r^q,\qquad
 K\subset\overline{Q_r^{q,-}}
\end{equation*}
be compact. For~$a>0$, set~$\redpot=\redpot^a_{K,Q}$ and~$M=\capc_a(K,Q)$.
\begin{enumerate}[(i)]
\item If~$p>2$, then
\begin{equation}\label{eq:reception}
 M>C_*qr^n
 \quad\Longrightarrow\quad
 \redpot\ge q\quad\text{on }B_{Lr}\times[-s,s/2]\,.
\end{equation}
\item If~$\pc<p<2$, there are constants~$0<c_0\le1$ and~$0<\eta<1/4$, depending only on~$n,p$, such that
\begin{equation}\label{s:reception-eq}
 M>C_*qr^n
 \quad\Longrightarrow\quad
 \redpot\ge c_0q\quad\text{on }B_r\times[-\eta s,\eta s]\,.
\end{equation}
\end{enumerate}
The constants are independent of~$a,K,r,q$.
\end{lemma}
\end{samepage}
\begin{proof}
\emph{Step 1: Normalization and approximation.} Scale~$x=ry$, $t=s\tau$ and divide the potential by~$q$. The equation is unchanged and the Riesz mass is divided by~$qr^n$. It suffices to prove the estimate for~$r=s=q=1$. We write~$a$ for the normalized obstacle height, which remains arbitrary.

For~$p>2$, choose~$L>4$ so large that
\begin{equation}\label{eq:L-choice}
 \kappa 2^{2-p}(4L)^p\ge4\,,
 \qquad \kappa(2L)^p\ge2\,,
\end{equation}
where~$\kappa$ is from Lemma~\ref{lem:persistence}. For~$p<2$, set~$L=8$.

Choose decreasing smooth outer obstacles~$0\le\psi_\ell\le a$, equal to~$a$ on~$K$, with supports in decreasing compact neighborhoods whose intersection is~$K$. Keep these neighborhoods in
\begin{equation*}
 \overline B_{9/8}\times[-33/8,-15/8]\,.
\end{equation*}
Let~$u_{\ell,N}$ be their penalizations from Lemma~\ref{lem:penalty} in~$B_{8L}\times(-5,\infty)$, with zero initial and lateral values, and put
\begin{equation*}
 f_{\ell,N}=N(\psi_\ell-u_{\ell,N})_+,\qquad
 M_{\ell,N}=\iint f_{\ell,N}\,dx\,dt\,.
\end{equation*}
They satisfy~$0\le u_{\ell,N}\le a$, are continuous, and vanish before~$-33/8$. All sources have one fixed compact support. Until Step~5, write~$u=u_{\ell,N}$ and~$m=M_{\ell,N}$.

\emph{Step 2: Bounds for earlier spatial averages.} We derive these bounds from a small-value assumption at a later time. In both cases use the earlier interval~$I=[-9/2,-3/2]$.

For~$p>2$, suppose~$\inf_{B_{2L}}u(\cdot,-1)\le1$. For almost every~$t\in I$, apply~\eqref{eq:weak-harnack} with~$R=L$ and final time~$-1$. The first term is bounded because~$-1-t\ge1/2$. If the average over~$B_L$ is sufficiently large, the second term forces
\begin{equation*}
 u\ge2\quad\text{on }B_{4L}\times(t+\theta/2,t+\theta),
 \qquad t+\theta\le-1\,.
\end{equation*}
Increase the threshold for the average so that~$\theta<-1-t$. Choose an initial time strictly inside this interval. Continuity gives an initial right trace at least~$2$ on~$B_{4L}$. The elapsed time from that initial time to~$-1$ is less than~$7/2$. Lemma~\ref{lem:persistence}, with~$b=2$, $R=4L$ and the first inequality in~\eqref{eq:L-choice}, therefore preserves the lower bound~$u\ge2$ on~$B_{2L}$ through time~$-1$, a contradiction. Thus
\begin{equation}\label{eq:receiver-average}
 \sup_{t\in I}\int_{B_L}u(x,t)\,dx\le C(n,p)\,.
\end{equation}
Continuity extends the bound from admissible Lebesgue times to every~$t\in I$. In particular, the same bound holds over~$B_2$.

For~$\pc<p<2$, let~$c_1$ be the constant in Lemma~\ref{s:BIV-lower}, decreased so that~$c_1\le1$, and suppose
\begin{equation}\label{s:small-receiver}
 \inf_{B_2}u(\cdot,-\eta)<c_1,\qquad 0<\eta<1/4\,.
\end{equation}
For~$t\in I$, start the homogeneous zero-boundary solution in~$B_{64}$ with initial value~$u(\cdot,t)\ind_{B_2}$. It lies below~$u$ by comparison. If its initial mass exceeds~$m_0$ from Lemma~\ref{s:BIV-lower}, its value at elapsed time~$-\eta-t$ is at least~$2c_1$ on~$B_2$. This elapsed time lies in~$[5/4,5]$, contradicting~\eqref{s:small-receiver}. Hence
\begin{equation}\label{s:reception-average}
 \sup_{t\in I}\int_{B_2}u(x,t)\,dx\le m_0\,.
\end{equation}
These bounds are independent of~$\ell,N,a$, and the singular bound is uniform for~$0<\eta<1/4$.

\emph{Step 3: The source mass.} Under either small-value assumption, Step~2 bounds the spatial averages on~$B_2$. The annulus~$\{9/8<|x|<2\}$ contains no source. Cover a neighborhood of~$\{5/4\le|x|\le7/4\}$ by finitely many balls compactly contained in this annulus. The homogeneous upper estimate in Subsection~\ref{subsec:homogeneous-lower}, using~\eqref{eq:receiver-average} when~$p>2$ and Lemma~\ref{s:BIV-upper} with~\eqref{s:reception-average} when~$p<2$, gives a uniform bound for~$u$ in source-free cylinders covering
\begin{equation*}
 \{5/4\le|x|\le7/4\}\times[-33/8,-13/8]\,.
\end{equation*}
Choose the time intervals of these cylinders slightly earlier and within~$(-9/2,-3/2)$. Because their spatial and temporal separations are fixed, the bound depends only on~$n,p$. Testing with~$u\xi^p$ on slightly larger source-free cylinders gives the Caccioppoli estimate
\begin{equation}\label{eq:annular-energy}
 \iint_{\{5/4<|x|<7/4\}\times(-33/8,-13/8)}
       |\nabla u|^p\,dx\,dt\le C(n,p)\,.
\end{equation}
Indeed, the energy is bounded by integrals of~$u^p|\nabla\xi|^p$ and~$u^2|\partial_t(\xi^p)|$, and both are controlled by the local supremum just obtained.

Choose~$\varphi$ equal to one on~$B_{5/4}$ and supported in~$B_{7/4}$. Choose~$\zeta$ equal to one on~$[-33/8,-15/8]$, rising in~$(-9/2,-17/4)$ and falling in~$(-7/4,-13/8)$. Then~$\varphi\zeta=1$ on the source support, so
\begin{equation*}
 m=-\iint u\varphi\zeta'\,dx\,dt
       +\iint|\nabla u|^{p-2}\nabla u\cdot\nabla\varphi\,\zeta\,dx\,dt
   \le C(n,p)\,.
\end{equation*}
Since~$u=0$ and~$\nabla u=0$ for~$t<-33/8$, the term where~$\zeta$ rises and the initial part of the flux integral vanish. Step~2 bounds the other time term, and H\"older's inequality together with~\eqref{eq:annular-energy} bounds the remaining flux term. In particular, for~$p>2$ we have proved
\begin{equation}\label{eq:small-receiver}
 \inf_{B_{2L}}u(\cdot,-1)\le1
 \quad\Longrightarrow\quad m\le C(n,p)\,.
\end{equation}
For~$p<2$, the same mass bound follows from~\eqref{s:small-receiver}, uniformly in~$\eta$.

\emph{Step 4: Propagation of a positive lower bound.} Choose~$C_*$ larger than the mass bound in Step~3. If~$m>C_*$ and~$p>2$, contraposition gives~$u(\cdot,-1)\ge1$ on~$B_{2L}$. Lemma~\ref{lem:persistence}, with~$b=1$, $R=2L$ and the second inequality in~\eqref{eq:L-choice}, gives~$u\ge1$ on~$B_L\times[-1,1/2]$.

If~$m>C_*$ and~$\pc<p<2$, contraposition gives~$u(\cdot,-\eta)\ge c_1$ on~$B_2$. The singular estimate of Lemma~\ref{lem:persistence}, with~$R=2$, gives
\begin{equation*}
 u\ge c_1-C\left(\frac{2\eta}{2^p}\right)^{\frac{1}{2-p}}
 \quad\text{on }B_1\times[-\eta,\eta]\,.
\end{equation*}
Fix~$\eta>0$ so small that the loss is at most~$c_1/2$, and put~$c_0=c_1/2$. The mass threshold was uniform in~$\eta$. Write~$d=1$ and~$G=B_L\times[-1,1/2]$ for~$p>2$, and~$d=c_0$ and~$G=B_1\times[-\eta,\eta]$ for~$p<2$. We have shown that~$m>C_*$ implies~$u\ge d$ on~$G$.

\emph{Step 5: Passage to the compact potential.} For fixed~$\ell$, Lemma~\ref{lem:penalty} gives local uniform convergence~$u_{\ell,N}\to\redpot_{\psi_\ell}$ and weak convergence of the source measures. Since every source has one fixed compact support,
\begin{equation*}
 M_{\ell,N}\longrightarrow
 M_\ell:=\mu_{\redpot_{\psi_\ell}}(Q)\,.
\end{equation*}
If~$M_\ell>C_*$, the same strict inequality holds for all sufficiently large~$N$, so~$\redpot_{\psi_\ell}\ge d$ on~$G$.

The potentials~$\redpot_{\psi_\ell}$ dominate~$\redpot=\redpot^a_{K,Q}$. If~$v$ is a nonnegative superparabolic majorant of~$a\ind_K$, lower semicontinuity makes~$\{v>a-\epsilon\}$ an open neighborhood of~$K$ for every~$\epsilon>0$. For all sufficiently large~$\ell$, we have~$\psi_\ell\le v+\epsilon$, hence~$\redpot_{\psi_\ell}\le v+\epsilon$. Taking the infimum over~$v$, regularizing, and sending~$\epsilon\downarrow0$ identifies the lower semicontinuous regularization of the decreasing limit of~$\redpot_{\psi_\ell}$ with~$\redpot$.

The bounded-supersolution convergence results in Subsection~\ref{subsec:capacity-foundations} give weak convergence of the Riesz measures. Their common compact support gives
\begin{equation*}
 M_\ell\longrightarrow\mu_{\redpot}(K)=M\,.
\end{equation*}
The region~$G$ is separated from that support. Interior compactness of bounded homogeneous solutions gives local uniform convergence there. If~$M>C_*$, then~$M_\ell>C_*$ for all sufficiently large~$\ell$, and the lower bound passes to~$\redpot$. No pointwise obstacle value on a time slice is used. Scaling back proves~\eqref{eq:reception} and~\eqref{s:reception-eq}.
\end{proof}

\subsection{Selection of nested cylinders}\label{sec:sufficiency}
The intrinsic radii need not decrease with the indices. If~$p>2$, the increments corresponding to earlier radii that are too small decay geometrically, so we select the indices backwards. If~$p<2$, the increments corresponding to later radii that are too large decay geometrically, so we select the indices forwards. In either case, the selected increments retain a fixed fraction of the sum.

\begin{lemma}[Selection of nested cylinders]\label{lem:selection}\label{s:forward-selection}
Let~$p>1$, $p\ne2$, and suppose a finite family satisfies~$0<q_j\le a$,~$s_j=2^{-j}$, and~$q_j^{2-p}r_j^p=s_j$. Given~$\Lambda>1$, there is an integer~$m=m(p,\Lambda)$ such that one can select~$j_1<\cdots<j_N$ with
\begin{equation}\label{eq:selection-nesting}
 j_{k+1}-j_k\ge m\,,
 \qquad r_{j_{k+1}}\le\Lambda^{-1}r_{j_k}\,,
\end{equation}
and
\begin{equation}\label{eq:selection-gain}
 \sum_{k=1}^Nq_{j_k}\ge\frac1{2m}\sum_jq_j\,.
\end{equation}
The integer~$m$ may be increased in advance by any prescribed amount.
\end{lemma}
\begin{proof}
If the index set is empty, take the empty selection. Assume it is nonempty.

\emph{The range~$p>2$.} Starting from the last index in a residue class, we choose earlier indices and assign each discarded term to the next selected one. Choose~$m$ large enough that
\begin{equation*}
 \frac{\Lambda^{\frac{p}{p-2}}}{2^{\frac{m}{p-2}}-1}\le1\,.
\end{equation*}
Choose a residue class modulo~$m$ for which the sum is at least~$1/m$ of the total. Select its last index~$j$. Next select the largest earlier index~$i$ satisfying~$r_i\ge\Lambda r_j$, and repeat from~$i$. Stop when no earlier index satisfies this inequality.

Assign each discarded index~$i$ to the least selected index~$j>i$. By the rule of selection,~$r_i<\Lambda r_j$. Therefore
\begin{equation*}
 q_i^{p-2}2^{-i}<\Lambda^pq_j^{p-2}2^{-j}\,,
 \qquad
 q_i<\Lambda^{\frac{p}{p-2}}2^{-\frac{j-i}{p-2}}q_j\,.
\end{equation*}
The differences~$j-i$ are positive multiples of~$m$. Summing the geometric series and using the choice of~$m$, we find that the discarded terms assigned to~$j$ have sum at most~$q_j$. Hence the selected terms account for at least half the sum in this residue class. This proves~\eqref{eq:selection-gain}; the construction also gives~\eqref{eq:selection-nesting}.

\emph{The range~$1<p<2$.} Choose $m$ large enough that
\begin{equation*}
 \frac{\Lambda^{\frac{p}{2-p}}}{2^{\frac{m}{2-p}}-1}\le1.
\end{equation*}
One residue class modulo $m$ carries at least $1/m$ of the sum. In that class start with its first index $i$. Select the next index $j$ whose radius is at most $r_i/\Lambda$, and repeat forwards. Assign the skipped indices to the preceding selected one. For a skipped $k$ assigned to $i$,
\begin{equation*}
 r_k>r_i/\Lambda
 \quad\Longrightarrow\quad
 q_k<\Lambda^{\frac{p}{2-p}}2^{-\frac{k-i}{2-p}}q_i.
\end{equation*}
Since the positive differences $k-i$ are multiples of $m$, the sum of $q_k$ over the skipped indices assigned to $i$ is at most $q_i$. The sum of~$q_j$ over the selected indices is therefore at least half the sum over the residue class. This proves both conclusions.
\end{proof}

For~$p>2$, take admissible increments~$q_j$ in~\eqref{eq:omega}, fix~$\Lambda=8$, and choose~$m\ge3$ in Lemma~\ref{lem:selection}. Write~$\widetilde Q_j=Q_{j,q_j}$ and~$G_j=G_{j,q_j}$ for the regions in~\eqref{eq:test-cylinder} and~\eqref{eq:receiver}. Successive selected indices satisfy
\begin{equation}\label{eq:condenser-nesting}
 \widetilde Q_{j_{k+1}}\subset G_{j_k}\,.
\end{equation}
Indeed,~$8Lr_{j_{k+1}}\le Lr_{j_k}$,~$5s_{j_{k+1}}\le s_{j_k}$, and~$s_{j_{k+1}}\le s_{j_k}/2$. Thus~$\widetilde Q_{j_{k+1}}\subset G_{j_k}$, where the lower bound for the preceding potential holds. We can therefore compare with the next potential on~$\widetilde Q_{j_{k+1}}$. Figure~\ref{fig:nested-reception} shows these inclusions.

\begin{figure}[H]
\centering
\begin{tikzpicture}[x=1cm,y=1cm,font=\footnotesize,
  condenser/.style={draw=black!60,densely dashed,line width=.65pt},
  receiver/.style={draw=black!70,line width=.65pt},
  guide/.style={draw=black!55,line width=.4pt}]
  % One smooth boundary separates the connected domain and complement.
  \path[use as bounding box] (-3.25,-4.85) rectangle (7.55,2.5);
  \fill[black!8] (-2.7,-4.7)
    --({.16*(-4.7)+.013*(-4.7)*(-4.7)},-4.7)
    -- plot[domain=-4.7:2,variable=\t,samples=65]
      ({.16*\t+.013*\t*\t},\t)
    --(-2.7,2)--cycle;
  \node at (-1.75,-4.27) {$E$};
  \node at (1.58,-4.27) {$\Theta$};
  % The compact obstacles are portions of this same complement.
  \begin{scope}
    \clip (-2.7,-4.7)
      --({.16*(-4.7)+.013*(-4.7)*(-4.7)},-4.7)
      -- plot[domain=-4.7:2,variable=\t,samples=65]
        ({.16*\t+.013*\t*\t},\t)
      --(-2.7,2)--cycle;
    \fill[black!35]
      (-.85,-3.9) rectangle (.85,-2.8);
    \fill[black!35]
      (-.5,-1.36) rectangle (.5,-.94);
    \fill[black!35]
      (-.26,-.42) rectangle (.26,-.23);
  \end{scope}
  \draw[black!70,line width=.5pt]
    (-.85,-3.9)--({.16*(-3.9)+.013*(-3.9)*(-3.9)},-3.9)
    -- plot[domain=-3.9:-2.8,variable=\t,samples=20]
      ({.16*\t+.013*\t*\t},\t)
    --(-.85,-2.8)--cycle;
  \draw[black!70,line width=.5pt]
    (-.5,-1.36)--({.16*(-1.36)+.013*(-1.36)*(-1.36)},-1.36)
    -- plot[domain=-1.36:-.94,variable=\t,samples=12]
      ({.16*\t+.013*\t*\t},\t)
    --(-.5,-.94)--cycle;
  \draw[black!70,line width=.5pt]
    (-.26,-.42)--({.16*(-.42)+.013*(-.42)*(-.42)},-.42)
    -- plot[domain=-.42:-.23,variable=\t,samples=10]
      ({.16*\t+.013*\t*\t},\t)
    --(-.26,-.23)--cycle;
  % The selected scales are enlarged successively to display the inclusions.
  \draw[receiver] (-2.3,-1.9) rectangle (2.3,1.25);
  \draw[receiver] (-1.55,-.72) rectangle (1.55,.62);
  \draw[receiver] (-.7,-.15) rectangle (.7,.15);
  \draw[condenser] (-2.7,-4.7) rectangle (2.7,2);
  \draw[condenser] (-2,-1.65) rectangle (2,.98);
  \draw[condenser] (-1.2,-.53) rectangle (1.2,.42);
  \draw[black,line width=.9pt]
    plot[domain=-4.7:2,variable=\t,samples=65]
      ({.16*\t+.013*\t*\t},\t);
  \node[above right,inner sep=2pt] at (.372,2)
    {$\partial\Theta$};
  \draw[->,line width=.5pt] (-3.02,-4.7)--(-3.02,2.3)
    node[above] {$t$};
  \draw[dotted,black!55] (-3.02,0)--(2.76,0);
  \node[left] at (-3.02,0) {$0$};
  \fill (0,0) circle (1.7pt);
  \node[anchor=north west,inner sep=3pt] at (-2.7,2)
    {$\widetilde Q_{j_k}$};
  \node[anchor=south west,inner sep=3pt] at (-2,-1.65)
    {$\widetilde Q_{j_{k+1}}$};
  \draw[guide] (-.36,-3.34)--(.68,-3.34);
  \node[right,inner sep=2pt] at (.68,-3.34) {$K_{j_k,q_{j_k}}$};
  \draw[guide] (-.17,-1.13)--(.48,-1.13);
  \node[right,inner sep=2pt] at (.48,-1.13)
    {$K_{j_{k+1},q_{j_{k+1}}}$};
  \draw[guide] (-.05,-.33)--(2.55,-.96)--(3.12,-.96);
  \node[anchor=west,inner sep=2pt] at (3.12,-.96)
    {$K_{j_{k+2},q_{j_{k+2}}}$};
  % Callouts keep the lower-bound labels legible at the smallest scale.
  \draw[guide] (2.3,1.13)--(3.12,1.5);
  \node[anchor=west,align=left,inner sep=2pt] at (3.12,1.5)
    {$G_{j_k}$:\ $v\ge b$};
  \draw[guide] (1.55,.53)--(2.65,.66)--(3.12,.66);
  \node[anchor=west,align=left,inner sep=2pt] at (3.12,.66)
    {$G_{j_{k+1}}$:\ $v\ge b+q_{j_{k+1}}$};
  \draw[guide] (.7,.13)--(2.65,-.18)--(3.12,-.18);
  \node[anchor=west,align=left,inner sep=2pt] at (3.12,-.18)
    {$G_{j_{k+2}}$:\\[-1pt]
     $v\ge b+q_{j_{k+1}}+q_{j_{k+2}}$};
  \node[anchor=north west,align=left,text width=4.15cm,inner sep=0pt]
    at (3.14,-2.0)
    {Dashed, from outside in:\\[2pt]
     $\widetilde Q_{j_k}$, $\widetilde Q_{j_{k+1}}$, $\widetilde Q_{j_{k+2}}$.\\[6pt]
     Solid rectangles: regions~$G_j$.\\[3pt]
     Dark gray: tested pieces of $E$.};
\end{tikzpicture}
\caption{Selected cylinders for~$p>2$ near a smooth boundary, with smaller scales enlarged. The dashed cylinders~$\widetilde Q_j$ and solid regions~$G_j$ cross~$\partial\Theta$; the dark sets are the compact test sets~$K_{j,q_j}\subset E$. Selection gives~$\widetilde Q_{j_{k+1}}\subset G_{j_k}$. The increments~$q_j$ are admissible, and~$v\ge0$ is a superparabolic majorant with~$v\ge H$ on the selected obstacles. If~$v\ge b$ on~$G_{j_k}$ and~$H-b\ge a$, Lemmas~\ref{lem:local-majorant-comparison} and~\ref{lem:reception} give~$v\ge b+q_{j_{k+1}}$ on~$G_{j_{k+1}}$. Proposition~\ref{prop:saturation} adds these increments until their sum exceeds~$H-a$.}
\label{fig:nested-reception}
\end{figure}
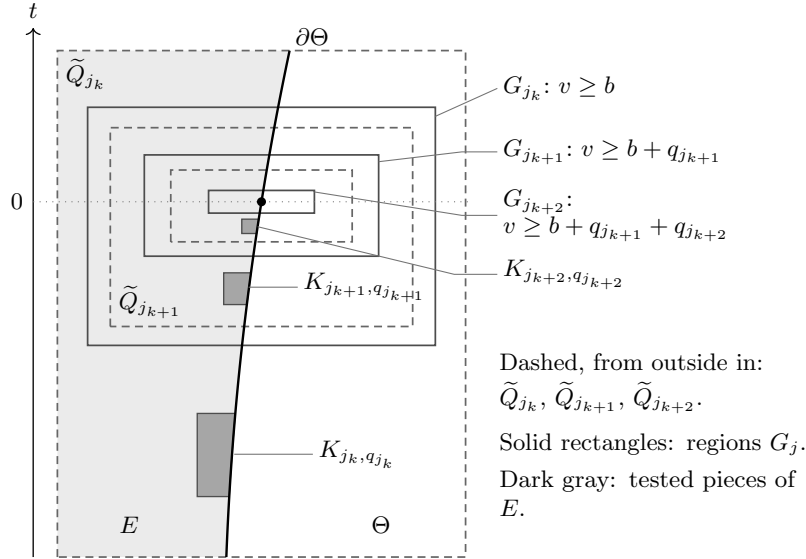

\begin{remark}[Summability of $q_j$ for $r_j\ge R$]\label{s:singular-locality}
For~$p<2$, the bound~$q_j\le a$ allows large intrinsic radii. Nevertheless, for every~$R>0$,
\begin{equation*}
 r_j\ge R\quad\Longrightarrow\quad
 q_j\le\left(\frac{s_j}{R^p}\right)^{\frac{1}{2-p}}.
\end{equation*}
These upper bounds form a summable geometric series. If~$\sum_jq_j=\infty$, the sum remains divergent after discarding all reference cylinders too large for any prescribed local spatial neighborhood. Also,~$s_j\to0$ independently of the radii. These two facts allow the following finite selections to be made in any prescribed space--time neighborhood.
\end{remark}

\subsection{Potentials of finite unions}\label{subsec:finite-reductions}
Lemma~\ref{lem:reception} converts an admissible capacity bound into a lower bound on a later cylinder. If~$\sum_j\omega_a(j)=\infty$, Lemma~\ref{lem:selection} gives finite selections with arbitrarily large sums of these lower bounds, with each reference cylinder contained in the preceding region where the lower bound holds. If~$v$ is a superparabolic majorant at height~$H>a$ and~$v\ge b$ on the next reference cylinder, with~$b\le H-a$, then~$v-b$ majorizes the corresponding obstacle at height~$a$. Lemma~\ref{lem:local-majorant-comparison} therefore adds the next lower bound. The following proposition gives the finite unions used in Proposition~\ref{prop:sufficiency}.

\begin{proposition}[Lower bound for potentials of finite unions]\label{prop:saturation}\label{s:singular-saturation}
Suppose~$\sum_j\omega_a(j)=\infty$ for a fixed height~$a>0$. Let~$H>a$ and let~$Q^*$ be any smooth cylinder containing the origin in its interior. In every sufficiently late tail there is a finite union~$K$ of sets~$K_{j,q}\subset E$ from~\eqref{eq:test-set}, with~$K\Subset Q^*\cap\{t<0\}$, such that
\begin{equation}\label{eq:saturation}
 \redpot_{K,Q^*}^{H}\ge H-a
 \quad\text{in an open neighborhood of the origin}\,.
\end{equation}
The origin is separated from~$K$, so this potential is continuous there.
\end{proposition}
\begin{proof}
Choose an admissible~$q_j>\omega_a(j)/2$ whenever~$\omega_a(j)>0$, and put~$\widetilde K_j=K_{j,q_j}$, $\widetilde Q_j=Q_{j,q_j}$ and~$r_j=r_j(q_j)$. These choices are from the sets defining the suprema; no interval property of admissibility is assumed. Set
\begin{equation*}
 (d_j,G_j)=
 \begin{cases}
 (q_j,\ B_{Lr_j}\times[-s_j,s_j/2]),&p>2,\\
 (c_0q_j,\ B_{r_j}\times[-\eta s_j,\eta s_j]),&\pc<p<2.
 \end{cases}
\end{equation*}
Lemma~\ref{lem:reception} gives
\begin{equation*}
 \redpot^a_{\widetilde K_j,\widetilde Q_j}\ge d_j\quad\text{on }G_j\,.
\end{equation*}
In both ranges,~$\sum_jd_j=\infty$ and~$0<d_j\le a$. For~$p>2$, all reference cylinders lie in~$Q^*$ for sufficiently large~$j$, since~$q_j\le a$. For~$p<2$, Remark~\ref{s:singular-locality} lets us discard the reference cylinders not contained in~$Q^*$ while preserving divergence.

Apply Lemma~\ref{lem:selection} with~$\Lambda=8$ and~$m\ge3$ when~$p>2$, and with~$\Lambda=128$ and~$m$ so large that~$5\cdot2^{-m}<\eta$ when~$p<2$. In the first case,~\eqref{eq:condenser-nesting} gives~$\widetilde Q_{j_{k+1}}\subset G_{j_k}$. In the second case,
\begin{equation}\label{s:strict-nesting}
 \overline{\widetilde Q_{j_{k+1}}}
 \subset B_{r_{j_k}}\times(-\eta s_{j_k},\eta s_{j_k})
 \subset G_{j_k}\,.
\end{equation}
Thus every sufficiently late tail has finite selections with arbitrarily large~$\sum d_j$ and with~$\widetilde Q_{j_{k+1}}\subset G_{j_k}$.

Take such a selection and stop when its partial sum first exceeds~$H-a$. Let~$K$ be the union of the selected compact obstacles. It is a compact subset of~$E\cap Q^*\cap\{t<0\}$.

Consider the nonnegative superparabolic majorants of~$H\ind_K$ in the continuation cylinder defining~$\redpot^H_{K,Q^*}$. Replacing each majorant~$v$ by~$\min\{v,H\}$ preserves admissibility and superparabolicity, so we may restrict to~$0\le v\le H$. Restrict such a majorant to~$Q^*$. Initially~$v\ge0$. Suppose the preceding comparisons give~$v\ge b$ on the next reference cylinder, where~$b\le H-a$. Then~$v-b$ is nonnegative and superparabolic there. On its compact obstacle,~$v-b\ge H-b\ge a$. Lemma~\ref{lem:local-majorant-comparison} therefore gives
\begin{equation*}
 v\ge b+\redpot^a_{\widetilde K_j,\widetilde Q_j}\ge b+d_j
 \quad\text{on }G_j\,.
\end{equation*}
For the first cylinder use~$b=0$; the nesting permits each subsequent comparison. At the last step,~$b+d_j\le H$ because~$b\le H-a$ and~$d_j\le a$. Every admissible majorant is thus bounded below by the same sum exceeding~$H-a$ on the last~$G_j$, whose interior contains the origin. Taking the infimum over the majorants and then regularizing preserves this lower bound on that open neighborhood. This proves~\eqref{eq:saturation}.
\end{proof}

\section{Upper bounds for capacitary potentials}
\label{sec:upper}

We bound the potential of a finite union of complementary portions in successive cylinders by the sum of their Wiener terms, uniformly in the number of cylinders.

We keep continuous obstacle neighborhoods fixed throughout the estimates. For a sufficiently small convergent tail, the increasing limit is a bounded superparabolic function that equals its prescribed height near the local complement in the past and has small interior values approaching the origin. This is the irregularity barrier used in Subsubsection~\ref{subsec:intrinsic-necessity}.

\subsection{The cumulative upper estimate}\label{sec:discussion}
Let~$n\ge2$, $p\in(\pc,2)\cup(2,\infty)$, and let~$\Theta\subset\R^{n+1}$ be bounded and open with~$0\in\partial\Theta$. Put~$E=\R^{n+1}\setminus\Theta$ and~$s_j=2^{-j}$. For~$R_0>0$, define
\begin{align}
 E_j&=E\cap\bigl(\overline B_{R_0}\times[-4s_j,-2s_j]\bigr),\label{s:finite-complement}\\
 E_{N,m}&=\bigcup_{j=N}^m E_j\,.\label{eq:finite-tail}
\end{align}
A potential without a displayed reference is the whole-space potential, obtained by exhaustion with zero values before the first obstacle time. For~$p<2$, we prove the estimates in bounded reference balls with constants independent of the ball and then pass to the exhaustion limit.

\begin{samepage}
\begin{theorem}[Cumulative upper estimate for capacitary potentials]\label{thm:upper-main}\label{s:upper-theorem}
There is~$C>0$ and, for~$p>2$, there is~$A>1$, depending only on~$n,p$, such that, for all~$a,R_0>0$ and finite~$N\le m$,
\begin{align}
 \redpot_{E_{N,m}}^{a/A}(0,0)
 &\le \min\left\{\frac aA,\ C\sum_{j=N}^m\omega_a(j)\right\},
 &&p>2,\label{eq:cumulative-upper}\\
 \redpot_{E_{N,m}}^a(0,0)
 &\le \min\left\{a,\ C\sum_{j=N}^m\omega_a(j)\right\},
 &&\pc<p<2\,.\label{s:upper-theorem-eq}
\end{align}
For~$\pc<p<2$, estimate~\eqref{s:upper-theorem-eq} also holds for~$\redpot^a_{E_{N,m},Q}(0,0)$ in every smooth bounded reference cylinder~$Q$ with~$E_{N,m}\cup\{(0,0)\}\Subset Q$, with zero initial and lateral data.
\end{theorem}
\end{samepage}

For~$p>2$, the nontrivial case has~$\omega_a(j)<a$ at every retained index. Choose~$\omega_a(j)<b_j<a$ and split~$E_j=C_j\cup F_j$ at the corresponding intrinsic radius, as in~\eqref{eq:core-remainder}. Lemma~\ref{lem:core-remainder} bounds the capacity of~$C_j$ and gives a capacity bound for~$F_j$ at every spatial radius. Lemma~\ref{lem:growth-neighborhood} preserves the outer capacity bound~\eqref{eq:remainder-capacity}, up to a fixed factor, when~$F_j$ is enlarged. Reducing the height to~$h=a/A$ makes the coefficient in the source growth bound small enough for Lemma~\ref{lem:forced-waiting}. Lemma~\ref{lem:remainder-waiting} thus gives a remainder potential that vanishes at~$(0,0)$. Theorem~\ref{thm:removal} then bounds the cost of adding the obstacles near~$C_j$ by~$C\sum_j b_j$, independently of the obstacles already present.

For~$\pc<p<2$, Lemma~\ref{s:regularize-sequence} first replaces the increments by a summable majorant~$b_j$ with controlled consecutive ratios. Its intrinsic radii~$\widetilde r_j$ decrease geometrically. We split~$E_j=C_j\cup F_j$ at the larger radius~$2\widetilde r_{j-1}$. The neighborhoods of~$F_j$ stay outside the iteration balls during every overlapping time interval. Lemma~\ref{s:local-budget} therefore converts the capacity bounds for the neighborhoods of~$C_j$ into local source bounds of size~$Cb_j\widetilde r_j^n$. Lemma~\ref{s:iteration} applies the one-step estimate of Lemma~\ref{s:excess}: the source mass on one interval bounds the initial excess for the next, while the level increases sum to at most~$C\sum_j b_j$. This gives the pointwise bound without reducing the obstacle height.

In both ranges, the neighborhoods are fixed before finite truncations, and the estimates are uniform in the number of cylinders. We combine the estimates in Subsection~\ref{subsec:upper-proof} to prove Theorem~\ref{thm:upper-main}. Corollary~\ref{cor:singular-tail} gives the fixed neighborhoods and small interior values used in the singular necessity argument.

Lemma~\ref{lem:reception} has already fixed the test constants~$L,C_*$ and, when~$p<2$, also~$c_0,\eta$. The upper estimates then determine~$A$ for~$p>2$ and the smallness threshold for the singular iteration. These choices depend only on~$n,p$, before the height, tail, and finite truncation are chosen; Remark~\ref{s:singular-constants} gives their order.

\subsection{Comparison estimates}\label{subsec:upper-comparison}
For~$p>2$, uniform monotonicity gives smoothing for the difference of two potentials, independently of the retained obstacle. For~$\pc<p<2$, we instead compare a solution with a homogeneous evolution above a prescribed level and bound the remaining excess by the source mass.

\subsubsection{Zero-trace interpolation for~$p>2$}\label{subsec:degenerate-upper}
Throughout this subsubsection,~$p>2$.

On~$\R^n$, we use bounded functions in~$L^p(I;W^{1,p}(\R^n))\cap L^\infty(I;L^2(\R^n))$ for every finite time interval~$I$. Here~$W^{1,p}_0(\R^n)=W^{1,p}(\R^n)$. Time regularization and approximation justify global tests with powers, approximations of the sign function, and truncations. For obstacle potentials and penalizations, the finite-propagation argument in Section~\ref{sec:potential} also permits bounded reference domains over each finite time interval.

For $2<p<n$, the spatial interpolation inequalities used below follow from the usual Sobolev inequality and H\"older's inequality. The following zero-trace version supplies the same powers when $p\ge n$. It is applied both to the difference iteration and to the truncated flux estimate.

\begin{samepage}
\begin{lemma}\label{lem:interpolation}
Let $p\ge n\ge2$. For $1\le\gamma\le p$ and $z\in C_c^\infty(\R^n)$,
\begin{equation}\label{eq:zero-trace-interpolation}
 \int_{\R^n}|z|^{p(1+\frac{\gamma}{n})}
 \le C(n,p)
 \left(\int_{\R^n}|z|^\gamma\right)^{\frac{p}{n}}
 \int_{\R^n}|\nabla z|^p.
\end{equation}
The constant is uniform for $\gamma\in[1,p]$. The inequality extends by approximation to zero-trace functions whenever the quantities on the right are finite, and its constant is independent of the spatial reference domain.
\end{lemma}
\end{samepage}
The proof is in Appendix~\ref{app:interpolation-proof}.

\begin{remark}[The case $p=n$]\label{rem:interpolation-endpoint}
At $p=n$, inequality~\eqref{eq:zero-trace-interpolation} supplies the finite powers used in the estimates below. It does not give an $L^\infty$ embedding for $W^{1,n}$.
\end{remark}

\subsubsection{Differences of obstacle potentials for~$p>2$}
\label{sec:removal}

Throughout this subsubsection,~$p>2$. We estimate the difference of two potentials on the entire spatial reference domain. Its flux depends on both potentials, but uniform monotonicity bounds the diffusion term below by a positive multiple of the~$p$th power of the gradient of their difference. The zero lateral trace permits global tests without spatial cutoff terms. The resulting estimate is independent of the retained potential, even though the difference need not be~$p$-parabolic.

\begin{lemma}\label{lem:difference-smoothing}
Let~$D$ be a smooth bounded domain or~$\R^n$, and let~$t_*<T_1$. Suppose
\begin{equation*}
 u,v\in L^p(t_*,T_1;W^{1,p}(D))\cap L^\infty(t_*,T_1;L^2(D)),
\end{equation*}
and~$w=v-u\ge0$ is bounded and belongs to~$L^p(t_*,T_1;W^{1,p}_0(D))$. Assume, in the weak sense,
\begin{equation*}
 w_t-\div\bigl(A(\nabla v)-A(\nabla u)\bigr)\le0\,,
\end{equation*}
and~$\esssup_{t_*<t<T_1}\int_D w(x,t)\,dx\le M$. Then, for almost every~$T\in(t_*,T_1)$,
\begin{equation}\label{eq:difference-smoothing}
 \|w(\cdot,T)\|_\infty
 \le C M^{\frac{p}{\lambda}}(T-t_*)^{-\frac{n}{\lambda}}\,.
\end{equation}
If~$u,v$ have finite representatives satisfying~\eqref{eq:past-representative}, then~\eqref{eq:difference-smoothing} holds pointwise for~$v-u$ at every~$(x,T)\in D\times(t_*,T_1)$.
\end{lemma}
\begin{proof}
We first prove an~$L^2$ bound in terms of the mass, then use iteration to obtain the~$L^\infty$ bound. If~$M=0$, then~$w=0$, so assume~$M>0$. The diffusion term is controlled by the uniform monotonicity inequality
\begin{equation}\label{eq:monotonicity}
 \bigl(A(\xi)-A(\eta)\bigr)\cdot(\xi-\eta)
 \ge 2^{2-p}|\xi-\eta|^p\,.
\end{equation}
Since~$\nabla w=\nabla v-\nabla u$, testing the difference inequality with~$w^{\ell-1}$ and applying~\eqref{eq:monotonicity} gives, for~$\ell\ge2$, with time regularization understood,
\begin{equation}\label{eq:difference-power}
 \frac{d}{dt}\int_D w^\ell
 +c_p\frac{\ell(\ell-1)}{(\ell+p-2)^p}
       \int_D\left|\nabla w^{\frac{\ell+p-2}{p}}\right|^p\le0\,.
\end{equation}
The factor~$p^p$ is included in~$c_p>0$.

Let~$Y(t)=\int_D w^2$. Extend~$w$ by zero if~$D$ is bounded. Interpolation using Lemma~\ref{lem:interpolation} with $\gamma=1$ when $p\ge n$, and the usual Sobolev interpolation when $p<n$, gives
\begin{equation*}
 \|w\|_2\le C\|\nabla w\|_p^\vartheta\|w\|_1^{1-\vartheta}\,,
 \qquad \vartheta=\frac{np}{2(np+p-n)}\,.
\end{equation*}
Use~$\|w\|_1\le M$ in this inequality to bound~$\|\nabla w\|_p$ from below in terms of~$Y$. The case~$\ell=2$ of~\eqref{eq:difference-power} then yields
\begin{equation*}
 Y'+cM^{-\frac{\lambda+p}{n}}Y^{1+\frac{\lambda}{n}}\le0\,.
\end{equation*}
Integrating this inequality and discarding the finite initial value gives
\begin{equation}\label{eq:difference-L2}
 Y(t)\le C M^{1+\frac{p}{\lambda}}(t-t_*)^{-\frac{n}{\lambda}}\,.
\end{equation}

We next prove the~$L^2$-to-$L^\infty$ estimate
\begin{equation}\label{eq:L2-Linfty}
 \|w(\cdot,t)\|_\infty
 \le C(t-s)^{-\frac{n}{\lambda+p}}
                  \|w(\cdot,s)\|_2^{\frac{2p}{\lambda+p}},\qquad t>s\,.
\end{equation}
Put~$\tau=t-s$ and~$L=\|w(s)\|_2$. If~$L=0$, the energy inequality gives~$w=0$ at later times. Otherwise choose~$a,r>0$ so that~$\tau=a^{2-p}r^p$ and~$L=ar^{\frac{n}{2}}$. Define, on~$r^{-1}D\times(0,1)$,
\begin{equation*}
 \widetilde w(y,\sigma)=a^{-1}w(ry,s+\tau\sigma)\,.
\end{equation*}
Scale~$u,v$ in the same way. Homogeneity of~$A$ preserves the difference inequality. The normalized initial~$L^2$ norm is one. We prove a uniform bound on the normalized unit interval and write~$w$ for~$\widetilde w$ during the iteration. Set
\begin{equation*}
 \kappa=1+\frac pn,\qquad \ell_0=2\,,
 \qquad \ell_{k+1}=\kappa\ell_k+p-2\,.
\end{equation*}
Choose increasing time levels tending to the midpoint of the unit interval, with successive distances comparable to~$2^{-k}$. At each step, restrict to the interval between the new time level and the final time. Apply~\eqref{eq:difference-power} with the associated time cutoff. On the smaller interval, it bounds both~$\sup\int w^{\ell_k}$ and~$\iint|\nabla w^{\frac{\ell_k+p-2}{p}}|^p\,dx\,dt$ by~$C2^k\ell_k^{p-2}$ times the space-time integral of~$w^{\ell_k}$ over the preceding interval. We then raise the integrability exponent by applying the zero-trace interpolation inequality to~$w^{\frac{\ell_k+p-2}{p}}$. For~$p\ge n$, use Lemma~\ref{lem:interpolation} with~$\gamma=p\ell_k/(\ell_k+p-2)\in[1,p]$; for~$p<n$, use Sobolev interpolation. This gives
\begin{equation*}
 \iint w^{\kappa\ell_k+p-2}\,dx\,dt
 \le C\left(\sup_t\int w^{\ell_k}\right)^{\frac{p}{n}}
             \iint\left|\nabla w^{\frac{\ell_k+p-2}{p}}\right|^p\,dx\,dt\,.
\end{equation*}
Thus, if~$E_k$ denotes the space-time integral of~$w^{\ell_k}$ at the~$k$th step, the choice of~$\ell_{k+1}$ gives
\begin{equation*}
 E_{k+1}\le
 \bigl(C2^k\ell_k^{p-2}\bigr)^\kappa E_k^\kappa\,.
\end{equation*}
Now
\begin{equation*}
 \ell_k=\kappa^k\left(2+\frac{n(p-2)}p\right)
                         -\frac{n(p-2)}p\,.
\end{equation*}
The initial integral satisfies~$E_0\le1$ by the~$\ell=2$ energy inequality on the normalized unit interval. The product obtained after taking the~$\ell_k$th roots converges, since~$\sum_k(k+\log\ell_k)/\ell_k<\infty$. The normalized supremum is therefore bounded by~$C(n,p)$. Scaling back gives~$\|w(t)\|_\infty\le Ca$, where~$a=\tau^{-\frac{n}{\lambda+p}}L^{\frac{2p}{\lambda+p}}$. This is~\eqref{eq:L2-Linfty}. The constants are independent of~$D$ because every Sobolev inequality was applied after zero extension.

To complete the proof, use~\eqref{eq:L2-Linfty} between the midpoint of~$(t_*,T)$ and~$T$, and bound the~$L^2$ norm at that midpoint by~\eqref{eq:difference-L2}. The exponent of~$M$ is~$(1+p/\lambda)p/(\lambda+p)=p/\lambda$, and the exponent of~$T-t_*$ is~$-n/(\lambda+p)-np/[\lambda(\lambda+p)]=-n/\lambda$. This proves~\eqref{eq:difference-smoothing}. For the pointwise assertion, write~\eqref{eq:difference-smoothing} as~$v\le u+CM^{p/\lambda}(t-t_*)^{-n/\lambda}$ almost everywhere and take the past essential lower limit of each side separately. The last term is continuous for~$t>t_*$.
\end{proof}

\begin{theorem}[Difference estimate for obstacle potentials]\label{thm:removal}
Let~$E,F$ be compact in a common reference problem with zero initial and lateral data, let~$h>0$, and suppose~$F\subset\{t\le t_*\}$. Then, for~$T>t_*$,
\begin{equation}\label{eq:removal}
 0\le \redpot_{E\cup F}^{h}(x,T)-\redpot_E^h(x,T)
 \le \min\left\{h,\ C\capc_h(F)^{\frac{p}{\lambda}}
                                     (T-t_*)^{-\frac{n}{\lambda}}\right\}\,.
\end{equation}
Also, for almost every~$t$,
\begin{equation}\label{eq:removal-L1}
 \int\bigl(\redpot_{E\cup F}^h-\redpot_E^h\bigr)(x,t)\,dx\le\capc_h(F)\,.
\end{equation}
The same estimates hold for continuous obstacles bounded by~$h$: the capacity of a compact support of the added obstacle can be used on the right. The constant is independent of the retained obstacle, which may continue to act after~$t_*$.
\end{theorem}
\begin{proof}
The retained obstacle may remain active after the added obstacle ends. After that time, the penalization source for the union is no larger than the source for the retained obstacle. Hence the nonnegative difference of the penalized potentials satisfies an unforced subsolution inequality. We bound its~$L^1$ norm by the capacity of the added obstacle and apply Lemma~\ref{lem:difference-smoothing}.

Take continuous compactly supported obstacles~$\psi,\phi\in[0,h]$. Let~$u_N,z_N,v_N$ be their penalized solutions for~$\psi,\phi$, and~$\max\{\psi,\phi\}$, respectively. Comparison gives~$u_N,z_N\le v_N\le h$. Denote the nonnegative sources by~$r_E,r_F,r_{E\cup F}$. Since~$u_N,z_N\le v_N$, we have pointwise
\begin{equation*}
 r_{E\cup F} =N\max\{(\psi-v_N)_+,(\phi-v_N)_+\} \le N(\psi-u_N)_++N(\phi-z_N)_+=r_E+r_F\,.
\end{equation*}
Outside~$\supp\phi$, we have the stronger inequality~$r_{E\cup F}-r_E\le0$. Thus~$w_N=v_N-u_N\ge0$ satisfies
\begin{equation*}
 (w_N)_t-\div\bigl(A(\nabla v_N)-A(\nabla u_N)\bigr)\le r_F\,.
\end{equation*}
After the obstacle~$\phi$ ends,~$r_F=0$, so~$w_N$ satisfies the unforced inequality in Lemma~\ref{lem:difference-smoothing}.

Testing with an approximation of the positive sign function and using~\eqref{eq:monotonicity} proves
\begin{equation*}
 \int w_N(x,t)\,dx\le\iint r_F\,dx\,dt\,.
\end{equation*}
We have~$0\le z_N\le h$, and~$r_F$ is supported in~$\supp\phi$. Hence~$\iint r_F\,dx\,dt\le\capc_h(\supp\phi)$ by the measure-based definition. This proves the required~$L^1$ bound for~$w_N$ at every time where its trace is defined. Applying Lemma~\ref{lem:difference-smoothing} after the last source time gives the capacity-dependent bound in~\eqref{eq:removal}; comparison also gives~$w_N\le h$. Passing to the continuous potentials uses Lemma~\ref{lem:penalty}.

For compact~$E,F$, choose decreasing compact neighborhoods and decreasing continuous obstacles converging to their indicators at height~$h$. The maximum obstacles approximate~$E\cup F$. The last source time is at most~$t_*+\eta$ for an approximation supported in such a temporal neighborhood. Continuity of capacity on decreasing compact sets, weak convergence of the Riesz measures, and almost-everywhere convergence of the potentials give the estimates with~$\eta\downarrow0$. Finally, regularize the order
\begin{equation*}
 \redpot_{E\cup F}^h\le \redpot_E^h+
 C\capc_h(F)^{\frac{p}{\lambda}}(t-t_*)^{-\frac{n}{\lambda}}\,.
\end{equation*}
Take the past essential lower limit of each side separately. The two potentials agree with their canonical representatives by~\eqref{eq:past-representative}, and the last term is continuous for~$t>t_*$. Hence the same order holds at every point with~$t>t_*$. The bound by~$h$ follows from~$0\le \redpot_E^h\le \redpot_{E\cup F}^h\le h$.
\end{proof}

\begin{remark}\label{rem:measure-union}
For a compact set~$K$, write~$\mu_K^h=\mu_{\redpot_K^h}$, with all potentials taken in the common reference problem of Theorem~\ref{thm:removal}. The same penalization calculation proves the measure inequality
\begin{equation*}
 \mu_{E\cup F}^h\le\mu_E^h+\mu_F^h\,.
\end{equation*}
For fixed continuous obstacles, pass the pointwise source inequality to the weak measure limits as~$N\to\infty$. Then pass through common decreasing outer approximations of~$E,F$ and~$E\cup F$. Both passages preserve the inequality against every nonnegative test function. For the criterion we use Theorem~\ref{thm:removal} with continuous obstacles.
\end{remark}

\subsubsection{Upper levels and excess mass for~$\pc<p<2$}\label{s:excess-section}
Throughout this subsubsection,~$\pc<p<2$. For a level~$c$, the integral of~$(u-c)_+$ measures the excess mass above~$c$. We compare~$u$ with a homogeneous solution whose initial value is~$\max\{u,c\}$. Smoothing bounds this comparison solution by a higher level~$c_+$ on the final interior ball. An~$L^1$ comparison then bounds the excess of~$u$ above~$c_+$ by the source mass.

\begin{lemma}[One-step excess estimate]\label{s:excess}
Fix $0<\theta<1$. Let $0\le u\le H$ be a bounded energy solution of
\begin{equation*}
 u_t-\Dp u=f\ge0
 \quad\text{in }B_R\times(t_0,t_0+T),
\end{equation*}
where $f$ is bounded. For a constant $c\ge0$, set
\begin{equation*}
 E_0=\int_{B_R}(u(x,t_0)-c)_+\,dx,
 \qquad M=\iint_{B_R\times(t_0,t_0+T)}f\,dx\,dt.
\end{equation*}
There is $C=C(n,p,\theta)$ such that, on putting
\begin{equation}\label{s:level-increment}
 c_+=c+C\left[
 T^{-\frac{n}{\lambda}}E_0^{\frac{p}{\lambda}}
 +\left(\frac{T}{R^p}\right)^{\frac{1}{2-p}}\right],
\end{equation}
we have
\begin{equation}\label{s:excess-eq}
 \int_{B_{\theta R}}(u(x,t_0+T)-c_+)_+\,dx\le M.
\end{equation}
The estimate uses the $L^2$ energy traces and holds at every time at which the solutions have their continuous representative, in particular for the penalizations of Lemma~\ref{lem:penalty}.
\end{lemma}
\begin{proof}
If $c\ge H$, the conclusion is immediate. Otherwise let $v$ be the homogeneous solution in the same cylinder with lateral values $H$ and initial value $\max\{u(\cdot,t_0),c\}$. To construct it, solve the evolution for $v-H$ with zero lateral values and its specified bounded $L^2$ initial value. No compatibility of that value with the lateral trace at the initial corner is needed for the energy evolution. Monotone evolution and comparison give
\begin{equation*}
 c\le v\le H.
\end{equation*}
The function $v-c$ is nonnegative and $p$-parabolic, and its initial mass is $E_0$. Lemma~\ref{s:BIV-upper} yields $v\le c_+$ on the final interior ball.

The positive part of $u-v$ has zero initial and lateral trace. Subtract the equations and test with a Lipschitz approximation $S_\epsilon$ to the positive sign function. Since
\begin{equation*}
 (A(\nabla u)-A(\nabla v))\cdot(\nabla u-\nabla v)\ge0,
\end{equation*}
the diffusion term containing $S_\epsilon'$ is nonnegative. The initial primitive vanishes, so letting $\epsilon\downarrow0$ gives
\begin{equation*}
 \int_{B_R}(u-v)_+(x,t_0+T)\,dx\le M.
\end{equation*}
The time derivatives of~$u,v$ belong to~$L^{p'}(t_0,t_0+T;W^{-1,p'}(B_R))$, which justifies the tests by time regularization. On the final interior ball,
$(u-c_+)_+\le(u-v)_+$. This proves~\eqref{s:excess-eq}. The estimate for~$v$ determines~$c_+$, while~$M$ bounds the excess carried to the next interval.
\end{proof}

\subsection{Propagation of mass bounds}\label{subsec:upper-mass}
The upper estimates require control over successive time intervals. For~$p>2$, a local flux estimate preserves a bound for the mass in every ball centered at the origin. For~$\pc<p<2$, the excess estimate carries the source mass from one interval to the next on geometrically decreasing balls.

\subsubsection{Flux estimates with sources for~$p>2$}
\label{sec:flux}

Throughout this subsubsection,~$p>2$. We split a solution with a bounded nonnegative source into a~$p$-parabolic replacement with the same initial and lateral data and a difference with zero data. The gradient of the difference is estimated by testing its level truncations on the full spatial domain. Combining this estimate with the homogeneous flux bound will control the mass transported across a spatial cutoff.

As in Subsubsection~\ref{sec:removal}, the spatial domain~$D$ in the difference estimate is smooth and bounded or is~$\R^n$. Zero lateral trace means the Sobolev zero trace used in the energy tests and zero extension; on~$\R^n$ no lateral boundary is present.

\begin{lemma}[Gradient estimate for a difference]
\label{lem:forced-difference}
Let~$u,v$ be bounded functions in~$L^p(I;W^{1,p}(D))\cap L^\infty(I;L^2(D))$, where~$|I|=T$. Suppose~$w=u-v\ge0$ has zero initial and lateral trace and
\begin{equation*}
 w_t-\div\bigl(A(\nabla u)-A(\nabla v)\bigr)=f,\qquad f\ge0\,.
\end{equation*}
Assume that~$f$ is bounded, and put~$M=\iint f\,dx\,dt$. Then
\begin{equation}\label{eq:forced-difference-flux}
 \iint_{D\times I}|\nabla w|^{p-1}\,dx\,dt
 \le C T^{\frac{1}{\lambda}}M^{1+\frac{p-2}{\lambda}}\,.
\end{equation}
The constant is independent of~$D$ and of the gradients of~$u,v$.
\end{lemma}
The proof is in Appendix~\ref{app:forced-difference-proof}.

\paragraph{Homogeneous flux.}
Let~$v\ge0$ be bounded and~$p$-parabolic on~$B_{4R}\times I$, where~$I=(t_0,t_0+T)$, and put~$S_v=\esssup_{t\in I}\int_{B_{4R}}v(x,t)\,dx$. We use the estimate
\begin{equation}\label{eq:homogeneous-flux}
 \iint_{B_{2R}\times I}|\nabla v|^{p-1}\,dx\,dt
 \le CT^{\frac{1}{\lambda}}S_v^{1+\frac{p-2}{\lambda}}
 \quad\text{if}\quad TS_v^{p-2}\le cR^\lambda\,.
\end{equation}
This is the Euclidean form of~\cite[Theorem~3.1(iii), equation~(3.4)]{KP}. The cited estimate uses ordinary time integrals and spatial averages. Substitute~$S_v/|B_{4R}|$ for its bound on the average and multiply the left side by~$|B_{2R}|$. The identities~$|B_{2R}|/|B_{4R}|=2^{-n}$ and~$(2R)^n/|B_{4R}|=(2^n|B_1|)^{-1}$ give~\eqref{eq:homogeneous-flux}. Its time restriction becomes~$TS_v^{p-2}\le(2R)^p|B_{4R}|^{p-2}=C(n,p)R^\lambda$, which follows from~$TS_v^{p-2}\le cR^\lambda$. The endpoint~$T$ follows by monotone convergence from shorter intervals. Nonzero initial and lateral values are allowed.

\begin{lemma}[Local flux estimate with a source]\label{lem:forced-flux}
Let~$u\ge0$ be a bounded energy solution of~$u_t-\Dp u=f\ge0$ on~$B_{4R}\times I$, with bounded~$f$ and~$|I|=T$. Set
\begin{equation*}
 S=\esssup_{t\in I}\int_{B_{4R}}u(x,t)\,dx
                  +\iint_{B_{4R}\times I}f\,dx\,dt\,.
\end{equation*}
If~$TS^{p-2}\le cR^\lambda$, then
\begin{equation}\label{eq:forced-flux}
 \iint_{B_{2R}\times I}|\nabla u|^{p-1}\,dx\,dt
 \le CT^{\frac{1}{\lambda}}S^{1+\frac{p-2}{\lambda}}\,.
\end{equation}
\end{lemma}
\begin{proof}
We use the homogeneous comparison construction of~\cite[Section~2.2, equation~(2.5)]{KuMiRiesz}. Write~$I=(t_0,t_0+T)$ and put~$V=W^{1,p}_0(B_{4R})$. The equation gives~$u_t\in L^{p'}(I;V^*)$. Using~$u$ as a lifting, solve
\begin{equation*}
 z_t-\div A(\nabla z+\nabla u)=-u_t,
 \qquad z\in L^p(I;V)\cap C(\overline I;L^2(B_{4R})),
 \qquad z(t_0)=0\,.
\end{equation*}
Existence follows from~\cite[Theorem~1.1 and Remark~1.2(2)]{Liu}: monotonicity of~$A$ and Young's inequality give coercivity and growth, with the integrable term~$C\|\nabla u(\cdot,t)\|_{L^p}^p$. Then~$v=u+z$ has the same initial and lateral traces as~$u$, and
\begin{equation*}
 v_t-\Dp v=0,\qquad u_t-\Dp u=f\ge0\,.
\end{equation*}
The common data are nonnegative, so positive-part testing gives~$0\le v\le u$. Thus~$v$ is bounded, and interior continuity follows from~\cite[Theorem~1]{DiB86}. Hence~$S_v\le S$, and~\eqref{eq:homogeneous-flux} bounds the integral of~$|\nabla v|^{p-1}$ on~$B_{2R}\times I$ by~$CT^{\frac{1}{\lambda}}S^{1+\frac{p-2}{\lambda}}$.

The difference~$w=u-v$ has zero initial and lateral trace and satisfies
\begin{equation*}
 w_t-\div\bigl(A(\nabla u)-A(\nabla v)\bigr)=f\,.
\end{equation*}
Lemma~\ref{lem:forced-difference} bounds~$\nabla w$ on the entire replacement cylinder, with~$M=\iint f\,dx\,dt\le S$. Restricting that bound to~$B_{2R}\times I$ and using~$\nabla u=\nabla v+\nabla w$ together with~$|\nabla u|^{p-1}\le C_p(|\nabla v|^{p-1}+|\nabla w|^{p-1})$ proves~\eqref{eq:forced-flux}.
\end{proof}

\subsubsection{Growth bounds and vanishing at the origin for~$p>2$}
\label{sec:waiting}

Throughout this subsubsection,~$p>2$. We bound the mass of~$u$ in balls centered at the origin. If this mass tends to zero faster than the volume, continuity gives~$u(0,t)=0$. The flux estimate from Subsubsection~\ref{sec:flux} preserves the bound for a short time. We group dyadic time strips so that the increase in the bound as~$t\uparrow0$ compensates for the constant in each application. If the sources end before zero,~$u$ is~$p$-parabolic near~$(0,0)$, and continuity also gives~$u(0,0)=0$.

Set
\begin{equation}\label{eq:nu-delta}
 \nu=\frac{\lambda}{p-2}=n+\frac p{p-2}\,,
 \qquad \delta=\frac1{p-2}\,.
\end{equation}
They satisfy
\begin{equation}\label{eq:critical-identities}
 \nu(p-2)=\lambda\,,
 \qquad \nu\left(1+\frac{p-2}{\lambda}\right)=\nu+1\,.
\end{equation}
For~$u\ge0$, define
\begin{equation}\label{eq:growth-norm}
 \mathcal N(u(t))=\sup_{R>0}R^{-\nu}\int_{B_R}u(x,t)\,dx\,.
\end{equation}
The balls in~$\mathcal N(u(t))$ are centered at the origin.

\begin{lemma}[Growth bound for a short time]
\label{lem:growth}
Let~$u\ge0$ be a bounded whole-space energy solution of~$u_t-\Dp u=f\ge0$ on~$I=(t_0,t_0+T)$, with bounded compactly supported forcing, finite mass, and~$L^1$-continuous time traces. Put
\begin{equation*}
 \mathcal F_I=\sup_{R>0}R^{-\nu}\iint_{B_R\times I}f\,dx\,dt\,,
 \qquad B=\mathcal N(u(t_0))+\mathcal F_I<\infty\,.
\end{equation*}
There are~$c_0,C_0>0$, depending only on~$n,p$, such that
\begin{equation}\label{eq:growth-short}
 TB^{p-2}\le c_0
 \quad\Longrightarrow\quad
 \sup_{t\in I}\mathcal N(u(t))\le C_0B\,.
\end{equation}
\end{lemma}
\begin{proof}
We first bound the mass in~$B_R$ with a spatial cutoff, then use continuity in time to show that the bound remains valid. Choose~$\eta$ equal to one on~$B_R$, supported in~$B_{2R}$, with~$|\nabla \eta|\le C/R$. Integrating the equation between~$t_0$ and~$t$ gives
\begin{equation}\label{eq:cutoff-growth}
 \int_{B_R}u(x,t)\,dx \le{}\int_{B_{2R}}u(x,t_0)\,dx +\iint_{B_{2R}\times(t_0,t)}f\,dx\,ds +\frac CR\iint_{B_{2R}\times(t_0,t)}|\nabla u|^{p-1}\,dx\,ds\,.
\end{equation}
The first two terms are bounded by~$CBR^\nu$. We restrict the supremum in~$\mathcal N$ to~$R\ge\rho$ so that it is finite and continuous in time. Fix~$\rho>0$ and set
\begin{equation*}
 \mathcal N_\rho(u(t))
 =\sup_{R\ge\rho}R^{-\nu}\int_{B_R}u(x,t)\,dx\,,
 \qquad
 \mathcal M_\rho(t)=\max_{t_0\le\tau\le t}\mathcal N_\rho(u(\tau))\,.
\end{equation*}
These are finite and continuous functions of time, since
\begin{equation*}
 |\mathcal N_\rho(u(t))-\mathcal N_\rho(u(s))|
 \le\rho^{-\nu}\|u(t)-u(s)\|_1\,.
\end{equation*}
For~$R\ge\rho$, the quantity~$S$ in Lemma~\ref{lem:forced-flux}, on~$B_{4R}\times(t_0,t)$, is at most~$C R^\nu(\mathcal M_\rho(t)+\mathcal F_I)$. Since~$\nu(p-2)=\lambda$, the restriction~$TS^{p-2}\le cR^\lambda$ in that lemma follows from
\begin{equation*}
 T(\mathcal M_\rho(t)+\mathcal F_I)^{p-2}\le c\,.
\end{equation*}
Under this condition, Lemma~\ref{lem:forced-flux} bounds the last term in~\eqref{eq:cutoff-growth}. Divide by~$R^\nu$. The second identity in~\eqref{eq:critical-identities} cancels the remaining power of~$R$. Taking the suprema over~$R\ge\rho$ and~$t_0\le\tau\le t$ gives
\begin{equation}\label{eq:growth-bootstrap}
 \mathcal M_\rho(t)\le C_1B+
 C_2T^{\frac{1}{\lambda}}
       (\mathcal M_\rho(t)+\mathcal F_I)^{1+\frac{1}{\nu}}\,.
\end{equation}

If~$B=0$, both the initial data and the forcing vanish, so comparison gives~$u=0$. Suppose~$B>0$, and choose~$K>4C_1+2$. We show that~$\mathcal M_\rho$ stays below~$KB$. If it first reaches~$KB$ at time~$t$, then~$\mathcal N_\rho(u(\tau))\le KB$ for~$t_0\le\tau\le t$. Since~$\mathcal F_I\le B$, taking~$TB^{p-2}$ sufficiently small, depending only on~$n,p,K$, ensures that the condition for Lemma~\ref{lem:forced-flux} holds up to~$t$. The right side of~\eqref{eq:growth-bootstrap}, divided by~$B$, is then at most
\begin{equation*}
 C_1+C_2(TB^{p-2})^{\frac{1}{\lambda}}(K+1)^{1+\frac{1}{\nu}}<K\,.
\end{equation*}
This contradicts~$\mathcal M_\rho(t)=KB$. Thus~$\mathcal M_\rho\le KB$, with a bound independent of~$\rho$. Letting~$\rho\downarrow0$ proves~\eqref{eq:growth-short}, with~$C_0=K$ and a suitable~$c_0$.
\end{proof}

\begin{lemma}[Vanishing at the spatial origin]\label{lem:forced-waiting}
There is~$\epsilon_0=\epsilon_0(n,p)>0$ with the following property. Let~$u\ge0$ be a bounded whole-space energy solution with zero data before a finite first source time, and with a bounded nonnegative compactly supported source~$f$. Suppose, for~$s_j=2^{-j}$,~$j\in\mathbb Z$,
\begin{equation}\label{eq:source-growth}
 \iint_{B_R\times[-4s_j,-2s_j]}f\,dx\,dt
 \le\epsilon R^\nu s_j^{-\delta}
 \qquad\text{for every }R>0\text{ and }j\in\mathbb Z\,,
\end{equation}
where~$0<\epsilon\le\epsilon_0$. Then
\begin{equation}\label{eq:waiting-growth}
 \mathcal N(u(t))\le C\epsilon(-t)^{-\delta},\qquad t<0\,,
\end{equation}
and~$u(0,t)=0$ for every~$t<0$. If the sources end strictly before zero, then~$u(0,0)=0$ as well.
\end{lemma}
\begin{proof}
We group~$m$ successive dyadic time intervals and apply Lemma~\ref{lem:growth} to each group. We choose~$m$ below so that the estimate at one endpoint gives the required bound at the next. Write the successive endpoints as~$-S_k$, where~$S_{k+1}=\vartheta S_k$ and~$\vartheta=2^{-m}$. The time intervals in~\eqref{eq:source-growth} partition the negative time axis up to their endpoints; for a bounded source the endpoints have zero source measure. A geometric sum gives
\begin{equation}\label{eq:group-source}
 \mathcal F_{(-S_k,-S_{k+1})}
 \le C\epsilon S_{k+1}^{-\delta}\,.
\end{equation}
The constant is independent of~$m$.

Assume inductively that~$\mathcal N(u(-S_k))\le K\epsilon S_k^{-\delta}$. For the next interval, the quantity~$B$ in Lemma~\ref{lem:growth} is at most
\begin{equation*}
 (K\vartheta^\delta+C)\epsilon S_{k+1}^{-\delta}\,.
\end{equation*}
Choose~$m$ large enough that~$C_0\vartheta^\delta\le1/4$, and then choose~$K\ge2C_0C$. Then~$C_0B\le K\epsilon S_{k+1}^{-\delta}$, which is the bound needed for the next step. To check~$TB^{p-2}\le c_0$ in Lemma~\ref{lem:growth}, use~$\delta(p-2)=1$ to obtain
\begin{equation*}
 (S_k-S_{k+1})B^{p-2}
 \le\vartheta^{-1}(K\vartheta^\delta+C)^{p-2}\epsilon^{p-2}\,,
\end{equation*}
whose right-hand side is at most~$c_0$ if~$\epsilon_0$ is sufficiently small. All choices depend only on~$n,p$. The induction starts at a dyadic endpoint before the first source, where~$u=0$. Lemma~\ref{lem:growth} also bounds~$\mathcal N(u(t))$ between successive endpoints. On each such interval,~$-t$ is comparable to~$S_{k+1}$ with factor~$\vartheta^{-1}$. This gives~\eqref{eq:waiting-growth}, with~$\vartheta^{-\delta}$ included in~$C$.

To show that~$u(0,t)=0$, fix~$t<0$. By~\eqref{eq:waiting-growth},
\begin{equation*}
 \vint_{B_R}u(x,t)\,dx
 \le C\epsilon(-t)^{-\delta}R^{\nu-n}
 =C\epsilon(-t)^{-\delta}R^{\frac{p}{p-2}}\longrightarrow0\,.
\end{equation*}
Since~$u$ is continuous,~$u(0,t)=0$. If~$f$ ends before zero,~$u$ is~$p$-parabolic near~$(0,0)$, and continuity gives~$u(0,0)=0$.
\end{proof}

We apply Lemma~\ref{lem:forced-waiting} to the sources~$N(\psi-u_N)_+$ in~\eqref{eq:penalty}. The constants are independent of~$N$.

\subsubsection{Iteration on nested balls for~$\pc<p<2$}\label{subsec:upper-iteration}
Throughout this subsubsection,~$\pc<p<2$. We apply Lemma~\ref{s:excess} on nested balls. The source mass from one interval bounds the initial excess for the next. The choice~$\widetilde r_j^p=b_j^{p-2}s_j$ makes each increase of level comparable to~$b_{j-1}+b_j$, so summability controls the total increase.

\begin{lemma}[Iteration on nested balls]\label{s:iteration}
Let $s_j=2^{-j}$ and let $b_j>0$, $j\ge N-1$, be summable and satisfy
\begin{equation}\label{s:b-ratio}
 \alpha b_j\le b_{j+1}\le\alpha^{-1}b_j,
 \qquad \alpha=2^{-\frac{1}{2(2-p)}}.
\end{equation}
Put
\begin{equation}\label{s:iteration-scales}
 \widetilde r_j^p=b_j^{p-2}s_j,\qquad
 t_j=-2s_j,\qquad I_j=(t_{j-1},t_j).
\end{equation}
Let $u$ be a bounded nonnegative energy solution with bounded nonnegative source $f$, in an ambient cylinder containing all $B_{\widetilde r_j}\times I_j$. Suppose $u=0$ at time $t_{N-2}$ and
\begin{equation}\label{s:mass-budget}
 M_j:=\iint_{B_{\widetilde r_j}\times I_j}f\,dx\,dt\le C_0b_j\widetilde r_j^n,
 \qquad j\ge N-1.
\end{equation}
There are nondecreasing constants $c_j\ge0$, independent of the particular solution, such that
\begin{align}
 \int_{B_{\widetilde r_{j+1}}}(u(x,t_j)-c_j)_+\,dx&\le M_j,
       \label{s:iteration-excess}\\
 c_j&\le C\sum_{i\ge N-1}b_i. \label{s:iteration-level}
\end{align}
Here $C$ depends only on $n,p,C_0$. If $f$ ends strictly before zero and $u$ is continuous near $(0,0)$, then
\begin{equation}\label{s:iteration-point}
 u(0,0)\le C\sum_{i\ge N-1}b_i.
\end{equation}
\end{lemma}
\begin{proof}
The ratio condition gives
\begin{equation}\label{s:R-ratio}
 2^{-\frac{3}{2p}}\le \widetilde r_{j+1}/\widetilde r_j\le2^{-\frac{1}{2p}}<1.
\end{equation}
Thus these radii decrease geometrically to zero. Also $|I_j|=2s_j=s_{j-1}$.

Set $\overline M_j=C_0b_j\widetilde r_j^n$ for $j\ge N-1$, and $\overline M_{N-2}=0$. Starting with $c_{N-2}=0$, define
\begin{equation*}
 c_j=c_{j-1}+C_\theta\left[
 \overline M_{j-1}^{\frac{p}{\lambda}}(2s_j)^{-\frac{n}{\lambda}}
 +\left(\frac{2s_j}{\widetilde r_j^p}\right)^{\frac{1}{2-p}}\right],
 \qquad \theta=2^{-\frac{1}{2p}}.
\end{equation*}
Here~$C_\theta=C(n,p,\theta)$ is the constant from Lemma~\ref{s:excess}. At the first step the initial excess is zero. At every following step the previous bound \eqref{s:iteration-excess} controls the initial excess on precisely $B_{\widetilde r_j}$. Lemma~\ref{s:excess}, and $\widetilde r_{j+1}\le\theta \widetilde r_j$, give \eqref{s:iteration-excess} by induction. If a level exceeds the height bound for $u$, the conclusion remains immediate and the same recurrence may be retained.

For $j\ge N$ the first term simplifies by the exact identity
\begin{equation}\label{s:mass-scale-cancel}
 (b_{j-1}\widetilde r_{j-1}^n)^{\frac{p}{\lambda}}(2s_j)^{-\frac{n}{\lambda}}=b_{j-1}.
\end{equation}
The second term is $2^{\frac{1}{2-p}}b_j$. Therefore
$c_j-c_{j-1}\le C(b_{j-1}+b_j)$, with the first step bounded by $Cb_{N-1}$. Summation proves \eqref{s:iteration-level}.

If the source ends before zero, then $M_j=0$ for large $j$. Equation~\eqref{s:iteration-excess} gives $u(\cdot,t_j)\le c_j$ almost everywhere on $B_{\widetilde r_{j+1}}$. Continuity on the source-free neighborhood implies the same bound at $x=0$. Let $j\to\infty$ to get \eqref{s:iteration-point}.
\end{proof}

\subsection{Intrinsic radii and capacity bounds}\label{sec:decomposition}
We choose increments larger than the Wiener terms. At the corresponding intrinsic radii, the failed capacity inequality bounds the capacities of the inner parts. In the singular range, we replace the increments by a summable majorant with controlled consecutive ratios; its intrinsic radii decrease geometrically. The outer parts require different estimates in the two ranges.

Fix~$a,R_0>0$ and a finite collection or an infinite tail of indices~$j\ge N$.

\subsubsection{Choice of increments above the Wiener terms}
For~$p>2$, choose numbers
\begin{equation}\label{eq:choice-b}
 \omega_a(j)<b_j<a\,,
\end{equation}
and put
\begin{equation}\label{eq:core-radius}
 \widetilde r_j=r_j(b_j)=(b_j^{p-2}s_j)^{\frac{1}{p}}\,.
\end{equation}
No ratio condition is needed in this range.

For~$p<2$, a smaller increment gives a larger intrinsic radius. We majorize the increments by a summable sequence with controlled consecutive ratios. The radii of the majorant decrease geometrically, and its sum is at most a fixed multiple of the original sum.

\begin{lemma}[Sequence regularization]\label{s:regularize-sequence}
Let~$\pc<p<2$ and let $d_i>0$ be indexed by a finite set $N\le i\le m$, or by all $i\ge N$, with finite sum. Set $d_i=0$ at absent indices. For~$K>0$ and~$\alpha$ from~\eqref{s:b-ratio}, define
\begin{equation}\label{s:regularized-b}
 b_j=K\sum_{i\ge N}\alpha^{|j-i|}d_i,
 \qquad j\ge N-1.
\end{equation}
Then \eqref{s:b-ratio} holds and
\begin{equation}\label{s:b-sum}
 \sum_{j\ge N-1}b_j
 \le K\frac{1+\alpha}{1-\alpha}\sum_{i\ge N}d_i.
\end{equation}
Let $\widetilde r_j^p=b_j^{p-2}s_j$, and put
\begin{equation}\label{s:wide-radius}
 \widehat r_i=2\widetilde r_{i-1},\qquad
 q_i=\left(\frac{s_i}{\widehat r_i^p}\right)^{\frac{1}{2-p}}
     =\gamma_p b_{i-1},\qquad
 \gamma_p=2^{-\frac{p+1}{2-p}}.
\end{equation}
If $K\ge2/(\alpha\gamma_p)$, then $q_i\ge2d_i$ at every present index.
\end{lemma}
\begin{proof}
Each summand in \eqref{s:regularized-b} changes by a factor between $\alpha$ and $\alpha^{-1}$ when $j$ increases by one. Summing first over $j$ gives \eqref{s:b-sum}. Since $b_{i-1}\ge K\alpha d_i$, the last bound follows from \eqref{s:wide-radius}. The identity for $q_i$ uses $s_{i-1}=2s_i$ and $\widetilde r_{i-1}^p=b_{i-1}^{p-2}s_{i-1}$.
\end{proof}

The choice~$\widehat r_i=2\widetilde r_{i-1}$ leaves room to enlarge the outer part without entering the iteration balls. Indeed, the enlarged time interval~$J_i^+$ in~\eqref{eq:enlarged-band} meets only~$I_{i-1},I_i,I_{i+1}$, and the corresponding iteration balls have radii at most~$\widetilde r_{i-1}$; see~\eqref{s:Li-location}.

For~$\pc<p<2$, assume
\begin{equation}\label{s:d-above-omega}
 d_i>\omega_a(i),\qquad \sum_{i\ge N}d_i\le\varepsilon a,
\end{equation}
where $\varepsilon=\varepsilon(n,p)>0$ will be small. Fix $K$ as in Lemma~\ref{s:regularize-sequence}. The bound $b_j\le K\sum_i d_i$ and \eqref{s:wide-radius} show that $q_i<a$ when $\varepsilon$ is small enough. For the following decomposition, set
\begin{equation*}
 q_j=
 \begin{cases}
 b_j,&p>2,\\
 \gamma_p b_{j-1},&\pc<p<2,
 \end{cases}
 \qquad
 \rho_j=r_j(q_j)=
 \begin{cases}
 \widetilde r_j,&p>2,\\
 \widehat r_j=2\widetilde r_{j-1},&\pc<p<2.
 \end{cases}
\end{equation*}
In both ranges,~$\omega_a(j)<q_j<a$ at every present index. The defining supremum in~\eqref{eq:omega} gives
\begin{equation}\label{s:failed-test}
 \capc_a(K_{j,q_j},Q_{j,q_j})\le C_*q_j\rho_j^n\,.
\end{equation}

\subsubsection{Inner capacity and outer capacity growth}
For a finite collection or an infinite tail, split the sets~$E_j$ from~\eqref{eq:finite-tail} into
\begin{equation}\label{eq:core-remainder}
 C_j=E_j\cap\{|x|\le\rho_j\},\qquad
 F_j=E_j\cap\{|x|\ge\rho_j\}\,.
\end{equation}
The two parts may overlap on the sphere or be empty. Figure~\ref{fig:core-remainder} shows the decomposition for~$p>2$. Capacities without a displayed reference are whole-space capacities.

\begin{samepage}
\begin{lemma}[Capacity bounds for the inner and outer parts]\label{lem:core-remainder}
In both ranges,
\begin{equation}\label{eq:core-capacity}
 \capc_a(C_j,Q_{j,q_j})\le C_*q_j\rho_j^n\,.
\end{equation}
If~$p>2$, there is~$B_0=B_0(n,p)>0$ such that
\begin{equation}\label{eq:remainder-capacity}
 \capc_a\bigl(F_j\cap(\overline B_R\times\R)\bigr)
 \le B_0 R^\nu s_j^{-\delta}
 \qquad\text{for every }R>0\,.
\end{equation}
Here~$\nu,\delta$ are given by~\eqref{eq:nu-delta}. The constants are independent of~$a,j,R_0$, the chosen increments, and the geometry of~$E$.
\end{lemma}
\end{samepage}
\begin{proof}
Since~$C_j\subset K_{j,q_j}$, capacity monotonicity and~\eqref{s:failed-test} prove~\eqref{eq:core-capacity} in both ranges.

Assume~$p>2$, so~$q_j=b_j$ and~$\rho_j=\widetilde r_j$. We bound~$\capc_a(F_j\cap(\overline B_R\times\R))$ in three ranges. If~$R<\widetilde r_j$, the intersection is empty. If
\begin{equation*}
 \widetilde r_j\le R\le a^{\frac{p-2}{p}}s_j^{\frac{1}{p}}\,,
\end{equation*}
put~$q=(R^p/s_j)^{\frac{1}{p-2}}$, so that~$r_j(q)=R$. Then~$b_j\le q\le a$ and~$q>\omega_a(j)$, so the definition of~$\omega_a(j)$ gives
\begin{equation*}
 \capc_a(K_{j,q},Q_{j,q})\le C_*qR^n
 =C_*R^\nu s_j^{-\delta}\,.
\end{equation*}
Since~$F_j\cap(\overline B_R\times\R)\subset K_{j,q}$, set monotonicity and~\eqref{eq:domain-capacity} give~\eqref{eq:remainder-capacity} for its whole-space capacity.

Finally, suppose~$R\ge a^{\frac{p-2}{p}}s_j^{\frac{1}{p}}$. By Lemma~\ref{lem:cylinder-capacity},
\begin{equation*}
 \capc_a\bigl(F_j\cap(\overline B_R\times\R)\bigr)
 \le C\bigl(aR^n+a^{p-1}s_jR^{n-p}\bigr)\,.
\end{equation*}
Since~$R^p\ge a^{p-2}s_j$, the second term is at most the first. Also~$aR^n\le R^\nu s_j^{-\delta}$. This proves~\eqref{eq:remainder-capacity} in all three ranges.
\end{proof}

\begin{figure}[htbp]
\centering
\begin{tikzpicture}[x=1cm,y=1cm,font=\footnotesize,
  line cap=round,line join=round]
  % One smooth boundary separates the same connected domain and complement
  % throughout the picture. The time and radial cuts are artificial.
  \def\complement{(-3.20,-4.15)--
    plot[domain=-4.15:0,variable=\t,samples=60]
      ({0.16*\t+0.013*\t*\t},\t)--(-3.20,0)--cycle}
  \path[fill=black!7] \complement;
  % Hatch the complementary portion of each time strip, then distinguish
  % its central part. The three intrinsic radii are intentionally nonmonotone.
  \begin{scope}
    \clip (-3.20,-4) rectangle (3.20,-0.5);
    \path[pattern=north east lines,pattern color=black!45] \complement;
  \end{scope}
  \begin{scope}
    \clip (-1.15,-4) rectangle (1.15,-2);
    \path[fill=black!30] \complement;
  \end{scope}
  \begin{scope}
    \clip (-0.70,-2) rectangle (0.70,-1);
    \path[fill=black!30] \complement;
  \end{scope}
  \begin{scope}
    \clip (-1.45,-1) rectangle (1.45,-0.5);
    \path[fill=black!30] \complement;
  \end{scope}
  % The fixed outer cut and dyadic time levels do not describe the boundary.
  \draw[black!45] (-3.20,-4)--(-3.20,0);
  \draw[black!45] (3.20,-4)--(3.20,0);
  \foreach \yy in {-4,-2,-1,-0.5}
    \draw[black!45] (-3.20,\yy)--(3.20,\yy);
  \foreach \xx in {-1.15,1.15}
    \draw[densely dashed] (\xx,-4)--(\xx,-2);
  \foreach \xx in {-0.70,0.70}
    \draw[densely dashed] (\xx,-2)--(\xx,-1);
  \foreach \xx in {-1.45,1.45}
    \draw[densely dashed] (\xx,-1)--(\xx,-0.5);
  \draw[line width=1pt]
    plot[domain=-4.15:0,variable=\t,samples=60]
      ({0.16*\t+0.013*\t*\t},\t);
  \draw[->] (-3.55,0)--(3.65,0) node[right] {$x$};
  \draw[->] (0,0)--(0,0.65) node[above] {$t$};
  \fill (0,0) circle (1.6pt);
  \node[anchor=south west,inner sep=4pt] at (0,0) {$(0,0)$};
  \node[above] at (-3.20,0) {$-R_0$};
  \node[above] at (3.20,0) {$R_0$};
  \node at (-2.20,-0.25) {$E$};
  \node at (2.20,-0.25) {$\Theta$};
  \draw[black!70] (-0.34,-2.65)--(0.38,-2.65);
  \node[anchor=west] at (0.40,-2.65) {$\partial\Theta$};
  \node[fill=black!30,inner sep=1.5pt] at (-0.78,-3.20) {$C_j$};
  \node[fill=white,inner sep=2pt] at (-2.05,-3.20) {$F_j$};
  \node[left] at (-3.25,-4) {$-4s_j$};
  \node[left] at (-3.25,-2) {$-2s_j$};
  \node[left] at (-3.25,-1) {$-s_j$};
  \node[left] at (-3.25,-0.5) {$-s_j/2$};
  \node[right] at (3.25,-3) {$i=j$};
  \node[right] at (3.25,-1.5) {$i=j+1$};
  \node[right] at (3.25,-0.75) {$i=j+2$};
  % Compact legend below the diagram.
  \path[draw=black!70,fill=black!30] (-3.20,-4.55) rectangle (-2.84,-4.32);
  \node[anchor=west] at (-2.72,-4.435) {inner part $C_i$};
  \path[draw=black!70,pattern=north east lines,pattern color=black!45]
    (-0.60,-4.55) rectangle (-0.24,-4.32);
  \node[anchor=west] at (-0.12,-4.435) {remainder $F_i$};
  \draw[densely dashed] (-3.20,-4.92)--(-2.84,-4.92);
  \node[anchor=west] at (-2.72,-4.92)
    {cuts $|x|=\widetilde r_i$,\quad $\widetilde r_i=(b_i^{p-2}s_i)^{\frac{1}{p}}$};
\end{tikzpicture}
\caption{A local cross-section of the decomposition for~$p>2$. The smooth boundary and connected domain~$\Theta$ are used only in the drawing. The solid curve is~$\partial\Theta$, and the time and radial cuts are artificial. The shaded inner parts~$C_i$ and hatched remainders~$F_i$ are portions of the same complement~$E$ inside~$\overline B_{R_0}$. The radii~$\widetilde r_i$ need not decrease with~$i$. Under~\eqref{eq:choice-b}, the inner parts have the capacity bound in~\eqref{eq:core-capacity}, and the remainders have the bound in~\eqref{eq:remainder-capacity} at every radius. For each finite collection, the continuous potential of the smooth remainder obstacles of height~$h=a/A$ vanishes at~$(0,0)$ by Lemma~\ref{lem:remainder-waiting}. Adding the inner parts increases its value by at most~$C\sum_i b_i$ by Theorem~\ref{thm:removal}. The conclusion of Lemma~\ref{lem:remainder-waiting} requires the capacity bound at every radius; spatial separation alone is insufficient.}
\label{fig:core-remainder}
\end{figure}
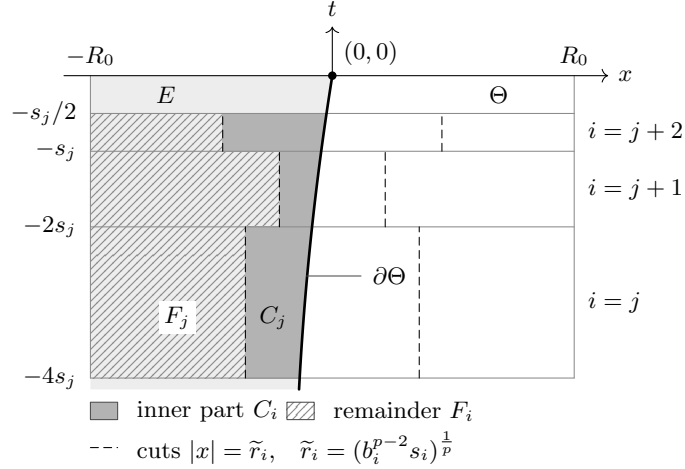

\subsection{Obstacle neighborhoods and source bounds}\label{subsec:upper-neighborhoods}\label{sec:s-upper}
We enlarge the inner and outer parts and fix the resulting obstacles before taking finite truncations. Keeping these neighborhoods fixed ensures that increasing limits retain the prescribed height near the complement. Continuity of capacity preserves the bound for the inner parts in both ranges. For~$p>2$, the outer neighborhoods must satisfy the capacity bound at every radius. For~$\pc<p<2$, the outer neighborhoods instead remain outside the balls used in the iteration during every overlapping time interval.

\subsubsection{Inner neighborhoods}
Choose compact neighborhoods~$C_j^+$ of the inner parts~$C_j$, still compactly contained in~$Q_{j,q_j}$. By~\eqref{eq:core-capacity} and continuity of capacity on decreasing compact sets (see~\cite[Lemma~5.8]{KKKP} in the degenerate range and~\cite[Lemma~3.3]{MS} for the full range~$1<p<\infty$), they can be chosen so that
\begin{equation}\label{eq:core-neighborhood}
 \capc_a(C_j^+,Q_{j,q_j})\le2C_*q_j\rho_j^n\,.
\end{equation}
Keep their supports in~$B_{2\rho_j}\times J_j^+$, where
\begin{equation}\label{eq:enlarged-band}
 J_j^+=[-9s_j/2,-3s_j/2]\,.
\end{equation}
If~$C_j$ is empty, take~$C_j^+$ empty. For~$p>2$, the whole-space capacity~$\capc_a(C_j^+)$ is no larger than the displayed finite-reference capacity. For~$\pc<p<2$, write~$O_j=C_j^+$; thus
\begin{equation}\label{s:Oi-cap}
 \capc_a(O_j,Q_{j,q_j})\le2C_*q_j\widehat r_j^n\,.
\end{equation}

\subsubsection{Outer neighborhoods with capacity growth for~$p>2$}
Throughout this subsubsection,~$p>2$. To prove irregularity, we enlarge~$F_j$ so that the majorant equals the prescribed height near every point of~$F_j$. The enlargement must retain the capacity bound at every radius.

\begin{lemma}[Outer neighborhoods with a growth bound]
\label{lem:growth-neighborhood}
Let~$F$ be compact, with
\begin{equation*}
 F\subset(\overline B_{R_0}\setminus B_{r_*})\times[t_1,t_2]\,,
 \qquad r_*>0\,,
\end{equation*}
and suppose
\begin{equation*}
 \capc_a\bigl(F\cap(\overline B_R\times\R)\bigr)
 \le B R^\nu s^{-\delta},\qquad R>0\,,
\end{equation*}
where~$B,s>0$. Given a temporal neighborhood of~$[t_1,t_2]$, there is a compact neighborhood~$G$ of~$F$, supported in that temporal neighborhood and outside~$B_{r_*/2}\times\R$, such that
\begin{equation}\label{eq:neighborhood-growth}
 \capc_a\bigl(G\cap(\overline B_R\times\R)\bigr)
 \le C B R^\nu s^{-\delta},\qquad R>0\,.
\end{equation}
The constant depends only on~$\nu$, not on~$r_*,R_0$ or the number of spatial scales between them.
\end{lemma}
\begin{proof}
If~$F$ is empty, take~$G$ empty. Otherwise we impose capacity bounds at finitely many dyadic radii and intersect the resulting neighborhoods. Monotonicity will then give the bound at every radius. Choose radii~$\rho_i=2^ir_*/2$, from~$r_*/2$ to a radius larger than~$2R_0$. To leave room for enlargement at radius~$\rho_i$, use the compact set~$F_i=F\cap(\overline B_{2\rho_i}\times\R)$, whose capacity is at most~$B(2\rho_i)^\nu s^{-\delta}$. By continuity of capacity on decreasing compact sets,~\cite[Lemma~5.8]{KKKP}, it has an open neighborhood~$O_i$ whose closure is compact and has capacity at most~$2B(2\rho_i)^\nu s^{-\delta}$. If~$F_i$ is empty, take~$O_i$ empty. All these neighborhoods can be confined to fixed spatial and temporal neighborhoods of~$F_i$.

To retain the bound at~$\rho_i$, we require only the part of the neighborhood inside~$\overline B_{\rho_i}$ to lie in~$O_i$. Accordingly, the finite intersection
\begin{equation*}
 \bigcap_i\left(O_i\cup\{(x,t):|x|>\rho_i\}\right)
\end{equation*}
is an open neighborhood of~$F$. Intersect it also with the prescribed temporal neighborhood, with~$\{|x|>r_*/2\}$, and with a fixed bounded spatial neighborhood inside the ball at the last grid radius. Choose a compact neighborhood~$G$ of~$F$ inside that intersection. At each grid radius,
\begin{equation*}
 G\cap(\overline B_{\rho_i}\times\R)\subset\overline O_i\,.
\end{equation*}
For an intermediate radius, use the next larger grid radius and monotonicity. If~$R\le\rho_i\le2R$, this gives
\begin{equation*}
 \capc_a\bigl(G\cap(\overline B_R\times\R)\bigr)
 \le 2B(2\rho_i)^\nu s^{-\delta}
 \le 2\cdot4^\nu B R^\nu s^{-\delta}\,.
\end{equation*}
Below~$r_*/2$ the intersection is empty. At the last grid radius,~$F_i=F$, so the definition of~$G$ gives~$G\subset O_i$; the same capacity bound then suffices at all larger radii. This proves~\eqref{eq:neighborhood-growth}. Each radius uses a single~$O_i$, so the constant is independent of the number of grid radii.
\end{proof}

Apply this lemma to~$F_j$, choosing its temporal support in~$J_j^+$ from~\eqref{eq:enlarged-band}. We obtain a compact neighborhood~$F_j^+$ with
\begin{equation}\label{eq:Gj-capacity}
 \capc_a\bigl(F_j^+\cap(\overline B_R\times\R)\bigr)
 \le C B_0 R^\nu s_j^{-\delta}\quad(R>0)\,.
\end{equation}
For nonempty~$F_j$, we can keep~$F_j^+$ in~$B_{3R_0}\times\R$. Each~$F_j^+$ stays a positive distance from the spatial origin, although this distance may tend to zero with~$j$.

\subsubsection{Vanishing of the remainder potential for~$p>2$}

Throughout this subsubsection,~$p>2$. At height~$h=a/A$, Lemma~\ref{lem:height-capacity} makes the coefficient in the capacity bound for each~$F_j^+$ small enough to apply Lemma~\ref{lem:forced-waiting}. Each dyadic time interval meets at most three enlarged intervals, so one choice of~$A$ works for every finite collection.

\begin{lemma}[Vanishing of the potential at the origin]\label{lem:remainder-waiting}
There is~$A=A(n,p)>1$ such that the following holds. Let~$h=a/A$, and let the compact neighborhoods~$F_j^+$ satisfy~\eqref{eq:Gj-capacity} and~\eqref{eq:enlarged-band}. For each~$j$ choose a smooth obstacle~$0\le\psi_j\le h$, equal to~$h$ on an open neighborhood of~$F_j$, and supported in~$F_j^+$. For every finite set of indices, the whole-space potential of their maximum satisfies
\begin{equation}\label{eq:remainder-zero}
 \redpot_{\max_j\psi_j}(0,0)=0\,.
\end{equation}
It is bounded by~$h$ and equals~$h$ on open neighborhoods of the corresponding original remainders~$F_j$.
\end{lemma}
\begin{proof}
We verify the source bound in Lemma~\ref{lem:forced-waiting} for the penalized solutions, then pass to the potential. The height estimate in Lemma~\ref{lem:height-capacity} gives the required smallness when~$h=a/A$ and~$A$ is sufficiently large. Fix a finite collection and put~$\psi=\max_j\psi_j$. Use the penalizations
\begin{equation*}
 (u_N)_t-\Dp u_N=f_N\,,
 \qquad f_N=N(\psi-u_N)_+\,,
 \qquad 0\le u_N\le h\,.
\end{equation*}
Their sources are bounded, nonnegative, and supported in~$\bigcup_jF_j^+$. For every compact subset~$K$ of this support,~\eqref{eq:measure-restriction} gives
\begin{equation*}
 \iint_Kf_N\,dx\,dt\le\capc_h(K)\,.
\end{equation*}
We next estimate the source in each time interval required by Lemma~\ref{lem:forced-waiting}. An interval~$[-4s_i,-2s_i]$ intersects~$J_j^+$ only when~$j\in\{i-1,i,i+1\}$. Source restriction and subadditivity of the source measure give the first inequality below. The height estimate in Lemma~\ref{lem:height-capacity}, followed by~\eqref{eq:Gj-capacity}, gives the second:
\begin{equation}\label{eq:penalty-source-growth}
 \iint_{B_R\times[-4s_i,-2s_i]}f_N\,dx\,dt \le\sum_{j=i-1}^{i+1} \capc_h\bigl(F_j^+\cap(\overline B_R\times\R)\bigr) \le C B_0(h/a)^{\frac{1}{p-1}}R^\nu s_i^{-\delta}\,.
\end{equation}
Indices absent from the finite collection contribute zero. The constant includes only the three neighboring time ratios; it is independent of the collection and~$N$.

Choose~$A>1$ so large that the coefficient on the right of~\eqref{eq:penalty-source-growth} is at most the~$\epsilon_0$ in Lemma~\ref{lem:forced-waiting}. Since~$h/a=1/A$, this choice depends only on~$n,p$ and works for every finite collection and every~$N$. Lemma~\ref{lem:forced-waiting} gives~$u_N(0,t)=0$ for~$t<0$. The sources in the finite collection end strictly before zero, so the same lemma gives~$u_N(0,0)=0$. In a fixed neighborhood of~$(0,0)$ the solutions are~$p$-parabolic and uniformly bounded by~$h$. By Lemma~\ref{lem:penalty}, they converge locally uniformly there to~$\redpot_\psi$, proving~\eqref{eq:remainder-zero}. Since this continuous potential lies above~$\psi$ and is at most~$h$, it equals~$h$ wherever one of the~$\psi_j$ equals~$h$.
\end{proof}

\subsubsection{Outer separation and local source bounds for~$\pc<p<2$}\label{subsec:singular-source}
Throughout this subsubsection,~$\pc<p<2$, with increments and radii chosen in Subsection~\ref{sec:decomposition}. The outer part~$F_i$ lies outside~$B_{\widehat r_i}$, where~$\widehat r_i=2\widetilde r_{i-1}$. Choose a compact neighborhood $L_i$ of $F_i$ such that
\begin{equation}\label{s:Li-location}
 L_i\subset\{|x|>3\widetilde r_{i-1}/2\}\times J_i^+.
\end{equation}
Here the common enlarged time interval is
\begin{equation}\label{s:enlarged-band}
 J_i^+=[-9s_i/2,-3s_i/2]\,.
\end{equation}
No capacity estimate for $L_i$ is required. If $F_i$ is empty, take $L_i$ empty.

Since the interiors of $O_i$ and $L_i$ cover $E_i$, there is a smooth function $0\le\psi_i\le a$, equal to $a$ on an open neighborhood of $E_i$, with support in $O_i\cup L_i$. These choices are made once for an infinite tail, before taking finite maxima. All spatial supports can be kept in a fixed ball: the radii $\widetilde r_j$ are decreasing, the spatially truncated sets are contained in $\overline B_{R_0}$, and the neighborhood enlargements can be chosen within a fixed larger ball.

Choose an ambient reference ball $D_*$ containing the spatial projection of $\Theta$, all $B_{64\widehat r_i}$, and all these obstacle neighborhoods, with positive margins. Choose its initial time~$b_*$ earlier than every time in~$\Theta$ and earlier than~$-10s_N$, and its final time later than zero. For a finite set of indices consider, in that fixed reference, the penalization
\begin{equation}\label{s:local-budget-penalty}
 (u_\nu)_t-\Dp u_\nu=f_\nu,
 \qquad f_\nu=\nu(\max_i\psi_i-u_\nu)_+,
 \qquad 0\le u_\nu\le a.
\end{equation}
Lemma~\ref{lem:penalty} applies to the finite maximum. We use $\nu$ for the penalty index to distinguish it from the initial dyadic index $N$.

\begin{lemma}[Local source bound]\label{s:local-budget}
The solutions \eqref{s:local-budget-penalty} satisfy
\begin{equation}\label{s:local-budget-eq}
 \iint_{B_{\widetilde r_j}\times I_j}f_\nu\,dx\,dt
 \le C_0 b_j\widetilde r_j^n,
 \qquad j\ge N-1,
\end{equation}
with $C_0=C_0(n,p)$ independent of the penalty index and the finite collection.
\end{lemma}
\begin{proof}
The time interval~$I_j=(-4s_j,-2s_j)$ meets $J_i^+$ only for $i\in\{j-1,j,j+1\}$. For each such index, $i-1\le j$, so decreasing radii give $\widetilde r_j\le \widetilde r_{i-1}$. Hence \eqref{s:Li-location} implies
\begin{equation*}
 L_i\cap(B_{\widetilde r_j}\times I_j)=\varnothing.
\end{equation*}
Since $f_\nu$ vanishes outside the supports of the $\psi_i$,
\begin{equation*}
 \iint_{B_{\widetilde r_j}\times I_j}f_\nu\,dx\,dt
 \le\sum_{i=j-1}^{j+1}\iint_{O_i}f_\nu\,dx\,dt.
\end{equation*}
An absent index contributes zero. Apply Lemma~\ref{s:source-restriction} to each $O_i\Subset Q_{i,q_i}$ and then use \eqref{s:Oi-cap}. This gives
\begin{equation*}
 \iint_{B_{\widetilde r_j}\times I_j}f_\nu\,dx\,dt
 \le2C_*\sum_{i=j-1}^{j+1}q_i\widehat r_i^n.
\end{equation*}
By \eqref{s:wide-radius}, each summand is
$2^n\gamma_p b_{i-1}\widetilde r_{i-1}^n$. The ratios \eqref{s:b-ratio} and \eqref{s:R-ratio} compare these finitely many neighboring values to $b_j\widetilde r_j^n$ by a constant depending only on $p,n$. This proves \eqref{s:local-budget-eq}.
\end{proof}

\subsection{Proof of the cumulative upper estimate}\label{subsec:upper-proof}
\begin{proof}[Proof of Theorem~\ref{thm:upper-main}]
\emph{The range~$p>2$.} The obstacle height bounds~$\redpot_{E_{N,m}}^{a/A}$ by~$a/A$. To obtain the series bound, we add the inner parts to a potential that vanishes at the origin. If~$\omega_a(j)=a$ for any~$j\in\{N,\ldots,m\}$, estimate~\eqref{eq:cumulative-upper} already follows from~$\redpot_{E_{N,m}}^{a/A}\le a/A$ after increasing~$C$. Otherwise choose~$b_j\in(\omega_a(j),a)$ with
\begin{equation}\label{eq:finite-excess}
 \sum_{j=N}^m b_j\le\sum_{j=N}^m\omega_a(j)+\eta\,,
\end{equation}
where~$\eta>0$ is arbitrary. Choose the inner parts~$C_j$ and remainders~$F_j$ by~\eqref{eq:core-remainder}, and their neighborhoods~$C_j^+,F_j^+$ as in Subsection~\ref{subsec:upper-neighborhoods}. Let~$h=a/A$ with the constant from Lemma~\ref{lem:remainder-waiting}.

Take the obstacles~$\psi_j$ from that lemma. For each inner part, choose a smooth obstacle~$0\le\theta_j\le h$, equal to~$h$ on an open neighborhood of~$C_j$, and supported in~$C_j^+$. Let
\begin{equation*}
 V_m=\redpot_{\max_{N\le j\le m}\psi_j}\,,
 \qquad
 U_m=\redpot_{\max_{N\le j\le m}\{\psi_j,\theta_j\}}\,.
\end{equation*}
By Lemma~\ref{lem:remainder-waiting},~$V_m(0,0)=0$. Add the obstacles~$\theta_j$ one at a time. For each~$j$, apply Theorem~\ref{thm:removal} with~$t_*=-s_j$ and~$T=0$; the obstacle~$\theta_j$ ends before~$-s_j$. Height and domain monotonicity, followed by~\eqref{eq:core-neighborhood}, give~$\capc_h(C_j^+)\le\capc_a(C_j^+,Q_{j,b_j})\le2C_*b_j\widetilde r_j^n$. Telescoping the successive removal estimates and using~$V_m(0,0)=0$ gives
\begin{equation}\label{eq:restore-cores}
 U_m(0,0) \le C\sum_{j=N}^m \bigl(b_j\widetilde r_j^n\bigr)^{\frac{p}{\lambda}}s_j^{-\frac{n}{\lambda}} =C\sum_{j=N}^m b_j\,.
\end{equation}
For the equality, use
\begin{equation*}
 b_j\widetilde r_j^n=b_j^{\frac{\lambda}{p}}s_j^{\frac{n}{p}}\,.
\end{equation*}
The constant is independent of~$V_m$ and the obstacles~$\theta_i$ already added, including those that act at later times.

The function~$U_m$ is an admissible height-$h$ majorant of~$E_{N,m}$. It equals~$h$ in neighborhoods of all points of~$E_{N,m}$. Therefore
\begin{equation*}
 \redpot_{E_{N,m}}^h(0,0)\le U_m(0,0)
 \le C\sum_{j=N}^m\omega_a(j)+C\eta\,.
\end{equation*}
Let~$\eta\downarrow0$, and combine this with~$\redpot^h\le h$. This proves~\eqref{eq:cumulative-upper}.

\emph{The range~$\pc<p<2$.} Write $W=\sum_{i=N}^m\omega_a(i)$. If $W$ exceeds a sufficiently small fixed multiple of $a$, the bound $\redpot^a\le a$ proves~\eqref{s:upper-theorem-eq} after increasing~$C$. Thus assume $W<\varepsilon a/2$, with $\varepsilon$ as in \eqref{s:d-above-omega}. Choose $d_i>\omega_a(i)$ with $\sum_i d_i\le W+\xi<\varepsilon a$.

Define~$b_j$ by~\eqref{s:regularized-b}, and choose the neighborhoods, obstacles and penalizations as in Subsubsection~\ref{subsec:singular-source}. Their first source time is later than $t_{N-2}=-8s_N$, so $u_\nu=0$ there. Lemma~\ref{s:local-budget} verifies the hypothesis of Lemma~\ref{s:iteration}. For each fixed finite collection, the sources end strictly before zero, so all its penalizations are $p$-parabolic in one common neighborhood of~$(0,0)$. Lemma~\ref{s:iteration} gives
\begin{equation*}
 u_\nu(0,0)\le C\sum_{j\ge N-1}b_j
 \le C'\sum_{i=N}^m d_i.
\end{equation*}
By Lemma~\ref{lem:penalty},~$u_\nu$ increases locally uniformly to the continuous-obstacle potential in the fixed reference, which satisfies the same pointwise bound. This potential equals $a$ on open neighborhoods of $E_{N,m}$ and therefore majorizes the potential of the compact obstacle. We conclude
\begin{equation*}
 \redpot_{E_{N,m},D_*,b_*}^a(0,0)\le C'(W+\xi).
\end{equation*}
This holds with constants independent of every sufficiently large reference ball $D_*$. Exhaustion gives the same estimate for the whole-space reference potential. Equivalently, one may retain the bounded references throughout the later Perron argument. Let $\xi\downarrow0$ and use $\redpot^a\le a$ to prove \eqref{s:upper-theorem-eq}. For a smooth bounded reference cylinder~$Q$ with~$E_{N,m}\cup\{(0,0)\}\Subset Q$, restricting whole-space admissible majorants to its continuation cylinder and then regularizing gives~$\redpot^a_{E_{N,m},Q}\le\redpot^a_{E_{N,m}}$ in~$Q$. Thus the same bound holds independently of~$Q$.

For the continuous obstacle used here, the potentials in bounded reference balls increase and remain bounded by~$a$. Their limit is superparabolic by bounded monotone convergence and still majorizes the obstacle pointwise. Every global admissible majorant dominates each bounded-reference potential, so the limit is the global continuous-obstacle potential. It therefore majorizes the original compact obstacle. The necessity proof in Subsubsection~\ref{subsec:intrinsic-necessity} uses one fixed bounded reference ball.
\end{proof}

For a convergent singular tail, the same construction gives excess bounds for one fixed family of obstacles.

\begin{corollary}[Uniform excess bounds for obstacle potentials]\label{cor:singular-tail}
Let~$n\ge2$, $\pc<p<2$ and~$a,R_0>0$. Suppose that~$\sum_j\omega_a(j)<\infty$. For every~$\epsilon\in(0,1)$ there are~$N\in\N$, a bounded reference cylinder~$Q$ with~$\overline\Theta\cap\{t\le0\}\Subset Q$, and functions~$\psi_i\in C_c^\infty(Q\cap\{t<0\})$, $i\ge N$, such that~$0\le\psi_i\le a$ and~$\psi_i=a$ on an open neighborhood of~$E_i$ from~\eqref{s:finite-complement}. Let~$U_m$ be the potential of~$\max_{N\le i\le m}\psi_i$ in~$Q$ with zero initial and lateral data, and put~$U=\sup_{m\ge N}U_m$.

There are radii~$\widetilde r_j\downarrow0$ and nondecreasing levels~$c_j\ge0$, $j\ge N-1$, with~$B_{\widetilde r_j}\times\{t_j\}\Subset Q$ for~$t_j=-2s_j$, such that, for every~$m\ge N$ and~$j\ge N-1$,
\begin{equation*}
 c_j<\frac{\epsilon a}{2},\qquad \vint_{B_{\widetilde r_{j+1}}}(U_m(x,t_j)-c_j)_+\,dx<\frac{\epsilon a}{2}\,.
\end{equation*}
The reference, obstacles, radii and levels are independent of~$m$. Moreover,~$U$ is bounded and~$p$-superparabolic in~$Q$, and there is a sequence~$z_\ell\in\Theta\cap\{t<0\}$ tending to~$(0,0)$ with~$U(z_\ell)<\epsilon a$.
\end{corollary}
\begin{proof}
Fix~$\alpha,K$ as in Lemma~\ref{s:regularize-sequence}. Convergence of the series allows us to choose~$N$ and~$d_i>\omega_a(i)$, $i\ge N$, so that~\eqref{s:d-above-omega} holds and the sequence~$b_j$ in~\eqref{s:regularized-b} satisfies
\begin{equation}\label{s:small-tail}
 C\sum_{j\ge N-1}b_j<\frac{\epsilon a}{2}\,.
\end{equation}
Here~$C=C(n,p)$ is chosen to dominate the constant in~\eqref{s:iteration-level} and the constant in the normalized excess bound below. These choices are possible by~\eqref{s:b-sum}. Define~$\widetilde r_j,t_j$ by~\eqref{s:iteration-scales}. Choose all neighborhoods~$O_i,L_i$, obstacles~$\psi_i$, and one reference cylinder~$Q$ as in Subsubsection~\ref{subsec:singular-source}, before taking finite maxima.

For every~$m\ge N$, Lemma~\ref{s:local-budget} bounds the sources of the penalizations for~$\max_{N\le i\le m}\psi_i$. They vanish at~$t_{N-2}$. Apply Lemma~\ref{s:iteration} with the same bounds~$\overline M_j=C_0b_j\widetilde r_j^n$ for every~$m$. Its recurrence gives the same levels~$c_j$ for all truncations. At each fixed time~$t_j$, local uniform convergence in Lemma~\ref{lem:penalty} passes the excess estimate to~$U_m$, giving
\begin{equation*}
 \int_{B_{\widetilde r_{j+1}}}(U_m(x,t_j)-c_j)_+\,dx\le C_0b_j\widetilde r_j^n\,.
\end{equation*}
Divide by~$|B_{\widetilde r_{j+1}}|$ and use the lower radius ratio in~\eqref{s:R-ratio}. The normalized excess is at most~$Cb_j<\epsilon a/2$, while~\eqref{s:iteration-level} and~\eqref{s:small-tail} give~$c_j<\epsilon a/2$.

Since~$0\le U_m\le a$ and~$U_m\uparrow U$, the limit is bounded and superparabolic. Monotone convergence at each fixed~$t_j$ gives
\begin{equation}\label{s:small-average}
 \frac1{|B_{\widetilde r_{j+1}}|}\int_{B_{\widetilde r_{j+1}}}U(x,t_j)\,dx\le c_j+Cb_j<\epsilon a\,.
\end{equation}
The fixed obstacle~$\psi_i$ equals~$a$ near~$E_i$, so~$U_m=U=a$ there for~$m\ge i$. For all sufficiently large~$j$, the ball at time~$t_j$ lies in~$B_{R_0/2}\times(-2s_N,0)$, where every point of~$E$ belongs to some~$E_i$, $i\ge N$. By~\eqref{s:small-average}, choose~$x_j\in B_{\widetilde r_{j+1}}$ with~$U(x_j,t_j)<\epsilon a<a$. Such points lie in~$\Theta$. Since~$\widetilde r_j\downarrow0$ and~$t_j\uparrow0$, they give the required sequence.
\end{proof}

\section{Wiener criteria}
\label{sec:necessity}

For sufficiency, the bound~$0\le H-u\le a$ near the origin for a finite-union potential~$u$ gives approximate barriers and hence Perron competitors with arbitrarily small error (Proposition~\ref{prop:sufficiency}). For necessity, the cumulative upper estimate supplies an irregularity barrier: it equals~$h$ near the local complement in the past but stays below~$h/2$ along an interior sequence approaching the origin. Perron comparison then proves irregularity (Proposition~\ref{prop:necessity}). We subsequently compare the intrinsic terms with the prescribed capacitary potentials in Theorem~\ref{thm:potential-wiener} and derive the formulation in terms of capacity differences.

\subsection{The intrinsic Wiener criterion}\label{subsec:intrinsic-criterion-proof}
The following theorem gives the countable-height formulation of Theorem~\ref{cor:intrinsic-wiener-intro}.

\begin{theorem}[Intrinsic Wiener criterion]\label{thm:main}\label{s:singular-theorem}
Let~$n\ge2$ and~$p\in(\pc,2)\cup(2,\infty)$, with the constants fixed before Definition~\ref{def:wiener}. Let~$\Theta$ be a bounded open subset of~$\spacetime$, and let~$0\in\partial\Theta$. For the terms of Definition~\ref{def:wiener}, the following conditions are equivalent:
\begin{enumerate}[(i)]
\item The origin is regular for~\eqref{eq:pde} in~$\Theta$.
\item For every~$a>0$, one has~$\sum_{j=1}^\infty\omega_a(j)=\infty$.
\item For every positive integer~$k$, one has~$\sum_{j=1}^\infty\omega_{1/k}(j)=\infty$.
\end{enumerate}
\end{theorem}

Changing finitely many terms does not affect the criterion. By translation, the same formulation holds at every boundary point. The reference cylinders cross the boundary, and no geometric assumption is imposed on~$\Theta$.

In the proof of Theorem~\ref{thm:main}, fix~$\Theta$ as in its hypotheses and put~$\Theta_-=\Theta\cap\{t<0\}$. We prove the two implications separately and complete the criterion in Subsubsection~\ref{subsec:intrinsic-proof}.

\subsubsection{Divergence and regularity}\label{subsec:joint-sufficiency}\label{s:singular-criterion}
The finite-union potentials have their Riesz measures in the complement, so they and their reflections are~$p$-parabolic in~$\Theta_-\cap Q^*$. The bound~$0\le H-\redpot^H_{K,Q^*}\le a$ near the origin gives the local Perron competitors. A large~$H$ permits pasting to constants away from the origin, while letting~$a\to0$ makes the local error arbitrarily small.

\begin{proposition}[Regularity from divergence of the Wiener series]\label{prop:sufficiency}\label{s:singular-sufficiency}
If~$\sum_j\omega_{1/k}(j)=\infty$ for every~$k\in\N$, then the origin is regular for~$\Theta$.
\end{proposition}
\begin{proof}
We may work in~$\Theta_-$. If the origin is not a boundary point of this part, it is regular by~\cite[Lemma~3.10]{BBGP}. Otherwise,~\cite[Theorem~3.11]{BBGP} reduces regularity to~$\Theta_-$, and the tested complementary sets at negative times are unchanged.

Let~$f\in C(\partial\Theta_-)$ and fix~$\epsilon>0$. We construct members of the upper and lower Perron classes whose values near the origin differ from~$f(0)$ by at most~$\epsilon+a$, where~$a<\epsilon$. Choose~$a=1/k$ with~$a<\epsilon$. Take a small cylinder~$Q^*=B_d\times(-\tau,\tau)$ so that
\begin{equation*}
 |f(z)-f(0)|<\epsilon
 \quad\text{on }\partial\Theta_-\cap\overline{Q^*}\,.
\end{equation*}
Let~$M>\|f\|_\infty+|f(0)|+2$, and choose~$H>a$ so large that
\begin{equation*}
 f(0)+\epsilon+H>M+1\,,
 \qquad f(0)-\epsilon-H<-M-1\,.
\end{equation*}
By Proposition~\ref{prop:saturation}, there is a compact set~$K\Subset Q^*\cap\{t<0\}$ in the complement whose potential~$\redpot=\redpot_{K,Q^*}^{H}$ satisfies~$\redpot\ge H-a$ near the origin. In~$\Theta_-\cap Q^*$,~$\redpot$ is~$p$-parabolic and~$0\le \redpot\le H$. It vanishes continuously on the artificial lateral and bottom boundary of~$Q^*$; those boundary portions are separated from~$K$.

We have~$H-\redpot\ge0$, and~$\redpot\ge H-a$ gives~$H-\redpot\le a$ near the origin. We use it to define, on~$\Theta_-\cap Q^*$,
\begin{equation*}
 V^+=\min\{M,f(0)+\epsilon+H-\redpot\}\,,
 \qquad
 V^-=\max\{-M,f(0)-\epsilon-H+\redpot\}\,.
\end{equation*}
Set~$V^+=M$ and~$V^-=-M$ on the rest of~$\Theta_-$. Because~$\redpot$ vanishes continuously on the artificial boundary, the strict margins in the choice of~$H$ make the local definitions equal to~$M$ and~$-M$ near that boundary. The pasted functions are therefore continuous at every gluing point inside~$\Theta_-$. The pasting lemma~\cite[Lemma~2.9]{BBGP}, applied also after negation, shows that~$V^+$ is superparabolic and~$V^-$ is subparabolic. The artificial top~$t=\tau>0$ does not meet~$\Theta_-$.

On the genuine boundary inside~$\overline{Q^*}$, the inequalities~$H-\redpot\ge0$ and~$|f-f(0)|<\epsilon$ show that the limiting inequalities for the upper and lower Perron classes hold. On the remaining boundary,~$M$ and~$-M$ bound~$f$. Both functions are bounded, so they belong to the respective Perron classes. By~\cite[Theorem~2.6]{BBGP},
\begin{equation*}
 V^-\le\lowerH_{\Theta_-}f
 \le\upperH_{\Theta_-}f\le V^+\,.
\end{equation*}
Near the origin,~$\redpot\ge H-a$ gives
\begin{equation*}
 f(0)-\epsilon-a\le V^-\le V^+\le f(0)+\epsilon+a\,.
\end{equation*}
Let~$\epsilon\downarrow0$. Both Perron solutions attain~$f(0)$. This proves regularity in~$\Theta_-$ and hence in~$\Theta$.
\end{proof}

\subsubsection{Convergence and irregularity}\label{subsec:intrinsic-necessity}
We construct an irregularity barrier~$U$: it equals a fixed height~$h$ near the local complement in the past but has values below~$h/2$ along an interior sequence approaching the origin. Fixed continuous obstacle neighborhoods preserve the equality~$U=h$ when we take increasing limits. The following comparison bounds a lower Perron solution by~$U$, so that it cannot attain a boundary datum equal to~$h$ near the origin.

\begin{lemma}[Irregularity from a superparabolic majorant]\label{lem:perron-irregularity}
Let~$U\ge0$ be bounded and~$p$-superparabolic in~$\Theta_-$. Suppose there are~$h>0$ and an open neighborhood~$W$ of the origin such that
\begin{equation*}
 \liminf_{\Theta_-\ni z\to\xi}U(z)\ge h
 \quad\text{for every }\xi\in\partial\Theta\cap W\cap\{t<0\}\,.
\end{equation*}
If there is a sequence~$z_\ell\in\Theta_-$ tending to the origin with~$\limsup_{\ell\to\infty}U(z_\ell)<h$, then the origin is irregular for~$\Theta$.
\end{lemma}
\begin{proof}
Restrict to~$\partial\Theta$ a continuous ambient cutoff supported in~$W$, equal to one near the origin, and multiply it by~$h$. This gives~$g\in C(\partial\Theta)$ with~$0\le g\le h$ and~$g=h$ near the origin. Let~$v$ be any lower Perron competitor for~$g$. At boundary points strictly before zero where~$g>0$, the hypothesis gives~$\liminf U\ge h\ge g\ge\limsup v$. At the other boundary points strictly before zero,~$\limsup v\le0\le\liminf U$. Apply the parabolic comparison principle~\cite[Theorem~2.4]{BBGP} in~$\Theta_-$ with terminal time zero. It gives~$v\le U$, and taking the supremum over the lower Perron class yields
\begin{equation*}
 \lowerH_\Theta g\le U\quad\text{in }\Theta_-\,.
\end{equation*}
Along the given sequence,~$\limsup_{\ell\to\infty}\lowerH_\Theta g(z_\ell)<h=g(0)$. Thus the lower Perron solution does not attain the boundary datum at the origin. Sign symmetry of the Perron solutions implies irregularity; see~\cite[Definition~3.1 and the following observation]{BBGP}.
\end{proof}

\begin{proposition}[Irregularity from convergence of the Wiener series]\label{prop:necessity}\label{s:necessity}\label{subsec:degenerate-necessity}\label{s:necessity-section}
If~$\sum_j\omega_a(j)<\infty$ for some~$a>0$, then the origin is irregular for~$\Theta$.
\end{proposition}
\begin{proof}
Fix~$R_0>0$ and put
\begin{equation*}
 h=
 \begin{cases}
 a/A,&p>2,\\
 a,&\pc<p<2,
 \end{cases}
\end{equation*}
where~$A>1$ is from Theorem~\ref{thm:upper-main} in the range~$p>2$. We construct an increasing family of potentials of continuous obstacles, bounded by~$h$. Its limit equals~$h$ near the complement in the local past, and has values below~$h/2$ at interior points approaching the origin. Lemma~\ref{lem:perron-irregularity} then applies in both ranges. Use the sets~$E_j$ of~\eqref{eq:finite-tail} with this~$R_0$.

\emph{Step 1: Continuous obstacles for a small tail.} We pass to a sufficiently far tail in each range and fix all obstacle neighborhoods before taking finite maxima.

For~$p>2$, choose~$N$ large enough and~$b_j\in(\omega_a(j),a)$ with a summable positive excess so that
\begin{equation}\label{eq:small-tail}
 C\sum_{j=N}^\infty b_j<h/4\,,
\end{equation}
where~$C$ is the constant in~\eqref{eq:restore-cores}. Choose the neighborhoods and continuous obstacles from Subsection~\ref{subsec:upper-neighborhoods}. Their spatial supports lie in a common bounded ball because~$\widetilde r_j\le a^{\frac{p-2}{p}}s_j^{\frac{1}{p}}$ and~$F_j\subset\overline B_{R_0}\times\R$. All their source times lie strictly below zero and above a fixed first time. Let~$U_m$ be the whole-space continuous-obstacle potential used in~\eqref{eq:restore-cores}. For every finite truncation,
\begin{equation}\label{eq:finite-small}
 U_m(0,0)<h/4\,.
\end{equation}

For~$\pc<p<2$, apply Corollary~\ref{cor:singular-tail} with~$\epsilon=1/4$ and use its fixed reference cylinder, obstacles, and potentials~$U_m$.

\emph{Step 2: The increasing limit and equality near the complement.} In both ranges,~$0\le U_m\le h$ and~$U_m\le U_{m+1}$. Thus
\begin{equation*}
 U=\sup_{m\ge N}U_m
\end{equation*}
is a bounded superparabolic function on the common reference, with~$0\le U\le h$; see Subsection~\ref{subsec:capacity-foundations}. In the degenerate case the reference is~$\spacetime$, and in the singular case it is the fixed ambient cylinder from Step~1.

Every complementary point with~$|x|<R_0/2$ and~$-2s_N<t<0$ belongs to some~$E_j$, $j\ge N$. The chosen continuous obstacle equals~$h$ on an open ambient neighborhood of that point. Hence~$U_m=h$ there for~$m\ge j$. Since the neighborhoods are fixed and~$U\le h$, the limit~$U$ equals~$h$ on the same neighborhood.

\emph{Step 3: Small interior values.} For~$p>2$, passing to the supremum in~\eqref{eq:finite-small} gives~$U(0,0)\le h/4$. To obtain small values at points approaching the origin from the past, use~\eqref{eq:past-representative}, or~\cite[Theorem~5.1]{KLpoint}:
\begin{equation*}
 U(0,0)=\essliminf_{(x,t)\to(0,0),\ t<0}U(x,t)\,.
\end{equation*}
Since~$U(0,0)\le h/4$, there is a sequence of points approaching the origin from the past on which~$U<h/2$. These points can be chosen in~$\Theta$: in the local past neighborhood just described, every point of the complement has~$U=h$.

For~$\pc<p<2$, Corollary~\ref{cor:singular-tail} gives a sequence in~$\Theta_-$ approaching the origin with~$U<a/4<h/2$.

In either range we have a sequence in~$\Theta_-$ approaching the origin with~$U<h/2$ along it.

\emph{Step 4: Perron comparison.} Apply Lemma~\ref{lem:perron-irregularity} to the restriction of~$U$ to~$\Theta_-$, with height~$h$ and neighborhood~$B_{R_0/2}\times(-2s_N,2s_N)$. Step~2 gives the boundary lower bound, and Step~3 gives an interior sequence with~$\limsup U\le h/2<h$. The origin is irregular.
\end{proof}

\subsubsection{Proof of the intrinsic criterion}\label{subsec:intrinsic-proof}
\begin{proof}[Proof of Theorems~\ref{cor:intrinsic-wiener-intro} and~\ref{thm:main}]
For the conditions in Theorem~\ref{thm:main}, Proposition~\ref{prop:sufficiency} proves (iii)$\Rightarrow$(i), and Proposition~\ref{prop:necessity} gives the contrapositive of (i)$\Rightarrow$(ii). Condition (ii) implies (iii). Conversely, for any~$a>0$, choose~$k$ with~$1/k\le a$. By~\eqref{eq:height-omega},~$\omega_{1/k}(j)\le\omega_a(j)$, so divergence at~$1/k$ implies divergence at~$a$. This also proves the formulation in Theorem~\ref{cor:intrinsic-wiener-intro}.
\end{proof}

\paragraph{The range~$p>2$.}
The criterion and cumulative upper estimate hold for every~$n\ge2$ and~$p>2$. Lemma~\ref{lem:interpolation} gives both spatial interpolation steps at and above~$p=n$, with the same powers of~$n,p$ and~$\lambda$. The homogeneous flux estimate~\cite[Theorem~3.1(iii), equation~(3.4)]{KP}, the weak Harnack estimate~\cite[Theorem~1.1]{KuusiH}, and the capacity results for compact sets, bounded penalizations and comparison principles all apply in this range. The strict inequality~$p>2$ is used in the sum over the levels below~$k_0$ after~\eqref{eq:level-flux}, the contradiction argument following~\eqref{eq:growth-bootstrap}, finite propagation in Lemma~\ref{lem:persistence}, and the backward selection of intrinsic cylinders.

\begin{remark}[Order of constants]\label{s:singular-constants}
For~$p>2$, first fix~$\kappa$ from Lemma~\ref{lem:persistence}. Choose~$L$ by~\eqref{eq:L-choice}, and then choose~$C_*$ from the normalized mass bound in Lemma~\ref{lem:reception}. These determine the test~\eqref{eq:omega}. Next fix~$B_0$ from Lemma~\ref{lem:core-remainder}. The flux estimate determines the constants in Lemma~\ref{lem:growth} and~$\epsilon_0$ in Lemma~\ref{lem:forced-waiting}. Finally choose~$A$ so that~$CB_0A^{-\frac{1}{p-1}}\le\epsilon_0$ in~\eqref{eq:penalty-source-growth}. All of these are functions only of~$n,p$. The height~$a>0$ is arbitrary.

In this range, the selection in the sufficiency proof depends on~$a$, whereas the necessity proof uses~$a/A$. The criterion retains the quantifier over every~$a>0$.

For~$\pc<p<2$, first fix the constants in Lemmas~\ref{s:BIV-lower} and~\ref{lem:persistence}, then~$\eta,c_0,C_*$ in Lemma~\ref{lem:reception}. These determine the test. For the upper estimate, fix~$\alpha$ and~$K$ as in Lemma~\ref{s:regularize-sequence}. They and~$C_*$ determine the local source bound in Lemma~\ref{s:local-budget}, which in turn determines the level bound in Lemma~\ref{s:iteration}. After fixing these constants, choose the smallness threshold~$\varepsilon$ in~\eqref{s:d-above-omega}. For sufficiency, choose the forward selection constants using the spatial enlargement factor~$64$ and~$\eta$. All constants depend only on~$n,p$ and may diverge as~$p\downarrow\pc$ or~$p\uparrow2$.
\end{remark}

\subsection{The potential Wiener criterion}\label{subsec:potential-proof}
To derive Theorem~\ref{thm:potential-wiener} from Theorem~\ref{thm:main}, we compare the intrinsic terms with the potentials of the condensers in~\eqref{eq:potential-test-sets}. Applying the cumulative upper estimate on one time interval gives the upper bounds. The lower bounds follow from the capacity estimates in Section~\ref{sec:lower}. In the singular range, the total cost of the increments whose reference cylinders do not fit inside~$Q_j$ is summable.

\begin{proof}[Proof of Theorem~\ref{thm:potential-wiener}]
For~$p>2$, let~$A>1$ be from Theorem~\ref{thm:upper-main}; for~$\pc<p<2$, put~$A=1$ in this proof. We first bound~$\redpot^a_{K_j,Q_j}(0,0)$ from above and then from below by the intrinsic terms.

\emph{Step 1: The upper bound.} Apply Theorem~\ref{thm:upper-main} with prescribed height~$Aa$, spatial cutoff~$R_j$ and~$N=m=j$. The potential in the finite reference cylinder is no larger than the whole-space potential. In either range,
\begin{equation*}
 \redpot^a_{K_j,Q_j}(0,0)\le\redpot^a_{K_j}(0,0)
 \le C\omega_{Aa}(j)\,.
\end{equation*}
For~$p<2$, this is the upper bound at the original height, since~$A=1$.

\emph{Step 2: The lower bounds.} Suppose that~$K_{j,q}\subset K_j$ and~$B_{8Lr_j(q)}\subset B_{2R_j}$. The initial times of~$Q_{j,q}$ and~$Q_j$ agree. Comparison and Lemma~\ref{lem:reception} then give, for every admissible~$q$ satisfying these containments,
\begin{equation*}
 \redpot^a_{K_j,Q_j}(0,0)
 \ge\redpot^a_{K_{j,q},Q_{j,q}}(0,0)
 \ge
 \begin{cases}
 q,&p>2,\\
 c_0q,&\pc<p<2.
 \end{cases}
\end{equation*}

For~$p>2$, uniformly over~$0<q\le a$,
\begin{equation*}
 \frac{8Lr_j(q)}{R_j}
 \le8La^{\frac{p-2}{p}}s_j^{\frac{1}{2p}}\longrightarrow0\,.
\end{equation*}
Thus~$K_{j,q}\subset K_j$ and~$B_{8Lr_j(q)}\subset B_{2R_j}$ for all sufficiently large~$j$ and every~$0<q\le a$. Taking the supremum over admissible~$q$ and using Step~1 gives
\begin{equation}\label{eq:potential-degenerate-comparison}
 \frac{\omega_a(j)}a\le\frac{\redpot^a_{K_j,Q_j}(0,0)}a
 \le\frac C a\omega_{Aa}(j)\,.
\end{equation}

For~$\pc<p<2$, small increments may give reference cylinders too wide for this comparison. We separate them using the threshold
\begin{equation}\label{eq:potential-singular-cutoff}
 \delta_j=\left(\frac{64^ps_j}{R_j^p}\right)^{\frac{1}{2-p}}
 =64^{\frac{p}{2-p}}2^{-\frac{j}{2(2-p)}}\,.
\end{equation}
Then~$\sum_j\delta_j<\infty$. An admissible increment~$q\ge\delta_j$ satisfies
\begin{equation*}
 r_j(q)^p=\frac{s_j}{q^{2-p}}\le\frac{R_j^p}{64^p}\,.
\end{equation*}
Hence~$K_{j,q}\subset K_j$ and~$B_{64r_j(q)}\subset B_{R_j}$, so comparison and Lemma~\ref{lem:reception} give~$\redpot^a_{K_j,Q_j}(0,0)\ge c_0q$. If~$\omega_a(j)>\delta_j$, take admissible increments above~$\delta_j$ tending to the supremum. This gives~$\redpot^a_{K_j,Q_j}(0,0)\ge c_0\omega_a(j)$. If~$\omega_a(j)\le\delta_j$, nonnegativity suffices. Together with Step~1, these estimates give
\begin{equation}\label{eq:potential-singular-comparison}
 \frac{c_0}{a}\bigl(\omega_a(j)-\delta_j\bigr)_+
 \le\frac{\redpot^a_{K_j,Q_j}(0,0)}a
 \le\frac C a\omega_a(j)\,.
\end{equation}
In particular,
\begin{equation*}
 \omega_a(j)\le\delta_j+\frac{\redpot^a_{K_j,Q_j}(0,0)}{c_0}\,.
\end{equation*}

\emph{Step 3: The series criteria.} The lower bounds show that divergence of every intrinsic series implies divergence of every potential series. For~$p>2$, the upper bound shows that convergence of an intrinsic series at height~$b$ gives convergence of the potential series at height~$b/A$. For~$\pc<p<2$, summability of~$\delta_j$ shows that the intrinsic and potential series have the same convergence or divergence at each fixed height. Theorem~\ref{thm:main} therefore proves~\eqref{eq:potential-wiener} in both ranges.

In both ranges, monotonicity in the obstacle height gives, for~$0<a_1\le a_2$,
\begin{equation*}
 \redpot^{a_1}_{K_j,Q_j}(0,0)\le \redpot^{a_2}_{K_j,Q_j}(0,0)\,.
\end{equation*}
For any~$a>0$, choose~$k$ with~$1/k\le a$. Divergence at height~$1/k$ implies divergence at~$a$. This proves the countable-height version. The equation and all the constructions are invariant under space--time translations.
\end{proof}

\subsection{Capacity differences and spatial averages}\label{subsec:capacity-proof}\label{subsec:potential-wiener}
The time-slice identity~\eqref{eq:capacity-time-slice} expresses spatial averages of the potentials as capacity differences. A summable averaging estimate then gives the equivalent Wiener series in Remark~\ref{cor:capacity-wiener}. For an obstacle supported on a time slice, the value of the balayage must be distinguished from its right trace: the right trace records the added height, whereas the canonical value on the slice is inherited from earlier times.

\begin{proposition}[Capacity of an added time slice]\label{lem:capacity-time-slice}
Let~$n\ge2$, $p\in(\pc,2)\cup(2,\infty)$, and~$Q=D\times(b,c)$, where~$D$ is smooth and bounded. Suppose~$b<T<c$, $K\Subset D\times(b,T)$ and~$F=\overline B_\rho(x_0)\Subset D$. For~$S=F\times\{T\}$ and every~$a>0$,
\begin{equation}\label{eq:capacity-time-slice}
 \capc_a(K\cup S,Q)-\capc_a(K,Q)
 =\int_F\bigl(a-\redpot^a_{K,Q}(x,T)\bigr)\,dx\,.
\end{equation}
In particular,~$\capc_a(S,Q)=a|F|$.
\end{proposition}
\begin{proof}
We construct a candidate by raising the right trace to~$a$ on~$F$ at time~$T$. Its time jump gives the additional Riesz measure. Two comparisons identify it with the potential of the enlarged obstacle, after which we pass from continuous obstacles to compact sets.

We first consider a continuous obstacle~$0\le\psi\le a$ compactly supported strictly before~$T$, and put~$v=\redpot_\psi$. Let~$\mathcal S_\tau$ denote homogeneous evolution in~$D$ with zero lateral values. Since~$p>2n/(n+2)$, the monotone evolution construction in~$W^{1,p}_0(D)\subset L^2(D)$ applies to bounded~$L^2$ initial data; see~\cite[Theorem~1.1 and Remark~1.2(2)]{Liu}. Testing the difference of two evolutions gives
\begin{equation*}
 \|\mathcal S_\tau f-\mathcal S_\tau g\|_{L^2(D)}
 \le\|f-g\|_{L^2(D)},
 \qquad \mathcal S_\tau g\longrightarrow g\text{ in }L^2(D)
 \quad(\tau\downarrow0)\,.
\end{equation*}
Set
\begin{equation}\label{eq:capacity-slice-datum}
 g=v(\cdot,T)+(a-v(\cdot,T))\ind_F\,.
\end{equation}
Define~$U=v$ before~$T$ and~$U(\cdot,t)=\mathcal S_{t-T}g$ after~$T$, with the canonical lower representative on the slice. Comparison gives~$0\le v\le U\le a$. Integration by parts in time gives
\begin{equation}\label{eq:capacity-slice-measure}
 \mu_U=\mu_v+(a-v(x,T))\ind_F(x)\,dx\otimes\delta_T\,.
\end{equation}
Here~$\delta_T$ is the unit point mass at~$T$. The homogeneous evolution after~$T$ adds no source. The initial datum is~$g$, which may be spatially discontinuous. The additional measure in~\eqref{eq:capacity-slice-measure} records the jump of~$U$ from~$v(\cdot,T)$ to~$g$ at time~$T$.

We identify~$U$ as the potential of~$\max\{\psi,a\ind_S\}$. Every nonnegative superparabolic majorant~$w$ of this obstacle satisfies~$w\ge v$ and~$w(\cdot,T)\ge g$. Choose continuous cutoffs~$0\le\chi_m\le\ind_F$ increasing to~$\ind_F$ off~$\partial F$. The continuous data
\begin{equation*}
 g_m=v(\cdot,T)+(a-v(\cdot,T))\chi_m
\end{equation*}
lie below~$w(\cdot,T)$ and converge to~$g$ in~$L^2(D)$. Lower semicontinuity of~$w$ supplies the boundary lower limits at~$T$. Comparison for the homogeneous evolution gives~$\mathcal S_{t-T}g_m\le w$; continuous initial-value solvability in both ranges is given in~\cite[Theorems~2.2--2.3]{GKM}. Contraction and canonical regularization give~$U\le w$. Taking the infimum and regularizing shows that~$U$ is no larger than the potential.

For the reverse inequality, we need majorants that reach~$a$ on~$F\times\{T\}$. The canonical value~$U(\cdot,T)=v(\cdot,T)$ need not do so. We therefore raise the potential slightly earlier and on the larger ball~$F_\varepsilon=\overline B_{\rho+\varepsilon}(x_0)\Subset D$. At time~$T-\tau_\varepsilon$, after the support of~$\psi$ has ended, replace~$v$ by~$a$ on~$F_\varepsilon$ and continue homogeneously. Denote this canonical concatenation by~$U_\varepsilon$. The time jump is nonnegative, so~$U_\varepsilon$ is superparabolic and majorizes~$\psi$.

If~$p>2$, choose~$\tau_\varepsilon\downarrow0$ with
\begin{equation*}
 0<\tau_\varepsilon<\kappa a^{2-p}(\varepsilon/2)^p\,.
\end{equation*}
For each~$x\in F$, the ball~$B_{\varepsilon/2}(x)$ lies in~$F_\varepsilon$. Lemma~\ref{lem:persistence} therefore gives~$U_\varepsilon(x,T)=a$ on~$F$. Put~$W_\varepsilon=U_\varepsilon$.

If~$\pc<p<2$, Lemma~\ref{lem:persistence} gives
\begin{equation*}
 U_\varepsilon(x,T)\ge a-C\left(\frac{\tau_\varepsilon}{(\varepsilon/2)^p}\right)^{\frac{1}{2-p}}
 \qquad(x\in F)\,.
\end{equation*}
Choose~$\tau_\varepsilon\downarrow0$ so that the loss is at most~$\varepsilon$, and set~$W_\varepsilon=\min\{a,U_\varepsilon+\varepsilon\}$. Addition of a constant and taking a finite minimum preserve superparabolicity~\cite[Lemma~2.7]{GKM}, so~$W_\varepsilon$ is a majorant of~$\max\{\psi,a\ind_S\}$. Admissible majorants in Definition~\ref{def:reduction} need not have zero boundary values; those values are properties of the potential. In either range,
\begin{equation*}
 0\le W_\varepsilon\le a,\qquad
 |W_\varepsilon-U_\varepsilon|\le\varepsilon\,.
\end{equation*}
The initial data for the evolution of~$U_\varepsilon$ after~$T-\tau_\varepsilon$ are
\begin{equation*}
 g_\varepsilon=v(\cdot,T-\tau_\varepsilon)
       +(a-v(\cdot,T-\tau_\varepsilon))\ind_{F_\varepsilon}\,.
\end{equation*}
They converge to~$g$ in~$L^2(D)$. Hence, for fixed~$t>T$,
\begin{equation*}
 \|U_\varepsilon(t)-U(t)\|_2 \le\|g_\varepsilon-g\|_2 +\|\mathcal S_{t-T+\tau_\varepsilon}g -\mathcal S_{t-T}g\|_2\longrightarrow0\,.
\end{equation*}
For fixed~$t<T$, we have~$U_\varepsilon(t)=v(t)$ eventually. The same almost-everywhere limit holds for~$W_\varepsilon$. The potential is bounded above by every~$W_\varepsilon$, so passing this comparison to the limit and taking canonical representatives gives the reverse inequality. At~$T$, the canonical value of~$U$ is the old value~$v(\cdot,T)$; its right trace is~$g$.

For compact~$K$, choose decreasing continuous outer obstacles~$\psi_m\downarrow a\ind_K$ supported in one compact set strictly before~$T$. Their potentials converge after regularization to~$\redpot^a_{K,Q}$. Interior regularity on a fixed neighborhood of~$F\times\{T\}$ that contains no source makes this convergence locally uniform there. Adding~$a\ind_S$ gives potentials converging to~$\redpot^a_{K\cup S,Q}$: if~$w$ majorizes~$a\ind_{K\cup S}$ and~$\eta>0$, lower semicontinuity and compactness give~$\psi_m\le w+\eta$ for all large~$m$, and the same majorant works on~$S$. Taking infima, regularizing and then sending~$\eta\downarrow0$ identifies the limit.

The convergence results for bounded supersolutions and potentials cited in Subsection~\ref{subsec:capacity-foundations} give weak convergence of the Riesz measures; see~\cite[Theorem~3.2, Lemma~4.2 and Theorem~5.7]{KKKP} and~\cite[Lemmas~2.2, 2.16--2.17 and~3.2--3.3]{MS}. Their supports lie in one common compact set. Testing weak convergence with a compactly supported function equal to one there therefore gives convergence of total masses. Passing~\eqref{eq:capacity-slice-measure} to the limit proves~\eqref{eq:capacity-time-slice}. Taking~$K=\varnothing$ gives~$\capc_a(S,Q)=a|F|$.
\end{proof}

For~$K_j,Q_j,s_j$ from~\eqref{eq:potential-scales}--\eqref{eq:potential-test-sets}, put~$\rho_j=s_j^{\frac{2}{p}}$ and~$S_j=\overline B_{\rho_j}\times\{0\}$, and define
\begin{equation}\label{eq:capacity-difference-term}
 \mathcal D_a(j)=
 \frac{\capc_a(K_j,Q_j)+\capc_a(S_j,Q_j)
              -\capc_a(K_j\cup S_j,Q_j)}{a|B_{\rho_j}|}\,.
\end{equation}
All capacities in this expression use the same height and reference cylinder. The slice~$S_j$ need not lie in~$E$.

For every~$a>0$, the capacity difference satisfies
\begin{equation}\label{eq:capacity-average-identity}
 \capc_a(S_j,Q_j)=a|B_{\rho_j}|,
 \qquad
 \mathcal D_a(j)=\frac1{a|B_{\rho_j}|}
       \int_{B_{\rho_j}}\redpot^a_{K_j,Q_j}(x,0)\,dx\in[0,1]\,.
\end{equation}

\begin{remark}[Capacity differences]\label{cor:capacity-wiener}
Under the hypotheses of Theorem~\ref{thm:potential-wiener}, with~$\mathcal D_a(j)$ as in~\eqref{eq:capacity-difference-term},
\begin{equation}\label{eq:capacity-wiener}
 0\text{ is regular}
 \quad\Longleftrightarrow\quad
 \sum_{j=1}^{\infty}\mathcal D_a(j)=\infty
 \quad\text{for every }a>0\,.
\end{equation}
It suffices to test~$a=1/k$, $k\in\N$. The exact identity~\eqref{eq:capacity-average-identity} and the summable averaging estimate~\eqref{eq:capacity-summable-average} give this equivalence.
\end{remark}

Remark~\ref{cor:capacity-wiener} is therefore an exact reformulation of Theorem~\ref{thm:potential-wiener}. To estimate~$\mathcal D_a(j)$, one needs control of the numerator in~\eqref{eq:capacity-difference-term}; separate multiplicative estimates of its three terms can have errors larger than their difference.

\begin{proof}[Proof of Remark~\ref{cor:capacity-wiener}]
Apply Proposition~\ref{lem:capacity-time-slice} with~$K=K_j$, $F=\overline B_{\rho_j}$ and~$T=0$. The set~$K_j$ ends at~$-2s_j$, strictly before~$S_j$. Subtracting the capacity increment from~$\capc_a(S_j,Q_j)$ proves~\eqref{eq:capacity-average-identity}.

It remains to compare the average with the pointwise value. Put
\begin{equation*}
 \ell_j(a)=(a^{p-2}s_j)^{\frac{1}{p}},\qquad
 v_j(y,\tau)=a^{-1}\redpot^a_{K_j,Q_j}(\ell_j(a)y,s_j\tau)\,.
\end{equation*}
For fixed~$a>0$ and all sufficiently large~$j$, one has~$4\ell_j(a)<2R_j$ and~$\rho_j<\ell_j(a)$. Thus
\begin{equation*}
 (v_j)_\tau=\Dp v_j,\qquad 0\le v_j\le1,
 \qquad\text{in }B_4\times(-3/2,1/2)\,.
\end{equation*}
The interior H\"older estimates in~\cite[Theorem~1 and Remark~0.2]{DiB86} for~$p>2$, and~\cite[Theorem~2.4]{CV} for~$\pc<p<2$, give constants~$C>0$ and~$\alpha\in(0,1)$ depending only on~$n,p$ such that
\begin{equation*}
 |v_j(y,0)-v_j(0,0)|\le C|y|^\alpha,
 \qquad |y|\le1\,.
\end{equation*}
Apply the cited interior H\"older estimates to~$v_j+1$ on~$B_4\times(-3/2,1/2)$. The supremum of~$v_j+1$ lies between~$1$ and~$2$, and the intrinsic distance from the smaller cylinder to the boundary of the fixed cylinder is bounded below. No positive lower bound for~$v_j$ is needed. Averaging over~$B_{\rho_j/\ell_j(a)}$ gives
\begin{equation}\label{eq:capacity-summable-average}
 |\mathcal D_a(j)-\frac{\redpot^a_{K_j,Q_j}(0,0)}a|
 \le C\left(\frac{\rho_j}{\ell_j(a)}\right)^\alpha
 =Ca^{-\frac{\alpha(p-2)}{p}}2^{-\frac{j\alpha}{p}}\,.
\end{equation}
For every fixed~$a$, the error is summable. Thus the two series have the same convergence or divergence at that height. Theorem~\ref{thm:potential-wiener} proves~\eqref{eq:capacity-wiener}. Its countable-height version and the same summable comparison prove the final statement of the remark.
\end{proof}

\section{Resolutivity}\label{sec:r-resolutivity}
We prove Theorem~\ref{thm:resolutivity} for continuous data on the full topological boundary, with the Perron classes of Definition~\ref{def:perron}. Throughout this section,~$n\ge1$ and~$1<p<\infty$. We construct upper and lower Perron competitors by positive and negative corrections of an approximation to the boundary datum, then use an energy estimate to identify their interior limits.

Fix~$\Theta$ as in Theorem~\ref{thm:resolutivity}. After translating time, choose a ball~$D_0\subset\R^n$ and~$T>0$ such that~$\overline\Theta\Subset Q:=D_0\times(0,T)$. Lemma~\ref{lem:r-data} approximates~$f$ uniformly on~$\partial\Theta$ by restrictions of Lipschitz functions~$F$ with~$F_t-\Dp F\in L^\infty(Q)$. Keep~$F$ fixed and set
\begin{equation*}
 h=1+\|F_t-\Dp F\|_{L^\infty(Q)},\qquad q=-F_t+\Dp F-h\le-1.
\end{equation*}
We first seek a nonnegative correction~$H$ that makes~$F+H$ a supersolution, then a nonpositive correction~$w$ that cancels its right-hand side:
\begin{align*}
 (F+H)_t-\Dp(F+H)&=F_t-\Dp F+h\ge1,\\
 (F+H+w)_t-\Dp(F+H+w)&=0\quad\text{in }\Theta.
\end{align*}

All correction problems are solved on~$Q$ with zero initial and lateral values. Upper and lower obstacles make the finite corrections vanish near~$\partial\Theta$, while the obstacle constraints become inactive on interior compact sets. The positive corrections satisfy~$H_j\uparrow H$ and~$H_j\ge0$. For each~$j$, we construct negative corrections~$w_{j,m}\le0$. Letting~$j\to\infty$ with~$m$ fixed gives~$w_m$, while letting~$m\to\infty$ with~$j$ fixed gives~$\widetilde w_j$. The remaining obstacle reactions give
\begin{equation}\label{eq:r-competitors}
 \begin{aligned}
 U_m&=F+H+w_m,&\qquad (U_m)_t-\Dp U_m&\ge0,\\
 V_j&=F+H_j+\widetilde w_j,&\qquad (V_j)_t-\Dp V_j&\le0.
 \end{aligned}
\end{equation}
For fixed~$m$,~$w_m$ vanishes near~$\partial\Theta$, so~$U_m\ge F$ there. For fixed~$j$,~$H_j$ vanishes near~$\partial\Theta$, so~$V_j\le F$ there. The canonical representatives therefore belong to the Perron classes on the full boundary, including its terminal portions. No boundary trace of the limiting corrections~$H$ and~$w$ is needed.

The energy estimate and Lemma~\ref{lem:r-two-limits} give~$\widetilde w_j\to w=\lim_m w_m$ by first letting~$m\to\infty$ with~$j$ fixed, then letting~$j\to\infty$. Thus~$U_m$ and a subsequence of~$V_j$ have the common almost-everywhere limit~$F+H+w$. The Perron inequalities identify the Perron solutions, and uniform stability in the boundary datum extends the conclusion to every continuous~$f$.

\subsection{Approximation of the boundary data}\label{sec:r-data}
Even if~$F$ is smooth,~$\Delta_pF$ may be unbounded when~$p<2$. For example, $F(x,t)=x_1^2$ has
\begin{equation*}
 \Delta_pF=2^{p-1}(p-1)|x_1|^{p-2}
\end{equation*}
away from the hyperplane $x_1=0$, and this is unbounded in the singular range. We therefore approximate the datum by an elliptic resolvent instead.

\begin{lemma}[Approximation by data with bounded source]\label{lem:r-data}
Fix a cylinder $Q=D_0\times(0,T)$ with $D_0$ a ball and $\overline\Theta\Subset Q$. The restrictions to $\partial\Theta$ of functions
\begin{equation}\label{eq:r-data-class}
 F\in W^{1,\infty}(Q)\cap C(\overline Q),\qquad F_t-\Dp F\in L^\infty(Q)
\end{equation}
are uniformly dense in $C(\partial\Theta)$. Here $W^{1,\infty}$ includes both the spatial and the time derivatives, and $F_t-\Dp F$ is interpreted distributionally.
\end{lemma}
\begin{proof}
By the real Stone--Weierstrass theorem, polynomial restrictions are uniformly dense in~$C(\partial\Theta)$. It therefore suffices to approximate a smooth function defined near~$\overline Q$. Choose a flat spatial torus~$\mathbb T^n$ containing~$\overline{D_0}$ in one coordinate chart. Multiplying by a spatial cutoff supported in that chart and equal to one near~$\overline{D_0}$ gives~$G\in C^\infty(\mathbb T^n\times[0,T])$ with the same values near~$\overline Q$.

For each $t$ and $\lambda>0$, minimize
\begin{equation}\label{eq:r-elliptic-functional}
 \frac12\int_{\mathbb T^n}|v-G(\cdot,t)|^2
 +\frac\lambda p\int_{\mathbb T^n}|\nabla v|^p
\end{equation}
over~$W^{1,p}(\mathbb T^n)\cap L^2(\mathbb T^n)$. A minimizing sequence is bounded in~$L^2$, and its gradients are bounded in~$L^p$. The~$L^2$ bound controls the mean, so Poincar\'e's inequality bounds the full~$W^{1,p}$ norm, also when~$p>2$. Weak compactness and weak lower semicontinuity give a minimizer~$F_\lambda(\cdot,t)$. Strict convexity gives uniqueness. The Euler equation for~$F_\lambda$ is
\begin{equation}\label{eq:r-resolvent}
 F_\lambda-\lambda\Delta_pF_\lambda=G(\cdot,t).
\end{equation}

Let~$u_i$ solve~\eqref{eq:r-resolvent} with bounded data~$g_i$. Subtract the weak equations and test by~$(u_1-u_2-\|g_1-g_2\|_\infty)_+$, which belongs to~$W^{1,p}\cap L^2$. Monotonicity gives
\begin{equation*}
 \int_{\mathbb T^n}(u_1-u_2-\|g_1-g_2\|_\infty)_+^2\le0.
\end{equation*}
Interchanging~$u_1,u_2$ gives
\begin{equation}\label{eq:r-resolvent-contraction}
 \|u_1-u_2\|_\infty\le\|g_1-g_2\|_\infty.
\end{equation}
Comparison with constants also gives~$\|F_\lambda\|_\infty\le\|G\|_\infty$.

Spatial translation invariance and \eqref{eq:r-resolvent-contraction} show that
\begin{equation*}
 \|F_\lambda(\cdot+y,t)-F_\lambda(\cdot,t)\|_\infty
 \le \|G(\cdot+y,t)-G(\cdot,t)\|_\infty.
\end{equation*}
By difference quotients,~$F_\lambda(\cdot,t)$ has a Lipschitz representative with constant at most~$\|\nabla G\|_\infty$. Similarly,
\begin{equation*}
 \|F_\lambda(\cdot,t)-F_\lambda(\cdot,s)\|_\infty
 \le\|G(\cdot,t)-G(\cdot,s)\|_\infty
 \le\|G_t\|_\infty|t-s|.
\end{equation*}
For these representatives, the last inequality holds pointwise. Thus~$F_\lambda$ is Lipschitz in~$(x,t)$, uniformly in~$\lambda$, and~$\|(F_\lambda)_t\|_\infty\le\|G_t\|_\infty$.

Comparison of \eqref{eq:r-elliptic-functional} at $F_\lambda$ and $G$ gives
\begin{equation}\label{eq:r-resolvent-L2}
 \sup_t\|F_\lambda(\cdot,t)-G(\cdot,t)\|_2^2
 \le\frac{2\lambda}{p}\sup_t\int_{\mathbb T^n}|\nabla G(\cdot,t)|^p.
\end{equation}
Since~$F_\lambda(\cdot,t)$ is uniformly Lipschitz, estimate~\eqref{eq:r-resolvent-L2} gives~$F_\lambda\to G$ uniformly on~$\mathbb T^n\times[0,T]$. Indeed, if~$|F_\lambda-G|\ge\varepsilon$ at one point, then~$|F_\lambda-G|\ge\varepsilon/2$ on a spatial ball whose radius is bounded below independently of~$t,\lambda$. This contradicts~\eqref{eq:r-resolvent-L2} as~$\lambda\to0$.
Finally \eqref{eq:r-resolvent} gives, on $Q$,
\begin{equation}\label{eq:r-bounded-source}
 (F_\lambda)_t-\Dp F_\lambda=(F_\lambda)_t-\frac{F_\lambda-G}{\lambda}\in L^\infty(Q).
\end{equation}
Integrating~\eqref{eq:r-resolvent} in~$t$ shows that it also holds in space-time distributions. Thus~$F_\lambda$ satisfies~\eqref{eq:r-data-class}, and its restriction to~$\partial\Theta$ approximates that of~$G$ uniformly.
\end{proof}

\begin{remark}\label{rem:r-fixed-data}
We keep~$F$ fixed while approximating the obstacles and exhausting~$\Theta$. The bound for~$\|(F_\lambda)_t-\Dp F_\lambda\|_\infty$ may depend on~$\lambda$. We approximate~$f$ uniformly only after taking these limits.
\end{remark}

\subsection{Obstacle solutions and stability}\label{sec:r-obstacles}
Adding a correction~$v$ to~$F$ changes the spatial flux by~$A(\nabla F+\nabla v)-A(\nabla F)$. We therefore consider the shifted flux~$A(a+\nabla v)-A(a)$ for~$a\in L^p(Q;\R^n)$. For a prescribed right-hand side~$q$ and obstacle~$\psi$, let~$W_q[a,\psi]$ be the solution constructed by smooth approximation in Appendix~\ref{sec:r-flux}. The estimates below compare~$W_q[a,\psi]$ as~$a$ varies, with~$q,\psi$ fixed. Norms without a displayed domain are over~$Q$, or over~$D_0$ when applied to a fixed time slice; the constants~$C_p,c_p>0$ depend only on~$p$.

\begin{lemma}[Obstacle solutions and stability]\label{lem:r-obstacle}
Let~$a\in L^p(Q;\R^n)$,~$q\in L^\infty(Q)$ with~$q\le0$, and~$\psi\in C_c^\infty(Q)$ with~$\psi\le0$.
\begin{enumerate}[(i)]
\item \textup{Existence and bounds.} There is a solution~$w=W_q[a,\psi]$ with initial value zero and
\begin{equation*}
 w\in C([0,T];L^2(D_0))\cap L^p(0,T;W_0^{1,p}(D_0))\,.
\end{equation*}
It satisfies
\begin{gather}
 \psi\le w\le0,\qquad w\ge-t\|q\|_\infty\,,\label{eq:r-obstacle-bounds}\\
 w_t-\div\bigl(A(a+\nabla w)-A(a)\bigr)\ge q\,.\label{eq:r-obstacle-super}
\end{gather}
The differential inequality holds in distributions, with equality on every open set where~$\psi=-M$ for a constant~$M>T\|q\|_\infty$. For fixed~$a,q$,~$W_q[a,\psi]$ is nondecreasing in~$\psi$. Moreover,
\begin{equation}\label{eq:r-obstacle-energy}
 \sup_t\|w(t)\|_2^2+\|\nabla w\|_p^p
 \le C_p\bigl(\|a\|_p^p+T\|q\|_\infty^2|Q|\bigr)\,.
\end{equation}
For~$p>2$, the term~$\|a\|_p^p$ on the right may be omitted.

\item \textup{Stability.} For~$b\in L^p(Q;\R^n)$ and the same~$q,\psi$, let~$w_a=W_q[a,\psi]$,~$w_b=W_q[b,\psi]$, and set
\begin{equation*}
 \mathcal E_{a,b}=A(a+\nabla w_b)-A(a)-A(b+\nabla w_b)+A(b)\,.
\end{equation*}
If~$1<p\le2$ and~$\|a\|_p+\|b\|_p\le B$, then
\begin{equation}\label{eq:r-common-singular}
 \sup_t\|w_a(t)-w_b(t)\|_2^2+\|\nabla w_a-\nabla w_b\|_p^2
 \le C\|\mathcal E_{a,b}\|_{p'}^2\,,
\end{equation}
where~$C=C(p,Q,\|q\|_\infty,B)$. If~$p>2$, then
\begin{equation}\label{eq:r-common-degenerate}
 \sup_t\|w_a(t)-w_b(t)\|_2^2+\|\nabla w_a-\nabla w_b\|_p^p
 \le C_p\|\mathcal E_{a,b}\|_{p'}^{p'}\,.
\end{equation}
These constants are independent of~$\psi$ and its derivatives.
\end{enumerate}
\end{lemma}

The variational theory of~\cite[Theorem~2.2]{GMTV} gives the solutions when~$a$ is smooth. In Appendix~\ref{sec:r-flux}, we check its hypotheses and compare two solutions by testing their variational inequalities. We then approximate~$a\in L^p(Q;\R^n)$ by smooth functions.

\begin{remark}\label{rem:r-reactions}
For~$a\in L^p(Q;\R^n)$ and~$w=W_q[a,\psi]$, the distribution
\begin{equation*}
 \mu:=w_t-\div\bigl(A(a+\nabla w)-A(a)\bigr)-q
\end{equation*}
is nonnegative. We use the identity~$\mu_\delta(w_\delta-\psi)=0$ from~\eqref{eq:r-regularized-contact} only for smooth~$a$, with the flux regularized when~$1<p\le2$; in this case~$\mu_\delta\in L^\infty(Q)$. The fluxes converge strongly, so we can pass to the limit in the differential inequality and in the equation where~$\psi=-M$ and~$M>T\|q\|_\infty$. We derive the comparison estimates from the variational inequalities, without pairing~$\mu$ with a difference of solutions.
\end{remark}

\subsection{Convergence of gradients and flux errors}\label{sec:r-limits}
We first obtain strong~$L^p$ convergence of the gradients on compact subsets where the obstacle constraints are inactive. We then control the flux error when the background fields satisfy~$a_j\to a$ almost everywhere and are bounded in~$L^p$. These estimates identify the limits of the negative corrections constructed using~$H_j$ and~$H$ by taking the obstacle limit first and the background limit second.

\paragraph{Estimates for~$A$.}
For~$1<p\le2$ and~$x,y\in\R^n$,
\begin{align}
 |A(x)-A(y)|&\le C_p|x-y|^{p-1},\label{eq:r-holder-unregularized}\\
 (A(x)-A(y))\cdot(x-y)
 &\ge c_p\frac{|x-y|^2}{(|x|+|y|)^{2-p}}.\label{eq:r-monotonicity-unregularized}
\end{align}
These are the unregularized cases of~\eqref{eq:r-holder-flux} and~\eqref{eq:r-weighted-monotonicity}, proved in Appendix~\ref{sec:r-flux}. Zero divided by zero is interpreted as zero; for~$p=2$ the denominator is one.

For~$a,b,\xi\in\R^n$, grouping the flux differences first by~$a,b$ and then by the increments~$\xi$ gives
\begin{align}
 |A(b+\xi)-A(b)-A(a+\xi)+A(a)|&\le C_p|b-a|^{p-1},\label{eq:r-singular-background}\\
 |A(b+\xi)-A(b)-A(a+\xi)+A(a)|&\le C_p|\xi|^{p-1}.\label{eq:r-singular-gradient}
\end{align}

For~$d\in L^p(Q;\R^n)$ and~$S\in L^p(Q)$ with~$S\ge|d|$, we use
\begin{equation}\label{eq:r-weighted-holder}
 \|d\|_p^2
 \le \left(\iint_Q\frac{|d|^2}{S^{2-p}}\,dx\,dt\right)\|S\|_p^{2-p}.
\end{equation}
Here zero divided by zero is interpreted as zero.
For $p<2$, apply H\"older with exponents $2/p$ and $2/(2-p)$ to
$|d|^p=(|d|^2/S^{2-p})^{\frac{p}{2}}S^{\frac{p(2-p)}{2}}$ and raise the result to $2/p$. For $p=2$ the inequality is an identity.

For~$p>2$ and~$a,b,\xi,\eta\in\R^n$, we instead use
\begin{align}
 (A(a+\xi)-A(a+\eta))\cdot(\xi-\eta)&\ge c_p|\xi-\eta|^p,\label{eq:r-deg-monotonicity}\\
 |A(a+\xi)-A(a)|&\le C_p|\xi|(|a|+|\xi|)^{p-2}.\label{eq:r-deg-growth}
\end{align}
These follow by integrating~$DA$ along a segment. In particular,~$\bigl(A(a+\xi)-A(a)\bigr)\cdot\xi\ge c_p|\xi|^p$.

Grouping the flux differences first by~$a,b$ and then by~$\xi$ gives
\begin{align}
 |A(b+\xi)-A(b)-A(a+\xi)+A(a)|&\le C_p|b-a|(|a|+|b|+|\xi|)^{p-2},\label{eq:r-deg-background}\\
 |A(b+\xi)-A(b)-A(a+\xi)+A(a)|&\le C_p|\xi|(|a|+|b|+|\xi|)^{p-2}.\label{eq:r-deg-gradient}
\end{align}

\begin{lemma}[Local convergence of gradients]\label{lem:r-local}
Let $U\subset Q$ be open, $a\in L^p(Q;\R^n)$ and $r\in L^\infty(Q)$. Suppose $v_k\to v$ in $L^p(Q)\cap L^2(Q)$ and $\nabla v_k$ is bounded in $L^p(Q)$. Assume that on every cylinder compactly contained in $U$, all sufficiently large $k$ satisfy
\begin{equation*}
 (v_k)_t-\div \bigl(A(a+\nabla v_k)-A(a)\bigr)=r.
\end{equation*}
Then $\nabla v_k\to \nabla v$ strongly in $L^p_{\rm loc}(U)$, and
$v_t-\div \bigl(A(a+\nabla v)-A(a)\bigr)=r$ in $U$.
\end{lemma}
\begin{proof}
Work first in a compact subcylinder where both equations hold. For a smooth compactly supported cutoff~$\eta$ there, set~$J=A(a+\nabla v_k)-A(a+\nabla v_\ell)$. Subtracting the equations gives~$(v_k-v_\ell)_t-\div J=0$ in distributions. Testing with $\eta^2(v_k-v_\ell)$ gives
\begin{equation}\label{eq:r-local-test}
 \iint\eta^2J\cdot (\nabla v_k-\nabla v_\ell)\,dx\,dt
 =\iint\eta\eta_t (v_k-v_\ell)^2\,dx\,dt-2\iint\eta (v_k-v_\ell)J\cdot \nabla \eta\,dx\,dt.
\end{equation}
To justify the test, convolve~$v_k-v_\ell$ and~$J$ in time inside a slightly larger cylinder. Test the convolved equation with~$\eta^2(v_k-v_\ell)_h$. Since~$\partial_t(v_k-v_\ell)_h\in L^2_{\rm loc}$, we can integrate by parts in time. The time mollifications converge strongly in~$L^2$,~$L^p$ and~$L^{p'}$, so we can let~$h\to0$.

Estimates~\eqref{eq:r-holder-unregularized} and~\eqref{eq:r-deg-growth} give a uniform~$L^{p'}$ bound for~$J$, since~$a$ is fixed and~$(\nabla v_k)$ is bounded in~$L^p$. Hence the right side of \eqref{eq:r-local-test} tends to zero. For~$p>2$,~\eqref{eq:r-deg-monotonicity} shows that~$(\nabla v_k)$ is Cauchy in~$L^p$ locally. For $1<p\le2$, use \eqref{eq:r-monotonicity-unregularized} and then \eqref{eq:r-weighted-holder} on a compact set where $\eta=1$, with $S=|a+\nabla v_k|+|a+\nabla v_\ell|$. Its $L^p$ norm is uniformly bounded. In either case,~$\nabla v_k\to\nabla v$ in~$L^p_{\rm loc}(U)$, since~$\nabla v$ is the distributional limit. The fluxes converge in~$L^{p'}_{\rm loc}(U)$, so~$v$ satisfies the equation. Covering compact subsets of $U$ by finitely many such cylinders completes the proof.
\end{proof}

The sequences~$(a_j)$ and~$(\xi_m)$ in the flux error may both converge only almost everywhere, with bounded~$L^p$ norms. These bounds allow concentration. We first let~$m\to\infty$ with~$j$ fixed, then let~$j\to\infty$ with~$\xi$ fixed. At each step, a fixed integrable factor gives uniform integrability.

\begin{lemma}[Successive limits of the flux error]\label{lem:r-two-limits}
Let~$a_j\to a$ and~$\xi_m\to\xi$ almost everywhere in~$Q$, with both sequences bounded in~$L^p(Q;\R^n)$. Set
\begin{equation*}
 \mathcal E_j(\eta):=A(a_j+\eta)-A(a_j)-A(a+\eta)+A(a).
\end{equation*}
Then, in~$L^{p'}(Q)$,
\begin{equation*}
 \begin{aligned}
 \mathcal E_j(\xi_m)&\longrightarrow\mathcal E_j(\xi)
 &&\text{as }m\to\infty\quad\text{for each fixed }j,\\
 \mathcal E_j(\xi)&\longrightarrow0
 &&\text{as }j\to\infty.
 \end{aligned}
\end{equation*}
\end{lemma}
\begin{proof}
Fix~$j$ and let~$S\subset Q$ be measurable. For~$1<p\le2$,~\eqref{eq:r-singular-background} gives
\begin{equation*}
 |\mathcal E_j(\xi_m)|^{p'}\le C_p|a_j-a|^p,
\end{equation*}
with an integrable right side independent of~$m$. For~$p>2$,~\eqref{eq:r-deg-background} and H\"older give
\begin{equation}\label{eq:r-vitali-first}
 \int_S|\mathcal E_j(\xi_m)|^{p'}
 \le C_p\left(\int_S|a_j-a|^p\right)^{\frac{1}{p-1}}
 \left(\int_S(|a_j|+|a|+|\xi_m|)^p\right)^{\frac{p-2}{p-1}}.
\end{equation}
The second factor is bounded uniformly in~$m$, and the first tends to zero as~$|S|\to0$. Thus~$|\mathcal E_j(\xi_m)|^{p'}$ is uniformly integrable in either range. Since~$\mathcal E_j(\xi_m)\to\mathcal E_j(\xi)$ almost everywhere by continuity of~$A$, Vitali gives the first limit.

For the second limit and~$1<p\le2$,~\eqref{eq:r-singular-gradient} gives
\begin{equation*}
 |\mathcal E_j(\xi)|^{p'}\le C_p|\xi|^p.
\end{equation*}
For~$p>2$,~\eqref{eq:r-deg-gradient} and H\"older give
\begin{equation}\label{eq:r-vitali-second}
 \int_S|\mathcal E_j(\xi)|^{p'}
 \le C_p\left(\int_S|\xi|^p\right)^{\frac{1}{p-1}}
 \left(\int_S(|a_j|+|a|+|\xi|)^p\right)^{\frac{p-2}{p-1}}.
\end{equation}
Here the second factor is uniformly bounded in~$j$, and the first tends to zero as~$|S|\to0$. Hence~$|\mathcal E_j(\xi)|^{p'}$ is uniformly integrable in either range. To see the almost-everywhere cancellation, write
\begin{equation*}
 \begin{aligned}
 \mathcal E_j(\xi)
 &=\bigl[A(a_j+\xi)-A(a_j)\bigr]-\bigl[A(a+\xi)-A(a)\bigr]\\
 &=\bigl[A(a_j+\xi)-A(a+\xi)\bigr]-\bigl[A(a_j)-A(a)\bigr].
 \end{aligned}
\end{equation*}
Since~$a_j\to a$ almost everywhere and~$\xi$ is fixed, continuity of~$A$ makes both brackets in the second line tend to zero almost everywhere. Vitali now gives the second limit. At~$p=2$,~$\mathcal E_j(\eta)=0$ for every~$\eta$.
\end{proof}

\subsection{Construction of the Perron competitors}\label{sec:r-perron-proof}
\begin{proof}[Proof of Theorem~\ref{thm:resolutivity}]
Both~$\upperH_\Theta$ and~$\lowerH_\Theta$ preserve order and commute with addition of constants. Hence
\begin{equation}\label{eq:r-perron-stability}
 \begin{aligned}
 \|\upperH_\Theta f-\upperH_\Theta g\|_{L^\infty(\Theta)}
 &\le\|f-g\|_{C(\partial\Theta)},\\
 \|\lowerH_\Theta f-\lowerH_\Theta g\|_{L^\infty(\Theta)}
 &\le\|f-g\|_{C(\partial\Theta)}.
 \end{aligned}
\end{equation}
By Lemma~\ref{lem:r-data}, it suffices to prove the theorem for~$f=F|_{\partial\Theta}$ with~$F$ satisfying~\eqref{eq:r-data-class} on the fixed cylinder~$Q$. Keep~$F$ fixed throughout the construction.

\emph{Step 1: Construction of~$H_j$.}
Set
\begin{equation}\label{eq:r-sources}
 h=1+\|F_t-\Dp F\|_{L^\infty(Q)},\qquad
 q=-F_t+\Dp F-h\le-1,
\end{equation}
and choose
\begin{equation*}
 M>1+T\max\{h,\|q\|_\infty\}.
\end{equation*}
Choose smooth cutoffs $0\le\rho_j\le\rho_{j+1}\le1$ with $\rho_j\in C_c^\infty(\Theta)$, such that every compact subset of $\Theta$ has a neighborhood where $\rho_j=1$ eventually. Explicitly, take a compact exhaustion $K_j$, choose $\eta_j\in C_c^\infty(\Theta)$ with $0\le\eta_j\le1$ and $\eta_j=1$ near $K_j$, and put
$\rho_j=1-\prod_{i=1}^j(1-\eta_i)$. Extend these cutoffs by zero to $Q$.

Since~$\nabla F\in L^p(Q;\R^n)$ and~$-M\rho_j\in C_c^\infty(Q)$ is nonpositive, Lemma~\ref{lem:r-obstacle} applies. Set
\begin{equation*}
 H_j=-W_{-h}[-\nabla F,-M\rho_j].
\end{equation*}
Since~$A$ is odd, the bounds and monotonicity of~$W_{-h}[-\nabla F,\psi]$ give
\begin{align}
 0\le H_j\le M\rho_j,\qquad H_j&\le ht,\qquad H_j\le H_{j+1},\label{eq:r-H-bounds}\\
 (H_j)_t-\div \bigl(A(\nabla F+\nabla H_j)-A(\nabla F)\bigr)&=h-\nu_j,\qquad \nu_j\ge0.\label{eq:r-H-equation}
\end{align}
Here $\nu_j$ is a positive distribution on $Q$. The sequence~$(\nabla H_j)$ is bounded in~$L^p(Q)$. Thus~$H_j\uparrow H$ almost everywhere and in every finite~$L^r(Q)$. By weak compactness,~$H\in L^p(0,T;W_0^{1,p}(D_0))$.

On each compact subcylinder of~$\Theta$,~$\rho_j=1$ for all sufficiently large~$j$. Since~$H_j\le hT<M$, we have~$-H_j>-M$ there. Lemma~\ref{lem:r-local} gives
\begin{equation}\label{eq:r-H-local}
 \nabla H_j\to \nabla H\text{ in }L^p_{\rm loc}(\Theta),\qquad
 H_t-\div \bigl(A(\nabla F+\nabla H)-A(\nabla F)\bigr)=h\text{ in }\Theta.
\end{equation}
Every $H_j$ vanishes almost everywhere outside a compact subset of $\Theta$, and $H=0$ almost everywhere on $Q\setminus\Theta$. The gradient of a Sobolev function vanishes almost everywhere on its zero set. Apply this to~$H_j(\cdot,t)$ and~$H(\cdot,t)$ for almost every~$t$, and take a diagonal subsequence in~\eqref{eq:r-H-local}. We obtain
\begin{equation}\label{eq:r-background-ae}
 a_j:=\nabla F+\nabla H_j\to a:=\nabla F+\nabla H\text{ a.e. in }Q,
 \qquad \sup_j\|a_j\|_p<\infty.
\end{equation}
Both~$\nabla H_j$ and~$\nabla H$ vanish almost everywhere on~$Q\setminus\Theta$, including any part of~$\partial\Theta$ with positive measure. Relabel this increasing subsequence of~$H_j$ and its associated cutoffs if necessary.

\emph{Step 2: Convergence of~$w_{j,m}$ and~$w_m$.}
For the fixed right-hand side~$q$ in~\eqref{eq:r-sources}, set
\begin{equation*}
 w_{j,m}=W_q[a_j,-M\rho_m],\qquad
 w_m=W_q[a,-M\rho_m].
\end{equation*}
For each~$j$,~$w_{j,m}$ decreases in~$m$, as does~$w_m$. Both satisfy~$-T\|q\|_\infty\le w_{j,m},w_m\le0$. By~\eqref{eq:r-background-ae},~$\|\nabla w_{j,m}\|_p$ and~$\|\nabla w_m\|_p$ are bounded uniformly in~$j,m$. Write~$\widetilde w_j=\lim_{m\to\infty}w_{j,m}$ and~$w=\lim_{m\to\infty}w_m$ almost everywhere. Here~$a_j$ remains fixed in the definition of~$\widetilde w_j$. For each fixed~$j$, $w_{j,m}\to \widetilde w_j$ and~$w_m\to w$ in every finite~$L^r(Q)$. Weak convergence of the gradients gives~$\widetilde w_j,w\in L^p(0,T;W_0^{1,p}(D_0))$.

For fixed~$m$, the second limit in Lemma~\ref{lem:r-two-limits}, applied with~$\xi=\nabla w_m$, and the stability estimates~\eqref{eq:r-common-singular}--\eqref{eq:r-common-degenerate} give~$w_{j,m}\to w_m$ in~$L^2(Q)$ as~$j\to\infty$.

On compact subcylinders of~$\Theta$,~$\rho_m=1$ for all sufficiently large~$m$, while~$w_{j,m},w_m\ge-T\|q\|_\infty>-M$. Lemma~\ref{lem:r-local} therefore gives
\begin{equation}\label{eq:r-w-equations}
 (\widetilde w_j)_t-\div \bigl(A(a_j+\nabla \widetilde w_j)-A(a_j)\bigr)=q,\qquad
 w_t-\div \bigl(A(a+\nabla w)-A(a)\bigr)=q
 \quad\text{in }\Theta,
\end{equation}
As~$m\to\infty$,~$\nabla w_m\to\nabla w$ in~$L^p_{\rm loc}(\Theta)$ and, for each fixed~$j$,~$\nabla w_{j,m}\to\nabla\widetilde w_j$ in that space.

On~$\{\rho_m=0\}$,~\eqref{eq:r-obstacle-bounds} gives~$w_m=w_{j,m}=0$ almost everywhere. Hence~$w,\widetilde w_j$ vanish almost everywhere on~$Q\setminus\Theta$. The gradients~$\nabla w_m,\nabla w$ also vanish almost everywhere on~$Q\setminus\Theta$, since a Sobolev function has zero gradient almost everywhere on its zero set. Taking a diagonal subsequence gives
\begin{equation}\label{eq:r-w-ae}
 \nabla w_m\to \nabla w\text{ a.e. in }Q,\qquad
 \sup_m\|\nabla w_m\|_p<\infty.
\end{equation}
For every fixed $j$, $\nabla w_{j,m}\rightharpoonup \nabla \widetilde w_j$ globally in $L^p(Q)$. Indeed,~$w_{j,m}\to\widetilde w_j$ in~$L^p(Q)$, so every weak limit of~$\nabla w_{j,m}$ equals~$\nabla\widetilde w_j$. The~$L^p(Q)$ bound then gives convergence of the full sequence.

\emph{Step 3: Convergence of~$\widetilde w_j$ to~$w$.}
We have~\eqref{eq:r-background-ae} for~$a_j$ and~\eqref{eq:r-w-ae} for~$\nabla w_m$. With~$\mathcal E_j$ as in Lemma~\ref{lem:r-two-limits}, the stability estimates~\eqref{eq:r-common-singular} and~\eqref{eq:r-common-degenerate}, followed by integration in~$t$, give
\begin{align}
 \|w_{j,m}-w_m\|_2^2+\|\nabla w_{j,m}-\nabla w_m\|_p^2
 &\le C\|\mathcal E_j(\nabla w_m)\|_{p'}^2,
 &&1<p\le2,\label{eq:r-paired-singular}\\
 \|w_{j,m}-w_m\|_2^2+\|\nabla w_{j,m}-\nabla w_m\|_p^p
 &\le C\|\mathcal E_j(\nabla w_m)\|_{p'}^{p'},
 &&p>2.\notag
\end{align}
The constants are independent of~$j,m$; for~$p>2$,~$C=C(p,T)$.

Fix~$j$ and let~$m\to\infty$ along~\eqref{eq:r-w-ae}. The first limit in Lemma~\ref{lem:r-two-limits}, applied with~$\xi_m=\nabla w_m$, gives convergence of the right-hand sides. Since~$w_{j,m}\to\widetilde w_j$ and~$w_m\to w$ in~$L^2(Q)$, weak lower semicontinuity of the gradient norms gives
\begin{equation}\label{eq:r-limit-singular}
 \begin{aligned}
 \|\widetilde w_j-w\|_2^2+\|\nabla\widetilde w_j-\nabla w\|_p^2
 &\le C\|\mathcal E_j(\nabla w)\|_{p'}^2,
 &&1<p\le2,\\
 \|\widetilde w_j-w\|_2^2+\|\nabla\widetilde w_j-\nabla w\|_p^p
 &\le C\|\mathcal E_j(\nabla w)\|_{p'}^{p'},
 &&p>2.
 \end{aligned}
\end{equation}
Now let~$j\to\infty$. The second limit in Lemma~\ref{lem:r-two-limits} gives, in both ranges,
\begin{equation}\label{eq:r-paired-convergence}
 \widetilde w_j\to w\text{ in }L^2(Q),\qquad \nabla \widetilde w_j\to \nabla w\text{ in }L^p(Q).
\end{equation}

\emph{Step 4: Perron inequalities on~$\partial\Theta$.}
Define~$U_m,V_j$ by~\eqref{eq:r-competitors} in~$\Theta$. Both are bounded, with~$\nabla U_m,\nabla V_j\in L^p_{\rm loc}(\Theta)$. Combining~\eqref{eq:r-sources},~\eqref{eq:r-H-local} and~\eqref{eq:r-obstacle-super} in distributions cancels the flux terms~$A(\nabla F)$ and~$A(\nabla F+\nabla H)$. Hence
\begin{equation}\label{eq:r-U-super}
 (U_m)_t-\Dp U_m\ge F_t-\Dp F+h+q=0.
\end{equation}
Likewise \eqref{eq:r-H-equation} and \eqref{eq:r-w-equations} give
\begin{equation}\label{eq:r-V-sub}
 (V_j)_t-\Dp V_j=F_t-\Dp F+h-\nu_j+q=-\nu_j\le0.
\end{equation}
For fixed $m$, $w_m=0$ almost everywhere in an open neighborhood of $\partial\Theta$, because $\operatorname{supp}\rho_m\Subset\Theta$. Since $H\ge0$, $U_m\ge F$ almost everywhere in that neighborhood. By Section~\ref{sec:preliminaries},~$(U_m)_*$ is~$p$-superparabolic and~$(U_m)_*\ge F$ throughout that neighborhood. Hence
\begin{equation*}
 \liminf_{\Theta\ni z\to\zeta}(U_m)_*(z)\ge F(\zeta)
 \quad(\zeta\in\partial\Theta).
\end{equation*}
Since~$(U_m)_*$ is bounded,~$(U_m)_*\in\mathcal U_f$ by Definition~\ref{def:perron}.

For fixed $j$, $H_j=0$ almost everywhere in an open neighborhood of $\partial\Theta$, while $\widetilde w_j\le0$. Thus~$(V_j)^*$ is bounded,~$p$-subparabolic, and satisfies
\begin{equation*}
 \limsup_{\Theta\ni z\to\zeta}(V_j)^*(z)\le F(\zeta)
 \quad(\zeta\in\partial\Theta).
\end{equation*}
Hence~$(V_j)^*\in\mathcal L_f$. The neighborhoods above contain~$\partial\Theta$, so both boundary inequalities hold at every~$\zeta\in\partial\Theta$.

We have~$U_m\to F+H+w$ almost everywhere. By~\eqref{eq:r-paired-convergence} and convergence of~$H_j$, a subsequence of~$V_j$ converges to~$F+H+w$ almost everywhere. Since only countably many functions are involved, replacing~$U_m,V_j$ by their canonical representatives preserves these almost-everywhere convergences. The definitions of~$\mathcal U_f,\mathcal L_f$ and comparison give
\begin{equation*}
 F+H+w\le\lowerH_\Theta f\le\upperH_\Theta f\le F+H+w
 \quad\text{a.e. in }\Theta.
\end{equation*}
Both~$\lowerH_\Theta f$ and~$\upperH_\Theta f$ are continuous, so equality holds throughout~$\Theta$. Lemma~\ref{lem:r-data} and~\eqref{eq:r-perron-stability} give the conclusion for every continuous boundary datum.
\end{proof}

\section{Geometric examples}\label{sec:g-examples}

We use the intrinsic formulation in Theorem~\ref{thm:main} for the geometric calculations below. The comparisons in Subsection~\ref{subsec:potential-proof} and Remark~\ref{cor:capacity-wiener} transfer the resulting all-height divergence conditions to the potential and capacity-difference formulations. For~$p>2$, the change of height from~$a$ to~$Aa$ in~\eqref{eq:potential-degenerate-comparison} does not alter a condition that tests every positive height.

We first prove Petrovski\u{\i}-type criteria, including the singular endpoint at the critical power~$1/p$. Critical domains and time-dependent graphs show why divergence is required at every obstacle height. Stationary cylinders connect the intrinsic series with the elliptic Wiener criterion without a thickness assumption on the complement. We finish with soda-can domains, where an upper capacity bound gives convergence.

Throughout this section,~$n\ge2$. The degenerate calculations assume~$2<p<n$, and the singular calculations assume~$\pc<p<2$. We use~$K_{j,q}$,~$Q_{j,q}$ and~$\omega_a(j)$ from Definition~\ref{def:wiener}, centered at~$(0,0)$. The degenerate estimates hold for every fixed~$L>4$ and~$C_*>0$; applications of Theorem~\ref{thm:main} use its prescribed constants. In the singular range, we use~$L=8$ and the threshold fixed before Definition~\ref{def:wiener}. We verify the capacity test for each chosen~$q$, since both~$K_{j,q}$ and~$Q_{j,q}$ depend on it.

\subsection{Capacity of a ball over a time interval}\label{sec:g-common}

A bounded source supported on a ball and a time interval gives the following lower bound. For~$p<2$, we use the capacity and energy conventions of Subsection~\ref{subsec:penalization}.

\begin{lemma}\label{lem:g-ball-band}
Let~$n\ge2$,~$\pc<p<n$,~$p\ne2$,~$L>4$,~$r,s>0$,~$D=B_{8Lr}$, and~$0<\sigma<1$. Suppose~$\overline B_{\sigma r}(y)\subset B_r$ and~$J\subset(-5s,s)$ is a closed interval of positive length. For~$a>0$,
\begin{equation}\label{eq:g-ball-band}
 \capc_a\bigl(\overline B_{\sigma r}(y)\times J,
                    D\times(-5s,s)\bigr)
 \ge c_\sigma a^{p-1}|J|r^{n-p}\,.
\end{equation}
Here~$c_\sigma>0$ depends only on~$n,p,\sigma$.
\end{lemma}
\begin{proof}
Put~$\rho=\sigma r$. Choose~$g>0$ and define the radial function
\begin{equation*}
 V(x)=\left(\frac{g r^{-p}}n\right)^{\frac{1}{p-1}}
 \int_{|x-y|}^{\infty}
 \left(\frac{\min\{\xi,\rho\}^n}{\xi^{n-1}}\right)^{\frac{1}{p-1}}\dd\xi\,.
\end{equation*}
Since~$p<n$, the integral is finite. Differentiating its radial flux gives
\begin{equation*}
 -\Dp V=g r^{-p}\mathbf1_{B_\rho(y)}\quad\text{in }\R^n\,.
\end{equation*}
The flux is continuous across~$|x-y|=\rho$, so there is no surface measure there. The maximum occurs at~$y$, where
\begin{equation*}
 V(y)=\left(\frac gn\right)^{\frac{1}{p-1}}
 \sigma^{\frac{p}{p-1}}(p-1)
 \left(\frac1p+\frac1{n-p}\right)\,.
\end{equation*}
Fix~$g=g(n,p,\sigma)>0$ small enough that~$V\le1$.

Let~$u$ solve
\begin{equation*}
 u_t-\Dp u=a^{p-1}g r^{-p}
       \mathbf1_{B_\rho(y)}\mathbf1_J(t)
\end{equation*}
with zero initial data at~$-5s$ and zero lateral values on~$\partial D$, continued for all later times. The source is bounded and compactly supported, so the evolution result recalled in Section~\ref{sec:preliminaries} gives a unique energy solution on each finite time interval; see~\cite[Theorem~1.1 and Remark~1.2(2)]{Liu}. Uniqueness gives consistent continuation to all later times.

The stationary function~$aV$ has a source at least as large, nonnegative lateral values, and nonnegative initial values. The positive-part energy test for~$u-aV$, and the corresponding test against~$0$, give~$0\le u\le aV\le a$ for every later time.

Testing the equation with~$u$ and discarding the nonnegative final~$L^2$ term gives
\begin{equation*}
 \int_{-5s}^{\infty}\!\int_D|\nabla u|^p\,dx\,dt
 \le a^{p}g r^{-p}|B_\rho||J|\,.
\end{equation*}
Poincar\'e's inequality puts~$u$ in~$L^p(-5s,\infty;W^{1,p}_0(D))$. The initial trace is zero, and~$u$ vanishes before the first source time. Its source is therefore an admissible measure in the definition of~$\capc_a$ (Definition~\ref{def:parabolic-capacity}). Its total mass is
\begin{equation*}
 a^{p-1}g r^{-p}|B_\rho||J|
 =g|B_1|\sigma^n a^{p-1}|J|r^{n-p}\,.
\end{equation*}
This proves~\eqref{eq:g-ball-band} with~$c_\sigma=g|B_1|\sigma^n$.
\end{proof}

For~$J=[-4s_j,-2s_j]$ and~$r=r_j(q)$, the intrinsic relation gives
\begin{equation}\label{eq:g-normalized-ball-band}
 \frac{\capc_a(\overline B_{\sigma r}(y)\times J,Q_{j,q})}
      {q r^n}
 \ge 2c_\sigma\left(\frac aq\right)^{p-1}\,.
\end{equation}
Set
\begin{equation}\label{eq:g-singular-ratio}
 \eta_\sigma=\min\left\{\frac12,\left(\frac{c_\sigma}{C_*}\right)^{\frac{1}{p-1}}\right\}\,.
\end{equation}
If~$\overline B_{\sigma r_j(q)}(y)\times[-4s_j,-2s_j]\subset K_{j,q}$ and~$0<q\le\eta_\sigma a$, then~\eqref{eq:g-normalized-ball-band} and capacity monotonicity make~$q$ admissible in~\eqref{eq:omega}. In particular, such an inclusion for one fixed~$q>0$ at every sufficiently large~$j$ gives divergence of the series. This holds in both ranges of~$p$. For fixed~$q$, we have~$r_j(q)\to0$; for fixed~$j$, the radius increases with~$q$ when~$p>2$ and decreases with~$q$ when~$p<2$.

\begin{remark}[Flat boundaries and exterior cones]\label{subsec:g-flat}
For~$\Theta=B_1\times(0,1)$, every past test cylinder lies in the complement, so~$\omega_a(j)\ge\eta_{1/2}a$ for all~$j$. For~$\Theta=(B_1\cap\{x_n>0\})\times(-1,1)$, the ball~$\overline B_{r/4}(-re_n/2)$ lies in the complement throughout each sufficiently late time interval, giving~$\omega_a(j)\ge\eta_{1/4}a$. More generally, let~$D\subset\R^n$ be bounded and open, with~$0\in\partial D$. If~$\R^n\setminus D$ contains a circular cone near zero, with half-angle~$\alpha\in(0,\pi/2]$ and axis~$e$, then for every sufficiently small~$r$,
\begin{equation*}
 \overline B_{r\sin\alpha/4}((r/2)e)\subset B_r\cap(\R^n\setminus D)\,.
\end{equation*}
For~$\Theta=D\times(-1,1)$, this gives~$\omega_a(j)\ge\eta_{\sin\alpha/4}a$ on sufficiently late time intervals. Theorem~\ref{thm:main} proves regularity in each case, in agreement with~\cite[Lemma~3.10, Theorem~3.9 and Proposition~5.3]{BBGP}. A plane moving at constant speed is included in the graph condition of Subsection~\ref{sec:g-parabolic-lipschitz}.
\end{remark}

\begin{remark}[Initial and terminal time slices]\label{sec:g-terminal-section}
The preceding initial-boundary example contrasts with~$\Theta=B_1\times(-1,0)$ at~$(0,0)$. For~$p>2$ and fixed~$a>0$, all~$q\le a$ have~$r_j(q)\le r_j(a)<1$ on sufficiently late time intervals, so~$K_{j,q}=\varnothing$ and~$\omega_a(j)=0$. For~$\pc<p<2$ and~$j\ge3$, every~$q>s_j^{1/(2-p)}$ has~$r_j(q)<1$ and is inadmissible. Hence
\begin{equation}\label{eq:g-singular-top}
 0\le\omega_a(j)\le\min\{a,2^{-\frac{j}{2-p}}\}\qquad(a>0,\ j\ge3)\,.
\end{equation}
Thus the series converges at every height in both ranges; in the singular range the test sets need not have a zero tail. This is the terminal-slice example discussed in~\cite[Section~4]{BBGP}.

A direct Perron argument applies for every~$p>1$. Prescribe the continuous boundary datum~$g(x,0)=(1-4|x|^2)_+$ and~$g=0$ on the boundary with~$t<0$. For~$0<\varepsilon<1$, the function~$u_\varepsilon(x,t)=(t+\varepsilon)_+/\varepsilon$ is a continuous weak supersolution with terminal limit~$1$ and nonnegative limits elsewhere, so it belongs to the upper class by~\cite[Lemma~2.7]{BBGP}. Zero belongs to the lower class. At each fixed~$(x,t)\in\Theta$, choosing~$\varepsilon<-t$ gives~$0\le\lowerH_\Theta g\le\upperH_\Theta g\le u_\varepsilon(x,t)=0$. Both Perron solutions are zero, although~$g(0,0)=1$, so the origin is irregular.
\end{remark}

\subsection{Petrovski\u{\i}-type criteria}
\label{sec:g-power-section}

Fix~$C,\beta>0$ and consider the domain
\begin{equation}\label{eq:g-power-domain}
 \Theta=\{(x,t)\in\R^{n+1}:|x|<C(-t)^\beta,\ -1<t<0\}\,.
\end{equation}
% Retain the labels of the two former domain definitions.
\begingroup\makeatletter
\protected@edef\@currentlabel{\theequation}
\label{eq:g-tube-irregular-domain}\label{eq:g-singular-power-domain}
\makeatother\endgroup
This is a bounded open set, and~$(0,0)\in\partial\Theta$. The relation between the spatial radius and the intrinsic spatial scale determines regularity as summarized in Table~\ref{tab:g-power-classification}. At the critical power~$\beta=1/p$, the origin is irregular for~$2<p<n$ and regular for~$\pc<p<2$. The degenerate calculation also distinguishes small and large obstacle heights.

\begin{table}[htbp]
\centering
\caption{Regularity of the origin for the domain~\eqref{eq:g-power-domain}, for every~$C>0$.}
\label{tab:g-power-classification}
\begin{tabular}{c|ccc}
 & $0<\beta<1/p$ & $\beta=1/p$ & $\beta>1/p$ \\
\hline
$2<p<n$ & Irregular & Irregular & Regular \\
$\pc<p<2$, $n\ge2$ & Irregular & Regular & Regular
\end{tabular}
\end{table}

The degenerate classification is~\cite[Theorem~1.1(a)]{BBG}. In the singular range, Theorem~1.1(c) of that paper gives the cases~$\beta\ne1/p$ and leaves~$\beta=1/p$ undecided. Proposition~\ref{prop:g-singular-power} below includes this endpoint by applying the singular Wiener characterization.

\paragraph{The singular range~$\pc<p<2$.}
The comparison of the spatial radius with~$r_j(q)$ gives upper and lower bounds for the singular Wiener terms. At the critical power~$1/p$, a sufficiently small fixed increment detects a ball outside the domain at every late scale.

\begin{proposition}[Petrovski\u{\i}-type criterion for~$\pc<p<2$]\label{prop:g-singular-power}
Let~$n\ge2$,~$\pc<p<2$, and let~$\Theta$ be given by~\eqref{eq:g-power-domain}, with~$C,\beta>0$.
There are~$c_1,c_2>0$, depending only on~$p,C,\beta$, such that, for every~$a>0$ and~$j\ge3$,
\begin{equation}\label{eq:g-singular-power-terms}
 \min\left\{\eta_{1/8}a,c_1s_j^{\frac{1-p\beta}{2-p}}\right\}
 \le\omega_a(j)
 \le\min\left\{a,c_2s_j^{\frac{1-p\beta}{2-p}}\right\}\,.
\end{equation}
The origin is regular if and only if~$\beta\ge1/p$.
\end{proposition}
\begin{proof}
Choose
\begin{equation*}
 c_1=(4^pC^p4^{p\beta})^{-\frac{1}{2-p}},\qquad
 c_2=(C^p2^{p\beta})^{-\frac{1}{2-p}}\,.
\end{equation*}
If~$q>c_2s_j^{\frac{1-p\beta}{2-p}}$, then~$r_j(q)<C(2s_j)^\beta$. The corresponding closed cylinder lies in~$\Theta$, so~$K_{j,q}=\varnothing$. Every admissible increment is therefore at most this bound, proving the upper inequality in~\eqref{eq:g-singular-power-terms}.

For the lower bound, take
\begin{equation*}
 q=\min\left\{\eta_{1/8}a,c_1s_j^{\frac{1-p\beta}{2-p}}\right\},
 \qquad r=r_j(q)\,.
\end{equation*}
The choice of~$c_1$ gives~$C(4s_j)^\beta\le r/4$. Every point of~$\overline B_{r/8}((r/2)e_1)$ has distance between~$3r/8$ and~$5r/8$ from zero, so this ball lies strictly inside~$B_r$ and strictly outside the domain throughout the time interval. Since~$q\le\eta_{1/8}a$,~\eqref{eq:g-normalized-ball-band} makes~$q$ admissible. Taking the supremum proves the lower inequality.

If~$\beta<1/p$, the upper bound is a summable geometric sequence, so the origin is irregular by Theorem~\ref{thm:main}. If~$\beta\ge1/p$, the lower bound stays positive for each fixed~$a>0$, and the series diverges at every height. The same theorem proves regularity.
\end{proof}

The same lower-bound argument proves regularity whenever the spatial radius is at most~$C(-t)^{\frac{1}{p}}$ near zero.

\paragraph{The degenerate range~$2<p<n$.}\label{sec:g-tube-irregular-sec}
For~$\beta>1/p$, a ball of fixed relative radius remains outside the domain throughout every sufficiently late test interval. For~$\beta<1/p$, all tested balls eventually lie inside the domain. At~$\beta=1/p$, their position relative to the domain depends on the obstacle height.

\begin{proposition}[Wiener terms for domains with power-law spatial radius]\label{prop:g-power-divergence}\label{prop:g-tube-irregular-prop}
Let~$2<p<n$ and let~$\Theta$ be given by~\eqref{eq:g-power-domain}, with~$C,\beta>0$.
\begin{enumerate}[(i)]
\item If~$\beta>1/p$, there is~$\eta\in(0,1)$, depending only on~$n,p,C_*$, such that for every~$a>0$ there is~$N=N(n,p,C_*,C,\beta,a)$ with
\begin{equation*}
 \omega_a(j)\ge\eta a\qquad\text{for all }j\ge N\,.
\end{equation*}
Consequently,~$\sum_{j=1}^\infty\omega_a(j)=\infty$ for every~$a>0$.
\item If~$\beta<1/p$, then for every~$a>0$ there is~$j_0$ such that
\begin{equation}\label{eq:g-tube-irregular-empty}
 K_{j,q}=\varnothing
 \quad\text{for every }j\ge j_0\text{ and every }0<q\le a\,.
\end{equation}
Thus~$\omega_a(j)=0$ for all~$j\ge j_0$ and~$\sum_{j=1}^\infty\omega_a(j)<\infty$.
\item If~$\beta=1/p$, the same conclusions hold for every height satisfying
\begin{equation}\label{eq:g-tube-irregular-height}
 0<a<(2C^p)^{\frac{1}{p-2}}\,.
\end{equation}
For every~$q_*>0$ with~$q_*^{(p-2)/p}\ge4C4^{1/p}$ and every~$a\ge q_*/\eta_{1/8}$, one has~$\omega_a(j)\ge q_*$ for all~$j\ge3$.
\end{enumerate}
\end{proposition}
\begin{proof}
\emph{Case \textup{(i)}.}
Fix~$a>0$ and put~$q=\eta_{1/8}a$. On~$[-4s_j,-2s_j]$, the spatial radius is at most~$C(4s_j)^\beta$, and
\begin{equation}\label{eq:g-power-radius-ratio}
 \frac{C(4s_j)^\beta}{r_j(q)}
 =C4^\beta q^{-\frac{p-2}{p}}s_j^{\beta-\frac1p}\longrightarrow0\,.
\end{equation}
For every sufficiently large~$j$, we have~$j\ge3$ and~$C(4s_j)^\beta\le r_j(q)/4$. Writing~$r=r_j(q)$, the ball~$\overline B_{r/8}((r/2)e_1)$ has~$3r/8\le|x|\le5r/8$ and therefore lies inside~$B_r$ and outside the domain throughout the time interval. The fixed increment~$q$ is admissible by~\eqref{eq:g-normalized-ball-band}.

\emph{Cases \textup{(ii)--(iii)}: vanishing of the test sets.}
For~$j\ge3$, the spatial radius on the tested time interval is at least~$C(2s_j)^\beta$. Since~$p>2$, every~$q\le a$ satisfies
\begin{equation}\label{eq:g-tube-irregular-ratio}
 \frac{r_j(q)}{C(2s_j)^\beta}
 \le\frac{a^{\frac{p-2}{p}}}{C2^\beta}s_j^{\frac1p-\beta}\,.
\end{equation}
If~$\beta<1/p$, this tends to zero uniformly in~$q\le a$. If~$\beta=1/p$, it is strictly less than one under~\eqref{eq:g-tube-irregular-height}. Thus every tested closed ball lies strictly inside the domain throughout the time interval, giving~\eqref{eq:g-tube-irregular-empty} and~$\omega_a(j)=0$ for all~$j\ge j_0$. The remaining terms are bounded by~$a$.

\emph{Case \textup{(iii)}, large heights.}\label{par:g-tube-irregular-large-height}
At~$\beta=1/p$, choose~$q_*>0$ with~$q_*^{(p-2)/p}\ge4C4^{1/p}$. For~$j\ge3$, the spatial radius throughout the time interval is at most~$C(4s_j)^{1/p}\le r_j(q_*)/4$. With~$r=r_j(q_*)$, we have~$\overline B_{r/8}((r/2)e_1)\times[-4s_j,-2s_j]\subset K_{j,q_*}$. For every~$a\ge q_*/\eta_{1/8}$, the increment~$q_*$ is admissible by~\eqref{eq:g-normalized-ball-band}, giving~$\omega_a(j)\ge q_*$ for all~$j\ge3$.
\end{proof}

Proposition~\ref{prop:g-tube-irregular-prop}\textup{(ii)--(iii)} and Theorem~\ref{thm:main} show that the origin is irregular for~\eqref{eq:g-power-domain} when~$0<\beta\le1/p$, in agreement with~\cite[Theorem~1.1(a)]{BBG}. The empty sets in these cases are the intrinsic sets~$K_{j,q}$. For the critical domain, the prescribed sets~$K_j$ in~\eqref{eq:potential-test-sets} are nonempty for all sufficiently large~$j$, since~$R_j/s_j^{1/p}\to\infty$.

\begin{remark}[Failure of a non-intrinsic capacity series]\label{rem:g-nonintrinsic-capacity}
Let~$2<p<n$, and take~$\beta=1/p$ and~$C>0$ with~$C4^{1/p}\le1/4$ in~\eqref{eq:g-power-domain}. The origin is irregular by Proposition~\ref{prop:g-power-divergence}\textup{(iii)} and Theorem~\ref{thm:main}. Take~$K_{j,1}$ and~$Q_{j,1}$ from the height-one tests and keep them fixed as the obstacle height~$a>0$ varies. For~$j\ge3$, put~$r=r_j(1)=s_j^{1/p}$. The domain has spatial radius at most~$r/4$ throughout~$[-4s_j,-2s_j]$, so
\begin{equation*}
 \overline B_{r/8}((r/2)e_1)\times[-4s_j,-2s_j]\subset K_{j,1}\,.
\end{equation*}
Lemma~\ref{lem:g-ball-band} and capacity monotonicity give, for every~$a>0$,
\begin{equation*}
 \capc_a(K_{j,1},Q_{j,1})\ge c a^{p-1}s_j^{n/p},
 \qquad c=c(n,p)>0\,.
\end{equation*}
Consequently, for every~$a>0$ and~$\theta>0$,
\begin{equation*}
 \sum_{j=3}^{\infty}
 \left(\frac{\capc_a(K_{j,1},Q_{j,1})}{a s_j^{n/p}}\right)^\theta
 =\infty\,.
\end{equation*}
\end{remark}

\paragraph{Comparison with Petrovski\u{\i}'s criterion.}
For the heat equation~$u_t=\Delta u$, the classical criterion of Petrovski\u{\i}~\cite{Petrovskii35}, recalled in~\cite[Section~1]{BBG}, concerns the domains
\begin{equation*}
 \left\{(x,t):|x|<C\sqrt{(-t)\log\log(1/(-t))},
 \ -e^{-e}<t<0\right\}\,.
\end{equation*}
The origin is regular precisely when~$C\le2$. The initial time is chosen so that both logarithms are positive. 
For the domains~\eqref{eq:g-power-domain}, the corresponding linear threshold is~$\beta\ge1/2$, for every~$C>0$; see~\cite[Theorem~1.1(b)]{BBG}. For~$p\ne2$, the classifications in Table~\ref{tab:g-power-classification} are also independent of~$C$, but the critical power is~$1/p$. It is excluded from the regular range when~$p>2$ and included when~$\pc<p<2$. Proposition~\ref{prop:g-singular-power} supplies the singular endpoint. The logarithmic example below illustrates the sufficient condition~$\zeta(t)=o((-t)^{\frac{1}{p}})$ of~\cite[Theorem~6.1]{BBG}.

\begin{remark}[Logarithmically shrinking tubes]\label{sec:g-log-section}
For~$p>2$, the proof of Proposition~\ref{prop:g-power-divergence}\textup{(i)} applies to any bounded tube with positive continuous radius~$\zeta$ satisfying~$\zeta(t)=o((-t)^{1/p})$ as~$t\uparrow0$. Indeed, for fixed~$q=\eta_{1/8}a$, this gives~$\sup_{[-4s_j,-2s_j]}\zeta=o(s_j^{1/p})$, so the same ball lies in the complement on every sufficiently late time interval. The series diverges at every height, recovering~\cite[Theorem~6.1]{BBG} in the present range. In particular, for every~$\gamma>0$, the origin is regular for
\begin{equation}\label{eq:g-log-domain}
 \Theta=\{(x,t):|x|<(-t)^{\frac1p}[\log(e/(-t))]^{-\gamma},\ -1<t<0\}\,.
\end{equation}
For~$\pc<p<2$, the radius is at most~$(-t)^{1/p}$, so the lower-bound argument of Proposition~\ref{prop:g-singular-power} also proves regularity.
\end{remark}

\subsection{Parabolic Lipschitz domains}
\label{sec:g-parabolic-lipschitz}

The time exponent~$1/p$ matches the intrinsic scaling. A graph with a larger time exponent moves by less than the test radius on sufficiently late time intervals. At the exponent~$1/p$, the graph can move by a fixed multiple of that radius. For~$2<p<n$, the series can then converge at small heights, whereas for~$\pc<p<2$ it diverges at every height.

Write~$x=(x',x_n)\in\R^{n-1}\times\R$. Let~$\Theta$ be bounded and open, with~$(0,0)\in\partial\Theta$. Suppose that, after a spatial rotation, there are~$\rho,T>0$ and a function~$f:B'_\rho\times(-T,T)\to\R$ such that
\begin{equation}\label{eq:g-lipschitz-graph}
 \Theta\cap\bigl(B_\rho\times(-T,T)\bigr)
 =\{(x,t)\in B_\rho\times(-T,T):x_n>f(x',t)\}\,.
\end{equation}
Here~$B'_\rho$ is the ball in~$\R^{n-1}$. We assume~$f(0,0)=0$ and
\begin{equation}\label{eq:g-lipschitz-modulus}
 |f(x',t)-f(y',\tau)|
 \le M\bigl(|x'-y'|+|t-\tau|^\gamma\bigr)\,,
 \qquad M\ge0,\quad 1/p\le\gamma\le1\,.
\end{equation}
This is the local~$\operatorname{Lip}(1,\gamma)$ graph condition. The usual parabolic Lipschitz condition has~$\gamma=1/2$; the intrinsic exponent is~$\gamma=1/p$.

For comparison, Avelin~\cite{A16} obtains boundary H\"older estimates for nonnegative solutions in time-dependent domains under intrinsic parabolic NTA and exterior conditions. Here we prove Perron regularity from a lower bound for the capacity of a ball below the graph.

\begin{proposition}\label{prop:g-lipschitz-series}
Assume~\eqref{eq:g-lipschitz-graph}--\eqref{eq:g-lipschitz-modulus}, with~$2<p<n$.
\begin{enumerate}[(i)]
\item If~$\gamma>1/p$, there is~$\eta=\eta(n,p,M,C_*)\in(0,1]$ such that, for every~$a>0$, we have~$\omega_a(j)\ge\eta a$ for every sufficiently large~$j$.
\item If~$\gamma=1/p$, there are~$q_*>0$ and~$a_0>0$, depending only on~$n,p,M,C_*$, such that, for every~$a\ge a_0$, we have~$\omega_a(j)\ge q_*$ for every sufficiently large~$j$.
\end{enumerate}
For the heights specified in each item,~$\sum_j\omega_a(j)=\infty$.
\end{proposition}
\begin{proof}
Set~$\sigma=1/(8(1+M))$ and~$\eta=\eta_\sigma$ from~\eqref{eq:g-singular-ratio}. For fixed~$q>0$, put~$r=r_j(q)$ and~$s=s_j$. Take~$j$ so large that~$r<\rho$ and~$4s<T$. On the closed ball~$\overline B_{\sigma r}(-re_n/2)$, we have
\begin{equation*}
 |x|\le(1/2+\sigma)r<r\,,\qquad
 x_n+M|x'|\le-r/2+(1+M)\sigma r=-3r/8\,.
\end{equation*}
By~\eqref{eq:g-lipschitz-modulus}, throughout~$[-4s,-2s]$,
\begin{equation*}
 f(x',t)\ge-M|x'|-M(4s)^\gamma\,.
\end{equation*}
Thus, whenever~$M(4s)^\gamma\le r/4$, the ball lies strictly below the graph throughout the time interval, and
\begin{equation}\label{eq:g-lipschitz-conductor}
 \overline B_{\sigma r}(-re_n/2)\times[-4s,-2s]
 \subset K_{j,q}\,.
\end{equation}
The bound for the graph displacement relative to the test radius satisfies
\begin{equation}\label{eq:g-lipschitz-displacement}
 \frac{M(4s_j)^\gamma}{r_j(q)}
 =4^\gamma M q^{-\frac{p-2}{p}}s_j^{\gamma-\frac{1}{p}}\,.
\end{equation}

If~$\gamma>1/p$, choose~$q=\eta a$. The right-hand side of~\eqref{eq:g-lipschitz-displacement} tends to zero, so~\eqref{eq:g-lipschitz-conductor} holds for all sufficiently large~$j$. Capacity monotonicity and~\eqref{eq:g-normalized-ball-band} give
\begin{equation}\label{eq:g-lipschitz-capacity}
 \frac{\capc_a(K_{j,q},Q_{j,q})}{q r_j(q)^n}
 \ge2c_\sigma(a/q)^{p-1}\ge2C_*>C_*\,.
\end{equation}
Hence~$\omega_a(j)\ge q=\eta a$ there.

If~$\gamma=1/p$, fix~$q_*>0$ satisfying
\begin{equation*}
 q_*^{\frac{p-2}{p}}\ge4^{1+\frac{1}{p}}M\,.
\end{equation*}
Then~\eqref{eq:g-lipschitz-conductor} holds with~$q=q_*$ whenever the corresponding cylinder lies in the graph neighborhood. Set~$a_0=q_*/\eta$. For every~$a\ge a_0$, we have~$q_*\le a$, and~\eqref{eq:g-lipschitz-capacity} again holds. This proves~$\omega_a(j)\ge q_*$ for all sufficiently large~$j$.
\end{proof}

Since~$p>2$, we have~$1/2>1/p$. Proposition~\ref{prop:g-lipschitz-series}\textup{(i)} and Theorem~\ref{thm:main} therefore prove regularity of the origin for every domain locally given by a~$\operatorname{Lip}(1,1/2)$ graph as in~\eqref{eq:g-lipschitz-graph}. At~$\gamma=1/p$, the displacement relative to the test radius is independent of~$j$ and depends on~$q$. For the next domain, the origin is irregular although the series diverges at every sufficiently large obstacle height.

\begin{proposition}\label{prop:g-lipschitz-critical-irregular}
Let~$M>0$ and~$2<p<n$, and set
\begin{equation}\label{eq:g-lipschitz-critical-domain}
 \Theta_M=(B_1\times(-1,1))\cap
 \{(x,t):x_n>-M|t|^{\frac{1}{p}}\}\,.
\end{equation}
The origin is irregular for~$\Theta_M$. For every height satisfying
\begin{equation}\label{eq:g-lipschitz-critical-small-height}
 0<a<(2M^p)^{\frac{1}{p-2}}\,,
\end{equation}
we have~$\omega_a(j)=0$ for every sufficiently large~$j$, and~$\sum_j\omega_a(j)<\infty$. For every sufficiently large height, the series diverges.
\end{proposition}
\begin{proof}
The inequality~$\bigl||t|^{\frac{1}{p}}-|\tau|^{\frac{1}{p}}\bigr|\le|t-\tau|^{\frac{1}{p}}$ shows that the defining graph belongs to~$\operatorname{Lip}(1,1/p)$. Fix~$a$ satisfying~\eqref{eq:g-lipschitz-critical-small-height}. For~$j\ge3$,~$0<q\le a$, and~$t\in[-4s_j,-2s_j]$, we have
\begin{equation*}
 r_j(q)\le a^{\frac{p-2}{p}}s_j^{\frac{1}{p}}
 <M(2s_j)^{\frac{1}{p}}\le M(-t)^{\frac{1}{p}}\,.
\end{equation*}
Once~$r_j(a)<1$, every point of~$\overline B_{r_j(q)}$ therefore satisfies~$x_n\ge-r_j(q)>-M(-t)^{\frac{1}{p}}$. The corresponding closed cylinder lies in~$\Theta_M$, so~$K_{j,q}=\varnothing$ for every~$q\le a$. Its capacity is zero, and hence~$\omega_a(j)=0$ for all sufficiently large~$j$. The remaining finitely many terms are at most~$a$. The large-height divergence follows from Proposition~\ref{prop:g-lipschitz-series}(ii).

Since the series converges at the positive height~$a$ fixed above, Theorem~\ref{thm:main} implies that the origin is irregular for~$\Theta_M$.
\end{proof}

\paragraph{Singular parabolic Lipschitz graphs.}
For the graph~\eqref{eq:g-lipschitz-graph} satisfying~\eqref{eq:g-lipschitz-modulus}, let~$\gamma\ge1/p$ and~$\sigma=1/[8(1+M)]$. The spatial comparison used in~\eqref{eq:g-lipschitz-conductor} applies to the same ball~$\overline B_{\sigma r}(-re_n/2)$. The bound for the graph displacement relative to the test radius is
\begin{equation*}
 \frac{M(4s_j)^\gamma}{r_j(q)}
 =4^\gamma Mq^{\frac{2-p}{p}}s_j^{\gamma-\frac{1}{p}}\,.
\end{equation*}
If~$\gamma>1/p$, fix~$q=\eta_\sigma a$ and then choose a tail on which this ratio is at most~$1/4$ and the ball lies in the graph neighborhood. If~$\gamma=1/p$, first choose~$0<q\le\eta_\sigma a$ so small that~$4^{\frac{1}{p}}Mq^{\frac{2-p}{p}}\le1/4$, and then choose the tail. In either case, the corresponding closed cylinder is contained in~$K_{j,q}$, so~$\omega_a(j)\ge q$ for all sufficiently large~$j$. Theorem~\ref{thm:main} therefore proves regularity of the origin for every domain satisfying these graph conditions. In particular, for the domain~\eqref{eq:g-lipschitz-critical-domain}, the origin is regular in the present range and irregular for~$p>2$.

\subsection{Stationary cylinders}
\label{sec:g-rough-sec}

For elliptic-capacity Wiener criteria in cylinders under structural and dimension restrictions, see~\cite{Sk00,Sk04s,Sk04d,Sk04n}. Here we compare the intrinsic series with the elliptic Wiener series. A stationary complement retains its elliptic capacity throughout each time interval. The intrinsic radius depends on the increment, so several spatial scales may correspond to the same time index. We bound the sum of their increments by one Wiener term. This transfers divergence of the elliptic Wiener series to~$\sum_j\omega_a(j)$.

Let~$2<p<n$, let~$D\subset\R^n$ be bounded and open, and let~$0\in\partial D$ be regular for the elliptic~$p$-Laplace equation. Put
\begin{equation*}
 \Theta=D\times(-1,1),\qquad F=\R^n\setminus D\,.
\end{equation*}
No thickness bound uniform in the radius is assumed. $(0,0)$ is parabolically regular by~\cite[Theorem~3.9]{BBGP}, after translating time by~$1$.

The following estimate applies even when the spatial compact set has empty interior.

\begin{lemma}[Capacity of a spatial set over a time interval]\label{lem:g-rough-product}
Let~$G$ be a bounded smooth spatial domain and let~$H\Subset G$ be compact. Let~$a>0$,~$b<c$, and let~$J\subset(b,c)$ be a closed interval. There is~$c_0=c_0(n,p)>0$ such that
\begin{equation}\label{eq:g-rough-product-bound}
 \capc_a(H\times J,G\times(b,c))
 \ge c_0a^{p-1}|J|\capce(H,G)\,.
\end{equation}
The finite reference cylinder is interpreted with the homogeneous continuation used in the test.
\end{lemma}
\begin{proof}
The infinite-time variational and measure capacities are comparable for every~$p\ge2$. At height~$1$, use the variational capacity on~$G\times(b,\infty)$ from~\cite[Definition~2.6 and Theorem~4.9]{AKP}. The variational cost contains the integral of~$|\nabla v|^p$. For each smooth competitor~$v\ge\mathbf1_{H\times J}$,
\begin{equation*}
 \int_J\!\int_G|\nabla v|^p\,dx\,dt
 \ge |J|\capce(H,G)\,.
\end{equation*}
Taking the infimum and applying the comparison between variational and measure capacity gives~\eqref{eq:g-rough-product-bound} with~$a=1$. Here the variational capacity is taken on the full continuation cylinder.

For height~$a$, set~$\tau=a^{p-2}(t-b)$ and~$u(x,t)=a v(x,\tau)$. The Riesz measure of~$u$ has total mass~$a$ times that of~$v$. The transformed interval has length~$a^{p-2}|J|$, whereas the spatial domain and elliptic capacity are unchanged. The measure definition of capacity therefore gives~\eqref{eq:g-rough-product-bound}. Homogeneous continuation adds no source mass, as in~\cite[Remark~5.2(2)]{KKKP}.
\end{proof}

Write~$\mathcal C(H)=\capce(H,\R^n)$ for the homogeneous elliptic capacity and set
\begin{equation}\label{eq:g-rough-dyadic-density}
 R_k=2^{-k},\qquad F_k=F\cap\overline B_{R_k},\qquad
 b_k=\left(\frac{\mathcal C(F_k)}{R_k^{n-p}}\right)^{\frac{1}{p-1}}\,.
\end{equation}
Here~$0\le b_k\le B$, where~$B=\mathcal C(\overline B_1)^{\frac{1}{p-1}}$. Elliptic regularity implies
\begin{equation}\label{eq:g-rough-elliptic-sum}
 \sum_{k=1}^{\infty}b_k=\infty\,.
\end{equation}
To obtain~\eqref{eq:g-rough-elliptic-sum} from the elliptic Wiener criterion~\cite[Theorem~1.1, p.~138]{KM}, first compare homogeneous and relative capacity. For compact~$H\subset\overline B_r$, multiplying a global competitor~$\phi$ by a cutoff equal to one on~$B_r$, supported in~$B_{2r}$, and with gradient bounded by~$C/r$, gives
\begin{equation*}
 \mathcal C(H)\le\capce(H,B_{2r})\le C\mathcal C(H)\,.
\end{equation*}
Indeed, the extra term is bounded by~$r^{-p}\int_{B_{2r}}|\phi|^p\le C\|\phi\|_{L^{np/(n-p)}(\R^n)}^p\le C\int_{\R^n}|\nabla\phi|^p$ by H\"older's and Sobolev's inequalities. The denominator in the relative Wiener ratio is a constant times~$r^{n-p}$. For~$R_{k+1}\le r\le R_k$, monotonicity gives
\begin{equation*}
 2^{-\frac{n-p}{p-1}}b_{k+1}
 \le\left(\frac{\mathcal C(F\cap\overline B_r)}{r^{n-p}}\right)^{\frac1{p-1}}
 \le2^{\frac{n-p}{p-1}}b_k\,.
\end{equation*}
Open and closed inner balls give the same divergence by comparison at~$r/2$ and~$r$. Integration with~$dr/r$ therefore gives~\eqref{eq:g-rough-elliptic-sum}.

The proof assigns an increment proportional to~$ab_k$ at radius~$R_k$ and groups equal time indices.

\begin{proposition}\label{prop:g-rough-divergence}
For every~$a>0$, the terms defined in~\eqref{eq:omega} satisfy
\begin{equation*}
 \sum_{j=1}^{\infty}\omega_a(j)=\infty\,.
\end{equation*}
\end{proposition}
\begin{proof}
Fix~$a>0$. Choose~$\varepsilon>0$, depending only on~$n,p,C_*$, so small that
\begin{equation*}
 \varepsilon B\le1,\qquad
 2c_0\varepsilon^{-(p-1)}2^{-\frac{n-p}{p}}>C_*\,.
\end{equation*}
For each~$k$ with~$b_k>0$, define
\begin{equation}\label{eq:g-rough-time-assignment}
 q_k=\varepsilon a b_k,\qquad
 \theta_k=pk+(p-2)\log_2q_k,\qquad j_k=\lfloor\theta_k\rfloor\,.
\end{equation}
Then~$0<q_k\le a$. For~$j_k\ge3$, the exact intrinsic radius obeys
\begin{equation*}
 r_{j_k}(q_k)=R_k2^{\frac{\theta_k-j_k}{p}},
 \qquad R_k\le r_{j_k}(q_k)<2^{\frac{1}{p}}R_k\,.
\end{equation*}
The tested time interval then lies in~$(-1,1)$, and~$F_k\times[-4s_{j_k},-2s_{j_k}]\subset K_{j_k,q_k}$. Set~$r=r_{j_k}(q_k)$. Capacity monotonicity,~\eqref{eq:g-rough-product-bound}, and~$\capce(F_k,B_{8Lr})\ge\mathcal C(F_k)$ give
\begin{align*}
 \frac{\capc_a(K_{j_k,q_k},Q_{j_k,q_k})}{q_kr^n}
 &\ge 2c_0\left(\frac a{q_k}\right)^{p-1}
              \frac{\mathcal C(F_k)}{r^{n-p}}\\
 &\ge 2c_0\varepsilon^{-(p-1)}2^{-\frac{n-p}{p}}>C_*\,.
\end{align*}
Here we used~$r^p=q_k^{p-2}s_{j_k}$ exactly. Consequently,~$\omega_a(j_k)\ge q_k$.

We now control the repeated indices. If~$j_k<3$, then~$\theta_k<3$, so
\begin{equation*}
 q_k<2^{\frac{3-pk}{p-2}}\,.
\end{equation*}
The sum of these discarded terms is finite. On the other hand,~\eqref{eq:g-rough-elliptic-sum} implies~$\sum_kq_k=\infty$. For each nonempty group with~$j_k=j\ge3$, let~$k_j$ be its least spatial index. From~\eqref{eq:g-rough-time-assignment},
\begin{equation*}
 2^{\frac{j-pk}{p-2}}\le q_k<2^{\frac{j+1-pk}{p-2}}\,.
\end{equation*}
It follows that
\begin{equation*}
 \sum_{k:j_k=j}q_k
 \le\frac{2^{\frac{1}{p-2}}}{1-2^{-\frac{p}{p-2}}}q_{k_j}
 \le\frac{2^{\frac{1}{p-2}}}{1-2^{-\frac{p}{p-2}}}\omega_a(j)\,.
\end{equation*}
Summing over~$j\ge3$ proves divergence. 
\end{proof}

\paragraph{Singular stationary cylinders.}
Let~$D\subset\R^n$ be bounded and open, let~$0\in\partial D$, and put~$\Theta=D\times(-1,1)$. By~\cite[Theorem~3.9]{BBGP}, after translating time,~$(0,0)$ is parabolically regular if and only if~$0$ is elliptically~$p$-regular for~$D$. Theorem~\ref{thm:main} therefore gives
\begin{equation*}
 0\text{ is elliptically regular for }D
 \quad\Longleftrightarrow\quad
 \sum_j\omega_a(j)=\infty\text{ for every }a>0\,.
\end{equation*}
If~$0$ is elliptically irregular, the series converges for at least one positive height. No thickness or smoothness assumption on~$\partial D$ is required.

\subsection{Soda-can domains}\label{sec:g-soda-regular-section}\label{sec:g-soda-irregular-section}

Let~$2<p<n$,~$l,\theta>0$, and put
\begin{equation}\label{eq:g-soda-regular-domain}
 \Theta_{l,\theta}=\{(x,t):0<-t<\theta|x|^l<\theta\}\,.
\end{equation}
At time~$t<0$, the complement contains~$B_{(-t/\theta)^{1/l}}$. The case~$l<p$ illustrates the upper capacity bound, which gives convergence at every height. For~$l\ge p$, a ball in the complement gives divergence. These calculations recover the classification in~\cite[Theorem~1.1]{AP26}; that paper treats all~$p>1$ and constructs traditional barriers at irregular points.

\begin{proposition}\label{prop:g-soda-regular-proposition}\label{prop:g-soda-irregular-series}
For every~$a>0$, the following hold.
\begin{enumerate}[(i)]
\item If~$l\ge p$, there are~$q_0\in(0,a]$ and~$j_0$ such that~$\omega_a(j)\ge q_0$ for all~$j\ge j_0$.
\item If~$l<p$, there is~$j_0=j_0(a,p,l,\theta)$ such that, for~$j\ge j_0$,
\begin{equation}\label{eq:g-soda-irregular-decay}
 \omega_a(j)\le C a^{\frac{p(p-1)}{\lambda}}
 \left(\frac4\theta\right)^{\frac{p(n-p)}{l\lambda}}
 2^{-\frac{j(n-p)(p-l)}{l\lambda}}\,.
\end{equation}
Here~$\lambda$ is defined in~\eqref{eq:lambda}, and~$C$ depends only on~$n,p,C_*$, for every fixed~$L>4$.
\end{enumerate}
The series diverges at every height if~$l\ge p$ and converges at every height if~$l<p$. The origin is regular if and only if~$l\ge p$.
\end{proposition}
\begin{proof}
For~$l=p$, choose~$q_0=\min\{\eta_{1/2}a,\theta^{-1/(p-2)}\}$. Then~$\theta r_j(q_0)^p\le s_j$. For~$l>p$, choose~$q_0=\eta_{1/2}a$; now
\begin{equation}\label{eq:g-soda-regular-supercritical-geometry}
 \frac{\theta r_j(q_0)^l}{s_j}
 =\theta q_0^{\frac{l(p-2)}p}s_j^{\frac{l-p}p}\longrightarrow0\,.
\end{equation}
Thus in both cases~$\theta r_j(q_0)^l\le s_j$ for all sufficiently large~$j$. For~$|x|\le r_j(q_0)$ and~$t\in[-4s_j,-2s_j]$, this gives~$\theta|x|^l\le s_j<-t$, so the tested ball lies in the complement. Its concentric ball of half the radius makes~$q_0$ admissible by~\eqref{eq:g-normalized-ball-band}, proving~(i).

For~$l<p$, write
\begin{equation*}
 b_j=(4s_j/\theta)^{\frac1l},\qquad
 d_j=b_j/s_j^{\frac1p}=(4/\theta)^{\frac1l}s_j^{\frac1l-\frac1p}\,.
\end{equation*}
Choose~$j_0$ so that~$r_j(a)<1$ and~$d_j^p\le a^{p-2}$ for every~$j\ge j_0$. For~$q\le a$, the tested ball does not meet~$\{|x|=1\}$, and the remaining complement on the tested time interval is contained in~$\overline B_{b_j}$. Hence~$K_{j,q}\subset\overline B_{b_j}\times[-4s_j,-2s_j]$.

If~$r_j(q)\le2b_j$, the intrinsic relation gives
\begin{equation}\label{eq:g-soda-irregular-small-radius}
 q\le2^{\frac{p}{p-2}}d_j^{\frac{p}{p-2}}\,.
\end{equation}
If~$r_j(q)>2b_j$, then~$B_{2b_j}\subset B_{8Lr_j(q)}$. Monotonicity in the set and spatial domain, with initial time~$-5s_j$, and Lemma~\ref{lem:cylinder-capacity} give
\begin{equation}\label{eq:g-soda-irregular-cylinder-upper}
 \begin{aligned}
 \capc_a(K_{j,q},Q_{j,q})
 &\le\capc_a(\overline B_{b_j}\times[-4s_j,-2s_j],B_{2b_j}\times(-5s_j,s_j))\\
 &\le C(ab_j^n+a^{p-1}s_jb_j^{n-p})\,.
 \end{aligned}
\end{equation}
If~$q$ is admissible, then~$qr_j(q)^n=q^{\lambda/p}s_j^{n/p}$ yields
\begin{equation}\label{eq:g-soda-irregular-large-radius}
 q^{\frac{\lambda}{p}}<\frac{C}{C_*}(ad_j^n+a^{p-1}d_j^{n-p})
 \le\frac{2C}{C_*}a^{p-1}d_j^{n-p}\,.
\end{equation}
Here we used~$d_j^p\le a^{p-2}$. To put the small-radius estimate~\eqref{eq:g-soda-irregular-small-radius} into the same form, use this inequality and~$\lambda=n(p-2)+p$ to obtain
\begin{equation*}
 \begin{aligned}
 q&\le C d_j^{\frac p{p-2}}\\
  &=C d_j^{\frac{p(n-p)}\lambda}
       (d_j^p)^{\frac{p(p-1)}{\lambda(p-2)}}\\
  &\le C a^{\frac{p(p-1)}\lambda}d_j^{\frac{p(n-p)}\lambda}\,.
 \end{aligned}
\end{equation*}
In either radius range, every admissible~$q$ is at most~$Ca^{p(p-1)/\lambda}d_j^{p(n-p)/\lambda}$. Taking the supremum proves~\eqref{eq:g-soda-irregular-decay}. Its decay exponent is positive, and the earlier terms are bounded by~$a$. The regularity classification follows from Theorem~\ref{thm:main}, in agreement with~\cite[Theorems~3.1 and~4.1]{AP26}.
\end{proof}

\begin{remark}[Exterior balls at a temporal pole]
For~$l=2$ and~$0<R<\min\{1,1/(2\theta)\}$, the exterior ball~$B_R^{n+1}(0,-R)$ touches~$\partial\Theta_{2,\theta}$ at the origin, since its upper boundary satisfies
\begin{equation*}
 t=-R+\sqrt{R^2-|x|^2}\le-|x|^2/(2R)\le-\theta|x|^2\,.
\end{equation*}
The contact is at the north pole. Nevertheless,~$2<p$ makes the point irregular, so the exclusion of this pole in~\cite[Proposition~4.1]{BBGP} is necessary; see~\cite{AP26}.
\end{remark}

\appendix
\section{Existence and stability for shifted obstacle problems}\label{sec:r-flux}
We prove Lemma~\ref{lem:r-obstacle}. For smooth~$a$, we first construct solutions, regularizing the flux when~$1<p\le2$; these approximating solutions have a bounded nonnegative reaction supported on the contact set. Comparison gives monotonicity in~$\psi$, and testing two variational inequalities gives stability in~$a$ with constants independent of~$\psi$ and its derivatives. This stability estimate permits strong approximation of~$a$ in~$L^p(Q;\R^n)$.

Throughout this appendix,~$n\ge1$,~$1<p<\infty$, and~$Q=D_0\times(0,T)$, where~$D_0\subset\R^n$ is a ball and~$T>0$. Norms without a displayed domain are over~$Q$, or over~$D_0$ at a fixed time. The constants~$C_p,c_p>0$ depend only on~$p$.

\subsection{Estimates for the regularized flux}
Assume~$1<p\le2$. Set~$A_0=A$ and, for~$\delta>0$, define
\begin{equation*}
 A_\delta(x)=(\delta^2+|x|^2)^{\frac{p-2}{2}}x,
 \qquad x\in\R^n.
\end{equation*}
For~$x,y\in\R^n$ and~$\delta\ge0$,
\begin{align}
 |A_\delta(x)-A_\delta(y)|&\le C_p|x-y|^{p-1},\label{eq:r-holder-flux}\\
 (A_\delta(x)-A_\delta(y))\cdot(x-y)
 &\ge c_p\frac{|x-y|^2}{(\delta+|x|+|y|)^{2-p}},\label{eq:r-weighted-monotonicity}\\
 \sup_{x\in\R^n}|A_\delta(x)-A(x)|&\le C_p\delta^{p-1}.\label{eq:r-uniform-regularization}
\end{align}
Zero divided by zero in \eqref{eq:r-weighted-monotonicity} is interpreted as zero; for $p=2$ its denominator is one.

Write~$\widetilde A(X)=|X|^{p-2}X$ for~$X\in\R^{n+1}$. Then~$A_\delta(x)$ is the projection of~$\widetilde A(x,\delta)$ onto its first~$n$ coordinates. Apply the vector inequalities in~\cite[Chapter~12, inequalities~(VII)--(VIII), p.~98]{Lindqvist17} to~$(x,\delta),(y,\delta)$ to obtain~\eqref{eq:r-holder-flux}--\eqref{eq:r-weighted-monotonicity}; homogeneity removes the additive constant in~(VII). Applying~(VIII) to~$(x,\delta),(x,0)$ gives~\eqref{eq:r-uniform-regularization}. The case~$p=2$ is immediate.

For~$a,\xi\in\R^n$, these estimates imply
\begin{align}
 |A_\delta(a+\xi)-A_\delta(a)|&\le C_p|\xi|^{p-1},\label{eq:r-singular-growth}\\
 \bigl(A_\delta(a+\xi)-A_\delta(a)\bigr)\cdot\xi
 &\ge c_p|\xi|^p-C_p(|a|^p+\delta^p).\label{eq:r-singular-energy}
\end{align}
For \eqref{eq:r-singular-energy}, expand the pairing with $x=a+\xi$, use
$A_\delta(x)\cdot x\ge c_p|x|^p-C_p\delta^p$, and absorb the terms $|x|^{p-1}|a|$ and $|a|^{p-1}|x|$ by Young's inequality.

\subsection{Construction and comparison}
\begin{proof}[Proof of Lemma~\ref{lem:r-obstacle}]
\emph{Step 1: Existence for smooth~$a$.}
Suppose that~$a$ is smooth on~$\overline Q$. For~$1<p\le2$, take~$0<\delta\le1$; for~$p>2$, take~$\delta=0$ and~$A_0=A$. Set~$V=W_0^{1,p}(D_0)\cap L^2(D_0)$, with the graph norm.

The flux is Carath\'eodory and strictly monotone. Its growth and coercivity follow from~\eqref{eq:r-singular-growth}--\eqref{eq:r-singular-energy} for~$1<p\le2$, and from~\eqref{eq:r-deg-monotonicity}--\eqref{eq:r-deg-growth} for~$p>2$. Smoothness of~$a$ and compact support of~$\psi$ give the initial and boundary compatibility conditions and
\begin{equation*}
 \psi_t-\div\bigl(A_\delta(a+\nabla\psi)-A_\delta(a)\bigr)-q\in L^\infty(Q).
\end{equation*}
The positive and negative parts of this residual belong to~$L^{p'}(0,T;V^*)$, verifying the order-dual hypothesis. Thus~\textup{(H$_1$)--(H$_4$)} of~\cite{GMTV} hold.

Theorem~2.2 of~\cite{GMTV} therefore gives
\begin{equation*}
 w_\delta\in L^p(0,T;V)\cap C([0,T];L^2(D_0)),
 \qquad (w_\delta)_t\in L^{p'}(0,T;V^*).
\end{equation*}
We have~$w_\delta(0)=0$ and~$w_\delta\ge\psi$. For every~$v\in L^p(0,T;V)$ with~$v\ge\psi$,
\begin{equation}\label{eq:r-obstacle-vi}
 \int_0^T\langle(w_\delta)_t,v-w_\delta\rangle\,dt
 +\iint_Q\bigl(A_\delta(a+\nabla w_\delta)-A_\delta(a)\bigr)\cdot \nabla (v-w_\delta)\,dx\,dt
 \ge\iint_Q q(v-w_\delta)\,dx\,dt.
\end{equation}
The same theorem gives
\begin{gather*}
 \mu_\delta:=(w_\delta)_t-\div\bigl(A_\delta(a+\nabla w_\delta)-A_\delta(a)\bigr)-q,\\
 0\le\mu_\delta
 \le\bigl(\psi_t-\div\bigl(A_\delta(a+\nabla \psi)-A_\delta(a)\bigr)-q\bigr)_+.
\end{gather*}
Taking~$v=\psi$ in~\eqref{eq:r-obstacle-vi} and using~$\mu_\delta\ge0$ and~$w_\delta\ge\psi$ gives~$\mu_\delta(w_\delta-\psi)=0$ almost everywhere. Hence
\begin{equation}\label{eq:r-regularized-contact}
 (w_\delta)_t-\div\bigl(A_\delta(a+\nabla w_\delta)-A_\delta(a)\bigr)=q+\mu_\delta,\qquad
 \mu_\delta\ge0,\quad \mu_\delta(w_\delta-\psi)=0.
\end{equation}

\smallskip
\emph{Step 2: Bounds and monotonicity in~$\psi$.}
Compare~$w_\delta$ with~$0$ and~$-t\|q\|_\infty$ by testing positive parts. This gives~$-t\|q\|_\infty\le w_\delta\le0$. These tests have zero lateral trace, and the time chain rule applies in~$V\subset L^2(D_0)\subset V^*$.

Testing the equation by~$w_\delta$, using~$\mu_\delta w_\delta\le0$, yields
\begin{equation}\label{eq:r-regularized-energy}
 \sup_t\|w_\delta(t)\|_2^2+\|\nabla w_\delta\|_p^p
 \le C_p\bigl(\|a\|_p^p+\delta^p|Q|+T\|q\|_\infty^2|Q|\bigr).
\end{equation}
For~$p>2$, strong monotonicity gives~\eqref{eq:r-regularized-energy} without the term~$\|a\|_p^p$.

For ordered obstacles~$\psi_1\le\psi_2$, use~$\min\{w_1,w_2\}$ and~$\max\{w_1,w_2\}$ as competitors in the respective variational inequalities. Adding the inequalities and applying the chain rule in time gives~$w_1\le w_2$.

\smallskip
\emph{Step 3: Stability with respect to~$a$.}
For smooth~$a,b$, let~$w_1,w_2$ solve~\eqref{eq:r-obstacle-vi} with background fields~$a,b$ and flux regularization parameters~$\delta,\sigma$, respectively, and with a common right-hand side~$q$ and obstacle~$\psi$. We derive an estimate independent of~$\psi$ and its derivatives, as needed when applying Lemma~\ref{lem:r-two-limits}. Put~$d=w_1-w_2$ and
\begin{equation*}
 \mathcal E_{\delta,\sigma}
 :=A_\delta(a+\nabla w_2)-A_\delta(a)-A_\sigma(b+\nabla w_2)+A_\sigma(b).
\end{equation*}
To compare the solutions on~$(0,t)$, take~$v_1=w_2$ before~$t$ and~$v_1=w_1$ afterwards in the inequality for~$w_1$. Interchange~$w_1,w_2$ to define~$v_2$ for the other inequality. Both~$v_1,v_2$ belong to~$L^p(0,T;V)$ and satisfy~$v_1,v_2\ge\psi$.

Add the inequalities and apply the time chain rule. Both initial values are zero, so splitting the flux difference gives
\begin{equation}\label{eq:r-common-before-conversion}
 \frac12\|d(t)\|_2^2
 +\iint_{D_0\times(0,t)}\bigl(A_\delta(a+\nabla w_1)-A_\delta(a+\nabla w_2)\bigr)\cdot\nabla d\,dx\,ds
 \le\iint_{D_0\times(0,t)}
 |\mathcal E_{\delta,\sigma}|\,|\nabla d|\,dx\,ds.
\end{equation}
For~$1<p\le2$ and~$\|a\|_p+\|b\|_p\le B$, set
\begin{equation*}
 S:=\delta+|a+\nabla w_1|+|a+\nabla w_2|.
\end{equation*}
The energy bound controls~$\|S\|_p$ independently of~$\psi$. Estimates~\eqref{eq:r-weighted-monotonicity} and~\eqref{eq:r-weighted-holder} give
\begin{equation*}
 \|\nabla d\|_p^2
 \le C\iint_Q\frac{|\nabla d|^2}{S^{2-p}}\,dx\,dt
 \le C\|\mathcal E_{\delta,\sigma}\|_{p'}\|\nabla d\|_p,
\end{equation*}
where~$C=C(p,Q,\|q\|_\infty,B)$. For~$p>2$,~$\delta=\sigma=0$, and~\eqref{eq:r-deg-monotonicity} gives instead
\begin{equation*}
 \|\nabla d\|_p^p
 \le C_p\|\mathcal E_{\delta,\sigma}\|_{p'}\|\nabla d\|_p.
\end{equation*}
In either case, Young's inequality and~\eqref{eq:r-common-before-conversion} give
\begin{align}
 \sup_t\|d(t)\|_2^2+\|\nabla d\|_p^2
 &\le C\|\mathcal E_{\delta,\sigma}\|_{p'}^2,
 &&1<p\le2,\label{eq:r-regularized-comparison}\\
 \sup_t\|d(t)\|_2^2+\|\nabla d\|_p^p
 &\le C_p\|\mathcal E_{\delta,\sigma}\|_{p'}^{p'},
 &&p>2.\notag
\end{align}
The second bound is~\eqref{eq:r-common-degenerate} for smooth~$a,b$.

\smallskip
\emph{Step 4: Approximation of~$a$ in~$L^p$.}
Fix smooth~$a$. For~$1<p\le2$,~\eqref{eq:r-uniform-regularization} bounds the flux error in~\eqref{eq:r-regularized-comparison} by~$C_p(\delta^{p-1}+\sigma^{p-1})$ pointwise. Thus~$w_\delta$ converges in~$C([0,T];L^2(D_0))$ and~$\nabla w_\delta$ converges in~$L^p(Q)$ as~$\delta\to0$. Set~$W_q[a,\psi]=\lim_{\delta\to0}w_\delta$. The fluxes converge strongly, so we can pass to the limit in the differential inequality and the energy estimate. Letting~$\delta=\sigma\to0$ also gives~\eqref{eq:r-common-singular} for smooth~$a,b$. For~$p>2$, set~$W_q[a,\psi]=w_0$.

For general~$a\in L^p(Q;\R^n)$, choose smooth~$a_\ell\to a$ in~$L^p(Q)$ and put~$w_\ell=W_q[a_\ell,\psi]$. Estimates~\eqref{eq:r-singular-background} and~\eqref{eq:r-deg-background} give
\begin{equation*}
 \begin{aligned}
 &\|A(a_\ell+\nabla w_k)-A(a_\ell)-A(a_k+\nabla w_k)+A(a_k)\|_{p'}\\
 &\quad\le
 \begin{cases}
 C_p\|a_\ell-a_k\|_p^{p-1},&1<p\le2,\\
 C_p\|a_\ell-a_k\|_p\bigl(\|a_\ell\|_p+\|a_k\|_p+\|\nabla w_k\|_p\bigr)^{p-2},&p>2.
 \end{cases}
 \end{aligned}
\end{equation*}
Combining this flux error bound with the stability estimates~\eqref{eq:r-common-singular}--\eqref{eq:r-common-degenerate} shows that~$(w_\ell)$ is Cauchy in~$C([0,T];L^2(D_0))$ and~$(\nabla w_\ell)$ is Cauchy in~$L^p(Q)$. Set~$W_q[a,\psi]=\lim_{\ell\to\infty}w_\ell$. Interleaving two choices of~$(a_\ell)$ shows that the limit is independent of the approximation. These convergences preserve the bounds, energy estimate, initial and lateral values, and monotonicity in~$\psi$. The fluxes converge strongly, so the differential inequality also passes to the limit. Approximating~$a,b$ simultaneously gives both comparison estimates for general~$a,b$.

Finally, on any open set where~$\psi=-M$ and~$M>T\|q\|_\infty$, every solution constructed for smooth~$a$, with flux regularization when~$p\le2$, satisfies
\begin{equation*}
 w_\delta-\psi\ge M-T\|q\|_\infty>0.
\end{equation*}
Thus~\eqref{eq:r-regularized-contact} gives~$\mu_\delta=0$ there. The fluxes converge strongly in both limits, so~$W_q[a,\psi]$ satisfies the equation on this open set.
\end{proof}

\section{Interpolation and gradient estimates}\label{app:upper-auxiliary}
We prove the two analytic estimates used in Section~\ref{sec:upper}. Lemma~\ref{lem:interpolation} supplies spatial interpolation at and above~$p=n$; Lemma~\ref{lem:forced-difference} bounds the gradient of a difference with zero initial and lateral trace. We use~$\lambda=n(p-2)+p$ and the constant convention from Section~\ref{sec:preliminaries}.

\subsection{Zero-trace interpolation}\label{app:interpolation-proof}
\begin{proof}[Proof of Lemma~\ref{lem:interpolation}]
This is the Gagliardo--Nirenberg inequality; its time-integrated form is recorded in~\cite[equation~(16), pp.~1196--1197]{KLPSh}. We give a short proof to track uniformity in~$\gamma$. Put~$s=p(1+\gamma/n)$ and~$\alpha=(n-1)s/n$. Since~$p\ge n\ge2$, we have~$\alpha>1$,~$\alpha-s/p\ge0$, and~$\alpha-1=(\alpha-s/p)+\gamma/n$. The~$W^{1,1}$ Sobolev inequality applied to~$|z|^\alpha$, followed by H\"older's inequality with exponents~$p$,~$s/(\alpha-s/p)$ and~$n$, gives
\begin{equation*}
 \|z\|_s^\alpha
 \le C_n\alpha\int |z|^{\alpha-1}|\nabla z|
 \le C_n\alpha\|\nabla z\|_p
       \|z\|_s^{\alpha-s/p}\|z\|_\gamma^{\gamma/n}.
\end{equation*}
When~$n=p=2$, the middle H\"older factor is omitted. Cancel~$\|z\|_s^{\alpha-s/p}$ for~$z\ne0$ and raise to the power~$p$ to obtain~\eqref{eq:zero-trace-interpolation}. The bound~$\alpha\le (n-1)p(1+p/n)/n$ makes the constant uniform for~$1\le\gamma\le p$. Approximation and zero extension give the asserted zero-trace version with a constant independent of the spatial domain.
\end{proof}

\subsection{Gradient estimate for a difference}\label{app:forced-difference-proof}
Throughout this subsection,~$p>2$.

\begin{proof}[Proof of Lemma~\ref{lem:forced-difference}]
We use the difference-testing argument of~\cite[Lemma~4.1, equations~(4.5) and~(4.9)]{KuMi14}, retaining the level estimates needed for the domain-independent bound~\eqref{eq:forced-difference-flux}. We estimate the gradient separately on the level bands of~$w$. Testing with approximations of the positive sign function gives~$\sup_t\int w\le M$. To control the energy below a given level~$k$, test with~$T_kw=\min\{w,k\}$, retaining the nonnegative time primitive at the final time. By~\eqref{eq:monotonicity},
\begin{equation}\label{eq:truncation-energy}
 \iint_{\{0<w<k\}}|\nabla w|^p\,dx\,dt\le CkM\,.
\end{equation}
For $z=T_kw$, Lemma~\ref{lem:interpolation} with $\gamma=1$ when $p\ge n$, and Sobolev interpolation when $p<n$, give
\begin{equation*}
 \int_D z^{p+\frac{p}{n}}\le
 C\left(\int_D z\right)^{\frac{p}{n}}\int_D|\nabla z|^p\,.
\end{equation*}
Integrating in time, using~$\sup\int z\le M$, and applying~\eqref{eq:truncation-energy}, we obtain
\begin{equation*}
 \iint z^{p+\frac{p}{n}}\,dx\,dt\le CM^{1+\frac{p}{n}}k\,.
\end{equation*}
On~$\{w>k\}$ we have~$z=k$. The bound for~$\iint z^{p+p/n}\,dx\,dt$ therefore bounds the measure of~$\{w>k\}$; the~$L^1$ bound gives a second estimate, useful at low levels. Together they give
\begin{equation}\label{eq:distribution-w}
 |\{w>k\}|\le
 \min\left\{\frac{TM}{k},\,
       CM^{1+\frac{p}{n}}k^{-(p-1+\frac{p}{n})}\right\}\,.
\end{equation}
The measure here is space-time Lebesgue measure on~$D\times I$.

On~$\{k<w\le2k\}$, H\"older's inequality and~\eqref{eq:truncation-energy} at level~$2k$ give
\begin{equation}\label{eq:level-flux}
 \iint_{\{k<w\le2k\}}|\nabla w|^{p-1}\,dx\,dt \le C(kM)^{\frac{p-1}{p}}|\{w>k\}|^{\frac{1}{p}} \le C\min\left\{ MT^{\frac{1}{p}}k^{\frac{p-2}{p}},\, M^{1+\frac{1}{n}}k^{-\frac{1}{n}}\right\}\,.
\end{equation}
If~$M=0$, comparison gives~$w=0$. Otherwise split the dyadic levels at
\begin{equation*}
 k_0=M^{\frac{p}{\lambda}}T^{-\frac{n}{\lambda}}\,.
\end{equation*}
At this level the two bounds in~\eqref{eq:level-flux} have the same scale. For levels below~$k_0$, sum the first term; for levels above it, sum the second. Both are convergent geometric series, since~$p>2$. Their values are bounded by~$CT^{\frac{1}{\lambda}}M^{1+\frac{p-2}{\lambda}}$. Since~$\nabla w=0$ almost everywhere on~$\{w=0\}$, summing over the bands accounts for the whole integral. This proves~\eqref{eq:forced-difference-flux}.
\end{proof}

\section*{Acknowledgements}
MP acknowledges funding from the Research Council of Finland (project~360185). The proof development and the preparation of this paper were assisted by ChatGPT. The authors take full responsibility for the contents of the paper.

\end{document}